\documentclass{article}
\usepackage[utf8]{inputenc}
\usepackage{amsmath}
\usepackage{amsfonts}
\usepackage{amssymb}
\usepackage{amsthm}
\usepackage{enumerate}
\usepackage{appendix}
\usepackage{xcolor}
\usepackage{comment}
\usepackage{hyperref}
\usepackage{geometry}

\newcommand{\rmk}[1]{{\color{red} #1}}
\numberwithin{equation}{section}

\newtheorem{thm}{Theorem}[section]
\newtheorem{cor}[thm]{Corollary}
\newtheorem{lem}[thm]{Lemma}
\newtheorem{prop}[thm]{Proposition}

\theoremstyle{definition}

\newtheorem{defn}[thm]{Definition}

\theoremstyle{remark}
\newtheorem{remark}[thm]{Remark}

\title{Blow-up of the one-dimensional wave equation with quadratic spatial derivative nonlinearity}

\author{Tej-eddine Ghoul\thanks{New York University Abu Dhabi; Email: teg6@nyu.edu} \and  Jie Liu\thanks{New York University Abu Dhabi; Email: jl15817@nyu.edu} \and Nader Masmoudi \thanks{New York University Abu Dhabi; Email: nm30@nyu.edu} }

\date{}
\begin{document}
	\maketitle

	\begin{abstract}
		We investigate the blow-up dynamics of smooth solutions to the one-dimensional wave equation with a quadratic spatial derivative nonlinearity, motivated by its applications in Effective Field Theory (EFT) in cosmology. Despite its relevance, explicit blow-up solutions for this equation have not been documented in the literature. In this work, we establish the non-existence of smooth, exact self-similar blow-up solutions and construct a five-parameter family of generalized self-similar solutions exhibiting logarithmic growth. Moreover, we prove the asymptotic stability of these blow-up solutions.
		
		Our proof tackles several significant challenges, including the non-self-adjoint nature of the linearized operator, the presence of unstable eigenvalues, and, most notably, the treatment of non-compact perturbations. By substantially advancing Donninger’s spectral-theoretic framework, we develop a robust methodology that effectively handles non-compact perturbations. Key innovations include the incorporation of the Lorentz transformation in self-similar variables, an adaptation of the functional framework introduced in \cite{MRR22}, and a novel resolvent estimate. This approach is general and robust, allowing for straightforward extensions to higher dimensions and applications to a wide range of nonlinear equations.

	\end{abstract}
	
	\section{Introduction}
	
	We consider the following one-dimensional wave equation with quadratic spatial derivative nonlinearity
	\begin{align}\label{eq:nlw}
		u_{tt} - u_{xx} = (u_x)^2,
	\end{align}
	where $u:\mathbb{R}^{1+1}\to \mathbb{R}$ is the unknown function. This equation arises from Effective Field Theory (EFT) applications in cosmology, particularly in studying the Universe’s late-time accelerated expansion and dark energy behavior \cite{eckmann2023instabilities}. In this context, nonlinear terms like $(u_x)^2$ are often used to capture nonlinear evolution in cosmic structures \cite{hassani2022new}. Numerical simulations reveal that such nonlinearities can lead to finite-time blow-up, where the solution diverges in finite time. Without a Laplacian term, this blow-up typically takes the form of a “V-type” singularity, characterized by rapid local curvature increase \cite{eckmann2023instabilities,hassani2022new}. However, when a Laplacian term is added, the divergence shifts to an “M-type” blow-up, where steepened edges form caustic-like structures as they propagate outward \cite{eckmann2023instabilities}. This study of singularity types in different conditions provides essential insights into the limitations of EFT models under extreme cosmic conditions and supports theoretical exploration of high-energy scale evolution in cosmic structures.
	
	Previous studies of Eq. \eqref{eq:nlw} have primarily focused on the lifespan of local solutions and the conditions under which blow-up solutions exist. For investigations in higher-dimensional settings, one may refer to \cite{rammaha1995upper,rammaha1997note,shao2024blow}, while \cite{eckmann2023instabilities,sasaki2023lifespan} addresses the one-dimensional case relevant to our study.  However, to the best of our knowledge, there are no explicit blow-up solutions or comprehensive qualitative results on the blow-up dynamics of Eq. \eqref{eq:nlw} available in the existing literature. Our aim of this paper is to fill this gap.
	%(life-span result;) LWP theory; existence of blow-up by F.John  etc.; 
	\subsection{Main results}
	
	%To the best of the authors’ knowledge, no explicit blow-up solutions to \eqref{eq:nlw} have been identified in the existing literature. 
	Unlike the wave equation with power nonlinearities \cite{donninger2016blowup,merle2007existence}, Eq. \eqref{eq:nlw} does not admit ODE blow-up solutions. By scaling, the self-similar ansatz for the nonlinear wave equation \eqref{eq:nlw} takes the form
	\begin{align*}
		u(t,x) = U\left(y\right), \quad \ y = \frac{x-x_0}{T-t}, \quad T>0,\ x_0\in \mathbb{R}.
	\end{align*}
	Under this ansatz, exact self-similar solutions satisfy the equation
	\begin{align*}
		2 y\partial_y U + (y^2-1) \partial_{yy}U = (\partial_y U)^2.
	\end{align*}
	In the following, we establish that there are no smooth, exact self-similar blow-up solutions to \eqref{eq:nlw}. Furthermore, no smooth solitary waves exist for \eqref{eq:nlw}. Specifically, all stationary solutions to \eqref{eq:nlw} are of the form
	$$u(x) = -\log(c_1 + x) + c_2, \quad c_1, c_2 \in \mathbb{R},$$
	which becomes singular for  $x \leq -c_1$. This suggests that type-II blow-up solutions may also not exist.
	
	Nonetheless, we construct a new five-parameter family of smooth ``generalized" self-similar blow-up solutions to \eqref{eq:nlw}. These solutions exhibit logarithmic growth and are expressed as
	\begin{align*}
		u(t,x) = U\left(s,y\right) = \alpha s + \tilde{U}(y), \quad s=-\log(T-t)+\log(T), \ y = \frac{x-x_0}{T-t}, \ \alpha \in \mathbb{R}.
	\end{align*}
	The first main result of this paper is as follows.
	
	\begin{thm}[Existence of smooth generalized self-similar blow-up solutions] \label{thm:existence} \quad 
		\begin{enumerate}[1)] 
			\item For any $T>0$ and $x_0\in \mathbb{R}$, there are no smooth exact self-similar blow-up solutions to \eqref{eq:nlw} in the backward lightcone $$\Gamma(T,x_0) :=\left\{(t, x) \in\left[0, T\right) \times \mathbb{R}:\left|x-x_0\right| \le T-t\right\}.$$
			\item There exists a five-parameter family of smooth generalized self-similar blow-up solutions to \eqref{eq:nlw} with logarithmic growth
			%\begin{align*}
				%u_{\alpha,\beta,\kappa,T,x_0} = -\alpha\log\left(1-\frac{t}{T}\right) -\alpha\log\left(\sqrt{1+\alpha}+\frac{x-x_0}{T-t}\right) + w_{\alpha,\beta,\kappa,T,x_0}(t,x) 
			%\end{align*}
			\begin{align*}
				u_{\alpha,\beta,\kappa,T,x_0} = -\alpha\log\left(1-\frac{t}{T}\right) + \tilde{u}_{\alpha,\beta,\kappa,T,x_0}(t,x) 
			\end{align*}
			well defined in the backward lightcone $\Gamma(T,x_0)$ for all $T>0, \kappa, x_0\in \mathbb{R},$ and either $ 0< \beta< \infty, \sqrt{1+\alpha}\in \mathbb{N}, \alpha\neq 0$ or $\beta=0, \infty, \alpha>0.$
			%Here, $w_{\alpha,\beta,\kappa,T,x_0}$ is given by
		%\begin{align*}
			%	&w_{\alpha,\beta,\kappa,T,x_0}(t,x)\\ 
				%& = \kappa + \int_0^{\frac{x-x_0}{T-t}} \frac{2\alpha\sqrt{1+\alpha}}{(\sqrt{1+\alpha}-z)(\sqrt{1+\alpha}+z)} \left( \beta\left(\frac{1-z}{1+z}\right)^{\sqrt{1+\alpha}}  \frac{2\alpha\sqrt{1+\alpha}(\sqrt{1+\alpha}+z)}{\sqrt{1+\alpha}-z} + 1 \right)^{-1} dz.
			%\end{align*}
			Here, $\tilde{u}_{\alpha,\beta,\kappa,T,x_0}$ is given by
			\begin{align*}
				&\tilde{u}_{\alpha,\beta,\kappa,T,x_0}(t,x) \\
				=&  \kappa -\alpha\log \sqrt{1+\alpha}+\int_0^{\frac{x-x_0}{T-t}} \frac{\alpha(1+z)^{\sqrt{1+\alpha}} - 2 \alpha^2\sqrt{1+\alpha}\beta(1-z)^{\sqrt{1+\alpha}}}{(1+z)^{\sqrt{1+\alpha}}(\sqrt{1+\alpha} -z) + 2 \alpha\sqrt{1+\alpha}\beta(1-z)^{\sqrt{1+\alpha}}(\sqrt{1+\alpha}+z)} dz.
			\end{align*}		
			In particular, we have
			\begin{align*}
				u_{\alpha,\infty,\kappa,T,x_0} =-\alpha\log\left(1-\frac{t}{T}\right) -\alpha\log\left(\sqrt{1+\alpha}+\frac{x-x_0}{T-t}\right) + \kappa,\\
				u_{\alpha,0,\kappa,T,x_0} =-\alpha\log\left(1-\frac{t}{T}\right) -\alpha\log\left(\sqrt{1+\alpha}-\frac{x-x_0}{T-t}\right) + \kappa.
			\end{align*}
		\end{enumerate}   
	\end{thm}
	
	\begin{remark}[Other choices of parameters]
		For $\alpha< 0$, $u_{\alpha,\beta,\kappa,T,x_0}$ also solves Eq. \eqref{eq:nlw}. However, there is always a singular point inside the light cone, regardless the choice of $\beta$. Thus, these solutions does not influence the evolution of smooth solutions to Eq. \eqref{eq:nlw} due to the propagation of regularity of wave equations. And we require $\beta\ge 0$ for the same reason. Moreover, for $0<\beta<\infty$ and $\sqrt{1+\alpha}\notin \mathbb{N}$, $u_{\alpha,\beta,\kappa,T,x_0}$ does not belong to $C^k$ for $k> \sqrt{1+\alpha}$.
	\end{remark}
	\begin{remark}[Global-in-space smooth blow-up solutions]
		By the finite propagation speed of wave equations and a cut-off of $u_{\alpha,\beta,\kappa,T,x_0}$, we can also construct global in space smooth blow-up solutions to Eq. \eqref{eq:nlw}. 
	\end{remark}
	\begin{remark}[Free parameters and symmetries]
		Parameters $\kappa, T, x_0$ corresponds to different symmetries of the equation \eqref{eq:nlw}. i.e. translation invariance in profile $u$, time $t$, and space $x$.
	\end{remark}
	\begin{remark}[Interpretation as exact self-similar blow-up solutions]
		The generalized self-similar blow-up solutions above can be regarded as \textit{exact} self-similar blow-up solutions to a transformed version of \eqref{eq:nlw} obtained by taking the spatial derivative, see Section \ref{subsec:interpretation}. This perspective provides a more natural interpretation of the generalized self-similar blow-up solutions and elucidates the origin of the logarithmic growth. For the same reason, we believe that these generalized self-similar blow-up solutions give the generic blow-up dynamics.
	\end{remark}
	\begin{remark}[Comparison with numerical simulations]
		The logarithmic growth exhibited by the profile  $u_{\alpha,\beta,\kappa,T,x_0}$  has not been observed in prior numerical simulations \cite{eckmann2023instabilities}. While numerical methods are invaluable tools, they may sometimes overlook subtle aspects of blow-up solutions. In particular, numerical observations of M-type blow-up typically capture the blow-up behavior at the first blow-up time  $T$. Building on this, we conjecture that such behavior could correspond to a linear combination of finitely many generalized blow-up solutions, with the parameter $ \alpha$  approaching  0  as $ t \to T$. One can refer to \cite{merle2007existence,merle2012existence,merle2012isolatedness} for a similar phenomenon observed in one-dimensional semilinear nonlinear wave equations with power nonlinearities.
	\end{remark}
	
	Let $k_{\alpha} = \lceil \sqrt{1+\alpha}\rceil$ denote the smallest integer that is greater than or equal to $\sqrt{1+\alpha}$. The second main result of this paper establishes the asymptotic stability of the generalized self-similar blow-up solutions $u_{\alpha,\infty,\kappa,T,x_0}$.  
	\begin{thm}\label{thm:main}
		Let $\alpha_0,T_0>0, \kappa_0, x_0\in \mathbb{R}$, $k\ge k_{\alpha_0}+3$. For all $0<\delta<1$, there exists a constant $\epsilon>0$ such that for any real-valued functions $(f, g)\in H^{k+1}(\mathbb{R})\times H^k(\mathbb{R})$ with 
		$$\|(f, g)\|_{H^{k+1}(\mathbb{R})\times H^k(\mathbb{R})} \le \epsilon $$
		there exist parameters $\alpha^*, T^*>0, \kappa^*\in \mathbb{R}$, and a unique solution $u: \Gamma(T^*, x_0)\mapsto \mathbb{R}$ of Eq. \eqref{eq:nlw} with initial data
		$$ u(0,x) = u_{\alpha_0,\infty,\kappa_0,T_0,x_0}(0,x) + f(x), \quad \partial_t u(0,x) = \partial_t u_{\alpha_0,\infty,\kappa_0,T_0,x_0}(0,x) + g(x), \quad x\in B_{T^*}(x_0), $$
		such that the bounds
		\begin{align*}
			&(T^*-t)^{-\frac{1}{2} + s} \| u(t,\cdot) - u_{\alpha^*, \infty, \kappa^*, T^*, x_0}\|_{\dot{H}^{s}(B_{T^*-t}(x_0))} \lesssim (T^*-t)^{1-\delta} ,
		\end{align*}
		for $s=0,1,\ldots,k+1$, and 
		\begin{align*}    
			(T^*-t)^{-\frac{1}{2} + s} \| \partial_t u(t,\cdot) - \partial_t u_{\alpha^*, \infty, \kappa^*, T^*, x_0}\|_{ \dot{H}^{s-1}(B_{T^*-t}(x_0))} \lesssim (T^*-t)^{1-\delta} ,
		\end{align*}
		for $s=1,\ldots,k+1$, hold for all $0\le t < T^*.$ Moreover, 
		\begin{align*}
			|\alpha^*-\alpha_0| + |\kappa^* -\kappa_0| + \left|\frac{T^*}{T_0}-1\right| \lesssim \epsilon.
		\end{align*}
	\end{thm}
	\begin{remark}
		The same result also holds for $u_{\alpha,0,\kappa,T,x_0}$. Indeed, we have 
		$$u_{\alpha,0,\kappa,T,x_0}(t,x) = u_{\alpha,\infty,\kappa,T,-x_0}(t,-x),$$ 
		and any solution $u(t,x)$ to Eq. \eqref{eq:nlw} implies that $u(t,-x)$ also solves Eq. \eqref{eq:nlw}. 
	\end{remark}
	\begin{remark}
		For $0<\beta<\infty$ and $\sqrt{1+\alpha}\in \mathbb{N}, \alpha>0$, we conjecture that $u_{\alpha,\beta,\kappa,T,x_0}$ are only co-dimensionally stable. For instance, when $\alpha=3, \beta = \frac{1}{12}$, the following explicit solution 
		$$ u_{3,\frac{1}{12},\kappa,T,x_0} = -3\log\left(1-\frac{t}{T}\right) + \frac{3}{2}\left(\frac{x-x_0}{T-t}\right)^2 +\kappa $$
		solves Eq. \eqref{eq:nlw}. However, analysis of the linearized equation of \eqref{eq:nlw} around $u_{3,\frac{1}{12},\kappa,T,x_0}$ reveals the presence of two additional unstable eigenvalues, i.e. eigenvalues with positive real part, which not associated with any symmetries of the equation. Consequently, we expect that,  for $0<\beta<\infty$, perturbations around $u_{\alpha,\beta,\kappa,T,x_0}$ would always converge to either $u_{\alpha,0,\kappa,T,x_0}$ or $u_{\alpha,\infty,\kappa,T,x_0}$. In other words, $u_{\alpha,0,\kappa,T,x_0}$ and $u_{\alpha,\infty,\kappa,T,x_0}$ serve as two stable attractors for the dynamics.
	\end{remark}

	\subsection{Related work} Over the past two decades, significant progress has been made in the blow-up theory of nonlinear wave equations, particularly for semilinear wave equations with power-type nonlinearities.
	
	For one-dimensional nonlinear wave equations with power nonlinearities, the blow-up theory has been comprehensively developed through a series of works by Merle and Zaag \cite{merle2003determination,merle2007existence,merle2008openness,merle2012existence,merle2012isolatedness}, as well as Côte and Zaag \cite{cote2013construction}. In higher dimensions with sub-conformal nonlinearities, the blow-up rate has been established \cite{merle2003determination,merle2005determination}, and the stability of the one-dimensional analogous self-similar solutions has been proven by Merle and Zaag \cite{merle2015stability,merle2016dynamics}. A different approach to stability is also presented in \cite{donninger2012stable}. 
	
	However, in the super-conformal case, the understanding remains limited. In a series of studies \cite{donninger2014stable,donninger2016blowup,donninger2017stable,donninger2017strichartz,donninger2020blowup,wallauch2023strichartz,ostermann2024stable}, Donninger and collaborators applied spectral-theoretic methods to establish the stability of ODE blow-up solutions, i.e., solutions of the associated ODE  $u^{\prime \prime}=u^p$. This method has also been extended to demonstrate the stability of other self-similar blow-up solutions for supercritical nonlinear wave equations \cite{glogic2020threshold,glogic2021co,csobo2024blowup,chen2024co}. Additionally, noteworthy contributions include the work of Dai and Duyckaerts \cite{dai2021self}, Duyckaerts and Negro \cite{duyckaerts2024global}, and Kim \cite{kim2022self}, which focus on the construction of self-similar solutions. In \cite{kim2022self}, Kim also established the co-dimensional stability of these self-similar profiles using the functional framework proposed in \cite{MRR22}.
	 
	For self-similar blow-up in other wave-type equations, relevant results can be found in studies of wave maps \cite{donninger2011stable,costin2016proof,donninger2023optimal}, the Yang-Mills equation \cite{costin2016stability}, and the Skyrme model \cite{chen2023singularity,mcnulty2024singularity}.
	
	For energy-critical and supercritical nonlinear wave equations, another type of blow-up solution, known as Type-II blow-up, arises. Comprehensive surveys on Type-II blow-up can be found in \cite{krieger2009slow,krieger2014full,krieger2020stability,hillairet2012smooth,collot2018type,duyckaerts2012universality,ghoul2018construction} and the references therein.
	
	%1d power nonlinearity results; 1. superconformal nlw, ode blow-up; --- compact perturbation+semigroup method+ smooth profile ;superconformal other self-similar blow-up; stability of blow-up---sigular profile ---Heun type ode 
	%2. wave-map; yang-mills; 
	%3. high-dimensional energy-critical, energy-supercritical type-II blow-up;
	
	%4. defocusing energy-supercritical --- semigroup method; Kim other self-similar blow-up codimenisonal stability; 

	\subsection{Outline of the proof}
	%difficulties and new ideas;structure of the paper;
	We outline the key difficulties and new ideas involved in the proof of Theorem \ref{thm:main}. The stability analysis of the generalized self-similar blow-up solutions presents three primary challenges that significantly complicate the argument.
	\begin{enumerate}
		\item \textbf{Non-existence of dissipative energy functionals} $\Longrightarrow$ \textit{direct energy methods fail}. In their seminal works \cite{merle2007existence,merle2015stability,merle2016dynamics} on the stability of self-similar solutions to sub-conformal nonlinear wave equations, Merle and Zaag rely heavily on the existence of a Lyapunov energy functional that satisfies the following dissipative property:
		\begin{align*}
			\frac{d E[\eta]}{ds} = -\frac{4}{p-1} \int_{-1}^1 (\partial_s \eta)^2 \frac{\rho}{1-y^2} dy \le0,
		\end{align*}
		where $E[\eta]$ is an appropriately chosen Lyapunov functional, $p>1$, and $\rho\ge 0$ is a weight function. However, in our setting, such a Lyapunov structure does not exist due to the non-Hamiltonian nature of the nonlinearity. Specifically, the natural energy estimate for the linearized equation \eqref{linear-ss}  around $U_{\alpha, \infty,\kappa}$ yields 
		\begin{align*}
			\frac{d E[\eta]}{ds} = \int_{-1}^1 \frac{2\alpha y}{\sqrt{1+\alpha} + y }(\partial_s \eta)^2 \frac{\rho}{1-y^2} dy,
		\end{align*}
		where the term  $2\alpha y$ has indefinite signs on the interval $[-1, 1]$, thus preventing the energy from dissipating. This fundamental difference renders direct energy methods ineffective in our analysis.
		
		\item \textbf{Non-self-adjoint structure of the linearized operator with unstable eigenvalues $0$ and $1$} $\Longrightarrow$ \textit{traditional spectral methods fail}. The stability analysis of self-similar blow-up solutions often involves non-self-adjoint operators with unstable eigenvalues, which arise naturally from the symmetries of the equation. This poses significant challenges, as discussed in \cite{merle2007existence,donninger2011stable,donninger2024spectral}. In our case, the linearized operator has unstable eigenvalues 0 and 1, making traditional spectral methods ineffective. This limitation arises because the spectral decomposition for non-self-adjoint operators is inherently more complex and lacks the orthogonality properties of self-adjoint operators, preventing the clear separation of stable and unstable modes needed for decay estimates. Furthermore, this difficulty explains the lack of dissipative energy functionals, as noted earlier. The unstable components of the energy cannot be effectively isolated, which prevents energy functionals from exhibiting dissipative properties.

		%Furthermore, this non-self-adjoint structure and the presence of unstable eigenvalues also account for the lack of dissipative energy functionals, as noted earlier. The inability to separate unstable components from stable ones, even with orthogonality conditions or weighted energy estimates, makes it difficult for energy functionals to exhibit dissipative properties, adding to the challenges in the analysis.

		\item \textbf{Non-compact perturbations} $\Longrightarrow$ \textit{Losing information about the essential spectrum}. Compared to nonlinear wave equations with power nonlinearities, an additional difficulty in our setting arises from the non-compact perturbations introduced by the nonlinear term. Specifically, the linearized operator can be expressed as		
		\begin{align*}
			{\mathbf{L}}_{\alpha}=\mathbf{L} + \mathbf{L}_{\alpha,1}, \quad \mathbf{L}_{\alpha,1}\begin{pmatrix} q_1\\ q_2 \end{pmatrix}= \begin{pmatrix} 0\\ 2 \partial_y U_{\alpha,\infty,\kappa} \partial_y q_1 \end{pmatrix},
		\end{align*}
		where 
		\begin{align*}
			\mathbf{L} \mathbf{q} := \begin{pmatrix}
				-y\partial_y q_1+ q_2 \\
				\partial_{yy} q_1 -q_2-y\partial_y q_2
			\end{pmatrix}
		\end{align*}
		 is the free wave operator and $\mathbf{L}_{\alpha,1}$ is the linearized term arising from the nonlinearity, see \eqref{linear-ss-system}. Notably, in our setting, $\mathbf{L}_{\alpha,1}$ constitutes only a bounded perturbation in $H^{k+1} \times H^k$, yet it lacks compactness relative to $\mathbf{L}$ due to the presence of the derivative term and the degeneracy of $\mathbf{L}$ on the light cone $y = \pm 1$. Specifically, if $\mathbf{q}$ and $\mathbf{L} \mathbf{q}$ are bounded in $H^{k+1} \times H^k$, the definition only allows one to deduce $(y^2 - 1)\partial_{yy} q_1 \in H^k$, rather than $\partial_{yy} q_1 \in H^k$. This implies that no additional regularity for $\partial_y q_1$ can be gained at $y = \pm 1$. Consequently, $\mathbf{L}_{\alpha,1}$ is not relatively compact with respect to $\mathbf{L}$. This degeneracy is also evident from the scalar form of the linearized equation \eqref{linear-ss}, which highlights the singular behavior of the free wave operator at $y = \pm 1$. Additionally, unlike the Skyrme model studied in \cite{chen2023singularity}, our problem lacks structural properties that could eliminate the derivative term. 
		 
		 This lack of compactness results in a complete loss of information about the essential spectrum, as the essential spectrum of the linearized operator is no longer preserved under non-compact perturbations \cite{kato1995}, which in turn leads to the breakdown of Donninger’s standard spectral-theoretic framework. Even if spectral information were accessible, a significant challenge would remain in translating it into growth bounds for the corresponding semigroup. As highlighted in \cite{chen2023singularity}, none of the soft arguments commonly used in previous studies can be applied to the non-compact case.
	\end{enumerate}
	
	To address these challenges, we advance Donninger’s spectral-theoretic framework for non-self-adjoint operators, originally developed by Donninger for wave maps \cite{donninger2011stable,donninger2012wavemap}. For an overview of this framework, we refer to the introductory note \cite{donninger2024spectral}. To  illustrate, we briefly outline the main steps of Donninger’s approach below. Consider an abstract Cauchy problem written in the operator form 
	\begin{align*}
		\partial_s \mathbf{q} = \mathbf{L}_\alpha \mathbf{q} + \mathbf{N}(\mathbf{q}), \quad \mathbf{q}(0,\cdot)  = \mathbf{q}_0, 
	\end{align*}
	where $\mathbf{L}_\alpha$ is the linearized operator and $\mathbf{N}(\mathbf{q})$ represents the nonlinear term.
	\begin{itemize}
		\item[Step 1] Establish the mode stability of the linearized operator by excluding all unstable eigenvalues except $\{0,1\}$. This reduces to solving the eigenvalue problem  $\mathbf{L}_{\alpha} \mathbf{q} = \lambda \mathbf{q}$, $\operatorname{Re}\lambda \ge 0$;
		\item[Step 2] Prove the generation of a semigroup $\mathbf{S}_{\alpha}(s)$ associated with the linearized operator $\mathbf{L}_\alpha$ and characterize its essential spectrum. The key idea is to decompose the linearized operator into a combination of a maximally dissipative operator (i.e., the free wave operator) and a compact perturbation, given by ${\mathbf{L}}_{\alpha}=\mathbf{L} + \mathbf{L}_{\alpha,1}$;
		\item[Step 3] Establish linear stability within the stable subspace. This is achieved by isolating the unstable modes using the Riesz projection $\mathbf{P}_{\lambda,\alpha} :=\frac{1}{2 \pi \mathrm{i}} \int_{\gamma_{\lambda}} \mathbf{R}_{\mathbf{L}_\alpha}(z) d z$ where  $\gamma_{\lambda}$ is a suitable contour for $\lambda= 0,1$.
		The growth bound of the semigroup is then derived using uniform resolvent estimates;
		%$\left\|\mathbf{S}_\alpha(s) \tilde{\mathbf{P}}_\alpha \mathbf{q}\right\|_{\mathcal{H}^k(B)} \le M e^{-w_0 s}\left\|\tilde{\mathbf{P}}_\alpha \mathbf{q}\right\|_{\mathcal{H}^k(B)}$
		\item[Step 4] Prove nonlinear stability by applying modulation techniques or the Lyapunov-Perron method, based on the following semigroup formulation
		\begin{align*}
			\mathbf{q}(s,\cdot) = \mathbf{S}_\alpha(s)\mathbf{q}(0,\cdot) + \int_0^{s} \mathbf{S}_\alpha(s-\tau) N(\mathbf{q}(\tau)) d\tau.
		\end{align*}
	\end{itemize}
	This spectral-theoretic framework is highly effective for analyzing non-self-adjoint operators with unstable eigenvalues and has been successfully applied to prove the stability of ODE blow-up solutions for nonlinear wave equations with power nonlinearities \cite{donninger2014stable,donninger2016blowup,wallauch2023strichartz,ostermann2024stable}. However, several new challenges arise in Steps 1–3 within our setting. Specifically,
	\begin{itemize}
		\item In Step 1, the eigen-equation for the eigenvalue $\lambda$ is given by
		\begin{align}\label{eigen-1-intro}
			(\lambda^2+\lambda )\varphi + \left((2\lambda+2) y+ \frac{2\alpha}{\sqrt{1+\alpha}+y} \right)\partial_y \varphi  +(y^2-1)\partial_{yy} \varphi =0,
		\end{align}    
		which is a Heun type ODE with four regular singular points at $\pm 1, -\sqrt{1+\alpha},\infty$. The existence of smooth solutions to such ODEs remains widely open, commonly referred to as the ``connection problem" \cite{donninger2024spectral}. In contrast, in the context of ODE blow-up stability  \cite{donninger2014stable,donninger2016blowup,wallauch2023strichartz,ostermann2024stable}, the corresponding ODE has only three regular singular points, making it a hypergeometric ODE. The connection problem for hypergeometric ODEs has already been solved in the literature; see, for example, \cite{olver2010nist}. To solve \eqref{eigen-1-intro}, one possible approach is to employ the ODE techniques developed by Cosin, Glogic, and Donninger  \cite{costin2016stability,costin2017mode,glogic2020threshold,glogic2021co}. However, such ODE analysis are highly problem-specific and technically delicate. To circumvent these challenges, we introduce a transformation inspired by the Lorentz transformation in self-similar variables:
		\begin{align*}
			\psi(y') = \left( \frac{1-\gamma y'}{\sqrt{1-\gamma^2}}\right)^{-\lambda} \varphi\left(\frac{y'-\gamma}{1-\gamma y'}\right), \quad \gamma = \frac{1}{\sqrt{1+\alpha}}.
		\end{align*}
		This transformation reduces \eqref{eigen-1-intro} to a hypergeometric ODE:
		\begin{align*}
			(\lambda^2-\lambda )\psi + \left(2\lambda y'+ 2 \sqrt{1+\alpha} \right)\partial_{y'} \psi  +(y'^2-1)\partial_{y'y'} \psi =0. 
		\end{align*}
		For further details, see Section \ref{subsec:mode-stability}. 
		\item In Step 2, as mentioned above, the non-compactness of the perturbation $\mathbf{L}_{\alpha,1}$ leads to a loss of information about the essential spectrum, which in turn causes the breakdown of Donninger’s spectral-theoretic framework. Since $\mathbf{L}_{\alpha,1}$ is merely a bounded operator, it becomes necessary to incorporate $\mathbf{L}_{\alpha,1}$ into the free wave operator $\mathbf{L}$ and identify the intrinsic dissipative structure of the full linearized operator $\mathbf{L}_{\alpha}$. This is achieved by establishing the following stronger dissipation estimate for higher derivatives:
		\begin{align*}
			\operatorname{Re }\langle\langle \mathbf{L}_{\alpha} \mathbf{q}, \mathbf{q}\rangle\rangle_k \le -\left(k -\sqrt{1+\alpha}-\frac{1}{2} -\frac{\epsilon}{2}\right) \left\|\mathbf{q} \right\|_{\mathcal{H}^k} + c_{\epsilon,k,\alpha} \left\|\mathbf{q} \right\|_{\mathcal{H}^{k-1}},
		\end{align*} 
		(see Proposition \ref{prop-hat-L-alpha}). Additionally, inspired by the abstract functional setting in \cite{MRR22,kim2022self}, we control the lower-order derivatives through a subcoercivity estimate (see Lemma \ref{lem-subcoer}). By combining these two properties, we derive a new decomposition to recover compactness:
		$$ \mathbf{L}_{\alpha} = \hat{\mathbf{L}}_{\alpha} + \hat{\mathbf{P}}_{\alpha},$$
		where 
		\begin{align*}
			\hat{\mathbf{P}}_{\alpha}=\sum_{i=1}^N\left \langle\cdot, {\Pi}_{\alpha,i}\right\rangle_k {\Pi}_{\alpha,i},
		\end{align*}
		is a finite rank projection operator and $\hat{\mathbf{L}}_{\alpha}$ is maximally dissipative operator (see Proposition \ref{prop-hat-L-alpha}).  This new strategy for the spectral analysis of non-self-adjoint operators is robust and well-suited to handling non-compact perturbations of the free wave operator. Moreover, it provides a versatile framework that can potentially be applied to the study of self-similar solutions in other equations. 
		%\begin{itemize}
		%\item  For $k\ge 0$, there exist $\epsilon_n>0$ with $\lim _{n \rightarrow \infty} \epsilon_n= 0$ and $\{\Pi_i\}_{1\le i\le n} \subset H^{k+1}(-1,1), c_n>0$ such that for all $n \geq 0, q \in H^{k+1}(-1,1)$,
		%\begin{align*}
		% \epsilon_n \|q\|_{H^{k+1}} \ge  \left\| q\right\|^2_{H^{k}} - c_n \sum_{i=1}^n\left(q, \Pi_i\right)_{L^2}^2  
		%\end{align*} 
		%\end{itemize}

		\item In Step 3, uniform resolvent estimates are crucial for translating spectral information into growth bounds for semigroups. Previous studies on uniform resolvent estimates, as noted in \cite{chen2023singularity}, are limited to compact perturbations of the free wave operator. Instead of relying on soft arguments, our approach directly estimates the resolvent by establishing a lower bound for the inner product $\langle (\lambda - \mathbf{L}_{\alpha}) \mathbf{q}, \mathbf{q}\rangle_k$. The key innovation lies in decomposing the inner product into its real and imaginary parts. The real part allows us to control the highest-order norm using the aforementioned stronger dissipation of higher derivatives. The  $L^2$  norm is then controlled by exploiting the structure of the imaginary part (see Proposition \ref{prop:resolvent-esti}). By combining this direct resolvent estimate method with the improved spectral analysis strategy outlined above, we successfully address the limitations of Donninger’s spectral-theoretic framework for non-compact perturbations.

		\item Furthermore, there are additional technical challenges that require careful consideration. For instance, due to the translation invariance of the profile $u$  (i.e., if  $u$  is a solution to \eqref{eq:nlw}, then  $u+\kappa$  is also a solution), even the free wave operator in self-similar variables has an unstable eigenvalue $0$. To address this issue, we introduce a modified free wave operator, which is constructed by adding a compact perturbation to the original operator. This modification effectively removes the eigenvalue  $0$  and ensures the generation of the semigroup for the linearized operator  $\mathbf{L}_{\alpha}$ .
		
		Another key aspect of our analysis is the existence of a generalized eigenfunction corresponding to the eigenvalue  $0$  (see Proposition \ref{prop-spectrum}). This new unstable mode arises from the free parameter $ \alpha$  and reflects the non-self-adjoint nature of the linearized operator  $\mathbf{L}_{\alpha}$. Consequently, the algebraic multiplicity of the eigenvalue  $0$  exceeds its geometric multiplicity, adding complexity to the linear analysis (see Proposition \ref{prop-growth-bound}).

	\end{itemize}

     Finally, we emphasize that the approach developed in this paper is both general and robust. Its flexibility allows for straightforward extensions to higher-dimensional settings, providing a natural framework for analyzing similar nonlinear wave equations in diverse contexts, as long as the self-similar solution is explicit. Moreover, the methodology is not limited to the specific equation studied here; it is adaptable to a broad class of nonlinear models with analogous structures. These features make it a powerful tool for investigating complex blow-up dynamics and stability properties in various applications. While this paper focuses on the one-dimensional case, we plan to explore specific higher-dimensional examples and other nonlinear models in future work, further demonstrating the adaptability of our approach.

	% 补充方法上的创新：问题的主要困难来自于线性化算子的非自伴性，存在不稳定特征值，以及非紧扰动。我们的方法改进了Donninger的算子半群框架使得可以处理非紧甚至非相对紧的扰动，这是基于我们novel的预解估计方法以及引入Merle等人的泛函框架来得到新的分解并recover紧性，这一方法不依赖于方程的特殊结构；另外特征值问题求解Heun type ODE，我们利用自相似变量下的Lorentz变换；第三是线性化算子的分解（引入modified free wave，free wave 由于scaling存在平移不变性带来的unstable mode）生成半群，更容易证明maximal性质；最后还需要处理广义特征值，代数重数严格大于几何重数
	
	%另外，提出了一个改进的数值scheme可以解决这类Heun problem；可以只放在intro
	
   \subsection{Structure of the paper} 
 In Section \ref{sec:2}, we construct smooth generalized self-similar blow-up solutions through a direct ODE analysis, providing the proof of Theorem \ref{thm:existence}. Section \ref{sec:3} is devoted to establishing the mode stability of the generalized self-similar solutions. In Section \ref{sec:4}, we perform a detailed spectral analysis of the linearized operator, which constitutes a central part of our proof. To address the non-compact perturbation arising from the derivative nonlinear term, we develop a new decomposition of the linearized operator to restore compactness (see Section \ref{subsec:4.3}). We then derive the uniform resolvent estimate in Section \ref{subsec:resolvent}. In Section \ref{sec:5}, we establish the linearized stability along the stable subspace. Finally, in Section \ref{sec:6}, we prove the nonlinear stability of the solutions using the Lyapunov-Perron method.

	\subsection{Notation}
	%Denote the interval $I=(-1,1)$.
	
	In this paper, we adopt the following notations. The open interval $(x_0-r,x_0+r)$ centered at $x_0\in \mathbb{R}$ is denoted by $B_r(x_0)$, with $B_r = B_r(0), B = B_1$. The closure of $B_r(x_0)$ is denoted by $\overline{B_r(x_0)} = [x_0-r,x_0+r]$. For $T>0, x_0\in \mathbb{R}$, the backward light cone is defined as
	$$\Gamma(T,x_0) :=\left\{(t, x) \in\left[0, T\right) \times \mathbb{R}:\left|x-x_0\right| \le T-t\right\}.$$ 
	%For $w\in \mathbb{R}$, we denote $\mathbb{H}_{w}=\{z\in \mathbb{C} \mid \operatorname{Re} z >w\}$ for a right half-plane and $\overline{\mathbb{H}_{w}} = \{z\in \mathbb{C} \mid \operatorname{Re} z \ge w\}$ a closed right half-plane. 
	For $a,b\in \mathbb{R},$ we use the notation $a\lesssim b$ to mean that there exists a constant $C>0$ such that $a\le Cb$, and $a\lesssim_{\alpha} b$ indicates that the constant $C$ depends on $\alpha$.
	
	For functions $f:\mathbb{R}\to \mathbb{R}, y\mapsto f(y)$, we denote the first derivative by  $f' = \partial_y f = \partial f$ and the $i$-th order derivative by $\partial_y^{i} f = \partial^{i} f$. For a domain $\Omega \subset \mathbb{R}$,  the set $C^{\infty}(\bar{\Omega})$  consists of all smooth functions on $\Omega$ whose derivatives are continuous up to the boundary of $\Omega$. If $\Omega$ is bounded, we define the inhomogeneous Sobolev norm and homogeneous Sobolev seminorm for $k \in \mathbb{N} \cup\{0\}$ as
	\begin{align*}
		\|f\|_{H^k(\Omega)}=\left(\sum_{0 \le i \leq k}\left\| \partial^i f\right\|_{L^2(\Omega)}^2\right)^{\frac{1}{2}} \quad \text { and } \quad\|f\|_{\dot{H}^k(\Omega)}=\left\| \partial^{k} f\right\|_{L^2(\Omega)}
	\end{align*}
	for $f \in C^{\infty}(\bar{\Omega})$, respectively. The Sobolev space $H^k(\Omega)$ is then defined as the completion of $C^{\infty}(\overline{\Omega)}$ with respect to the inhomogeneous Sobolev norm.  For $k\ge 0$, we define $\mathcal{H}^k =H^{k+1}(-1,1) \times H^{k}(-1,1)$.
	
	We use boldface notation for tuples of functions, for example:
	\begin{align*}
		\mathbf{f} \equiv\left(f_1, f_2\right) \equiv\left[\begin{array}{l}
			f_1 \\
			f_2
		\end{array}\right] \quad \text { or } \quad \mathbf{q}(t, .) \equiv\left(q_1(t, .), q_2(t, .)\right) \equiv\left[\begin{array}{l}
			q_1(t, .) \\
			q_2(t, .)
		\end{array}\right].
	\end{align*}
	Linear operators that act on tuples of functions are also displayed in boldface notation. For a closed linear operator $\mathbf{L}$ on a Banach space, we denote its domain by $\mathcal{D}(\mathbf{L})$, its spectrum by $\sigma(\mathbf{L})$, and its point spectrum by $\sigma_p(\mathbf{L})$. The resolvent operator is denoted by $\mathbf{R}_{\mathbf{L}}(z):=(z-\mathbf{L})^{-1}$ for $z \in \rho(\mathbf{L})=\mathbb{C} \backslash \sigma(\mathbf{L})$. The space of bounded operators on a Banach space $\mathcal{X}$ is denoted by $\mathcal{L}(\mathcal{X})$.
	
	For details on the spectral theory of linear operators, we refer to \cite{kato1995}. The theory of strongly continuous operator semigroups is covered in the textbook \cite{engel2000one}.
	
	%\section{Preliminaries}
	%\subsection{Solutions in lightcones}
	%\subsection{blow-up surface}
	
	\section{Construction of smooth blow-up solutions}\label{sec:2}
	In this section, we prove that all exact self-similar blow-up solutions to Eq. \eqref{eq:nlw} are singular inside the backward lightcone. Then, we derive our construction of smooth generalized blow-up solutions. 
	
	\subsection{Non-existence of smooth self-similar solutions}
	Eq. \eqref{eq:nlw} is invariant under the scaling transformation 
	\begin{align*}
		u\mapsto u^{\lambda} \quad \text{with~} \quad u^{\lambda}(t,x) = u\left(\frac{t}{\lambda},\frac{x}{\lambda}\right)
	\end{align*}
	for any $\lambda>0$. For any $T>0, x_0\in \mathbb{R}$, the self-similar ansatz of the nonlinear wave equation \eqref{eq:nlw} is 
	\begin{align*}
		u(t,x) = U\left(s,y\right),
	\end{align*}
	where
	$$  s=-\log(T-t)+\log(T) = -\log\left(1-\frac{t}{T}\right), \quad \text{and~} y = \frac{x-x_0}{T-t}.$$
	In terms of the self-similar variables $s $ and $y$, Eq. \eqref{eq:nlw} can be rewritten as
	\begin{align}\label{eq:nlw-ss}
		\partial_{ss} U + \partial_{s} U +2y\partial_{y}\partial_s U + 2 y\partial_y U  + (y^2-1) \partial_{yy}U = (\partial_y U)^2.
	\end{align}
	We first consider the exact self-similar solutions, i.e. stationary solutions $U = U(y)$ to Eq. \eqref{eq:nlw-ss}, satisfying
	\begin{align}\label{eq:ss}
		2 y\partial_y U + (y^2-1) \partial_{yy}U = (\partial_y U)^2,
	\end{align}
	which is equivalent to
	$$\partial_y ((y^2-1)\partial_y U) = (\partial_y U)^2.$$
	Regarding the above equation as a first-order ODE, we can solve \eqref{eq:ss} explicitly as 
	\begin{align*}%\label{ss-profile}
		\partial_y U = \left(c(y^2-1) +\frac14 \left(2y+(y^2-1)\log \left|\frac{y-1}{y+1}\right|\right)\right)^{-1} \text{~or~} 0,
	\end{align*}
	where $c\in \mathbb{R}$ is an arbitrary constant. We find that all non-trivial solutions have a singular point inside the light cone $|y|<1$. Indeed, denote the denominator $$p_c(y):= c(y^2-1) +\frac14 \left(2y+(y^2-1)\log \left|\frac{y-1}{y+1}\right|\right),$$ then we have 
	$$p_c(-1) = -\frac{1}{2}, \quad p_c(1) = \frac{1}{2}.$$
	By the Intermediate Value Theorem, there exists a point $y_0 \in (-1,1)$ such that $p_c(y_0) = 0$ which is a singular point of $\partial_y U$.
	
	\subsection{Smooth generalized self-similar blow-up solutions with a log growth}
	Instead of taking exact self-similar ansatz $U=U(y)$, we now consider the following generalized self-similar profiles 
	\begin{align}\label{eq:ss-profile-log}
		u(t,x) =U(s,y) = \alpha s + \tilde{U}(y). 
		%= -\alpha\log\left(1-\frac{t}{T}\right) + \tilde{U}(y)
	\end{align}
	for any non-zero constant $\alpha\in \mathbb{R}$, which allows for a linear growth in $s$. Note that this introduces a logarithmic growth in $t$ since $s= -\log\left(1-\frac{t}{T}\right)$. Then, $\tilde{U}$ satisfies
	%\begin{equation}\label{ss-eq-log}
	%    \partial_{ss} \tilde{U} + \partial_s \tilde{U} + 2 y\partial_y \partial_s \tilde{U} + 2 y\partial_y \tilde{U} +  (y^2-1) \partial_{yy}\tilde{U} + \alpha
	%   = (\partial_y \tilde{U})^2.
	%\end{equation}
	%We aim to find the stationary solutions to
	\begin{align}\label{regu-profile}
		2 y\partial_y \tilde{U} +  (y^2-1) \partial_{yy}\tilde{U} + \alpha
		= (\partial_y \tilde{U})^2. 
	\end{align}
	Fortunately, we find that for any $\alpha>0$, there exist regular solutions to \eqref{regu-profile} inside the light cone. Actually, let $\tilde{U}=a\log(b\pm y)$, we can solve \eqref{regu-profile} for 
	$$a=-\alpha,\quad  b=\sqrt{1+\alpha},$$
	i.e. 
	$$\tilde{U} = -\alpha\log(\sqrt{1+\alpha} \pm y).$$
	Since \eqref{regu-profile} is a Riccati type equation, we can further obtain all solutions to \eqref{regu-profile} as
	$$\tilde{U}_{\alpha,\beta,\kappa}= -\alpha\log(\sqrt{1+\alpha} + y) + W_{\alpha,\beta,\kappa}, \quad  y \in [-1,1],$$
	where 
	\begin{align}\label{eq:W}
		\begin{aligned}
			\partial_y W_{\alpha,\beta,\kappa} &=  \left(\frac{1+y}{1-y}\right)^{\sqrt{1+\alpha}} \frac{1}{ (y+\sqrt{1+\alpha})^2} \left( \beta + \int_{-1}^{y} \frac{(1+z)^{\sqrt{1+\alpha}-1}}{(1-z)^{\sqrt{1+\alpha}+1}} \frac{1}{ (z+\sqrt{1+\alpha})^2} dz \right)^{-1}\\
			&=  \frac{2 \alpha\sqrt{1+\alpha}}{(\sqrt{1+\alpha}-y)(\sqrt{1+\alpha}+y)} \left( \beta\left(\frac{1-y}{1+y}\right)^{\sqrt{1+\alpha}}  \frac{2 \alpha\sqrt{1+\alpha}(\sqrt{1+\alpha}+y)}{\sqrt{1+\alpha}-y} + 1 \right)^{-1}
		\end{aligned}
	\end{align}
	and 
	\begin{align*}
		W_{\alpha,\beta,\kappa}(0) =\kappa.
	\end{align*}
	The parameter $\beta$ is given by
	$$ \beta = \frac{1}{(1+\alpha)\partial_y W_{\alpha,\beta,\kappa}(0)} -\frac{1}{2 \alpha\sqrt{1+\alpha}},$$
	i.e.
	$$ \partial_y W_{\alpha,\beta,\kappa}(0) = \frac{2\alpha}{2\alpha(1+\alpha)\beta + \sqrt{1+\alpha}}.$$
	\iffalse  
	Note that $W_{\alpha,\beta,\kappa}$ satisfies
	$$ \lim_{y\to 1}\partial_y W_{\alpha,\beta,\kappa} =2\sqrt{1+\alpha}$$
	and 
	\begin{align*}
		\lim_{y\to -1}\partial_y W_{\alpha,\beta,\kappa} = \begin{cases}
			0, & \beta>0,\\
			2\sqrt{1+\alpha}, & \beta=0.
		\end{cases}
	\end{align*}
	\fi
	Moreover, we have
	\begin{itemize}
		\item When $\beta=0,$ $$\tilde{U}_{\alpha,0,\kappa} = -\alpha\log(\sqrt{1+\alpha}-y)+\kappa;$$
		\item When $\beta=+\infty,$ 
		$$\tilde{U}_{\alpha,\infty,\kappa} = -\alpha\log(\sqrt{1+\alpha}+y)+\kappa;$$
	\end{itemize}    
	Thus, for $\alpha>0$, both $\tilde{U}_{\alpha,0,\kappa}$ and $\tilde{U}_{\alpha,\infty,\kappa}$ are smooth inside the light cone. However, when $0<\beta<\infty, \alpha>0$, $\tilde{U}_{\alpha,\beta,\kappa}$ is smooth only if $\sqrt{1+\alpha}\in \mathbb{N}$. This gives countable    families of smooth blow-up profiles $\tilde{U}_{n^2-1,\beta,\kappa}, 0<\beta<\infty$ for $n\ge 2, n\in \mathbb{N}$. 
	
	Thus, we obtain the following smooth generalized self-similar profiles 
	$$U_{\alpha,\beta,\kappa}(s,y) = \alpha s+\tilde{U}_{\alpha,\beta,\kappa}.$$ 
	Going back to the original variables $t,x$, we have
	\begin{align}\label{gss}
		u_{\alpha,\beta,\kappa,T,x_0}(t,x) = -\alpha\log\left(1-\frac{t}{T}\right) -\alpha\log\left(\sqrt{1+\alpha}+\frac{x-x_0}{T-t}\right) + w_{\alpha,\beta,\kappa,T,x_0}(t,x), 
	\end{align}
	where 
	\begin{align*}
		&w_{\alpha,\beta,\kappa,T,x_0}(t,x) = W_{\alpha,\beta,\kappa}(s,y)\\ 
		& = \kappa + \int_0^{\frac{x-x_0}{T-t}} \frac{2\alpha\sqrt{1+\alpha}}{(\sqrt{1+\alpha}-z)(\sqrt{1+\alpha}+z)} \left( \beta\left(\frac{1-z}{1+z}\right)^{\sqrt{1+\alpha}}  \frac{2\alpha\sqrt{1+\alpha}(\sqrt{1+\alpha}+z)}{\sqrt{1+\alpha}-z} + 1 \right)^{-1} dz.
	\end{align*}
	Denote $\tilde{u}_{\alpha,\beta,\kappa,T,x_0}(t,x):=-\alpha\log\left(\sqrt{1+\alpha}+\frac{x-x_0}{T-t}\right) + w_{\alpha,\beta,\kappa,T,x_0}(t,x),$ we have 
	\begin{align*}
		&\tilde{u}_{\alpha,\beta,\kappa,T,x_0}(t,x) \\
		=&  \kappa -\alpha\log \sqrt{1+\alpha}+\int_0^{\frac{x-x_0}{T-t}} \frac{\alpha(1+z)^{\sqrt{1+\alpha}} - 2 \alpha^2\sqrt{1+\alpha}\beta(1-z)^{\sqrt{1+\alpha}}}{(1+z)^{\sqrt{1+\alpha}}(\sqrt{1+\alpha} -z) + 2 \alpha\sqrt{1+\alpha}\beta(1-z)^{\sqrt{1+\alpha}}(\sqrt{1+\alpha}+z)} dz.
	\end{align*}		
	%(Q: domain of the function $\tilde{U}$. Is there a parameter $B$ such that the solution $\tilde{U}_1$ is regular all all $y$. For $C=3,$ the desired $B=6$.)

	\subsection{Interpretation of generalized self-similar blow-up solutions}\label{subsec:interpretation}
	In this subsection, we demonstrate that the generalized self-similar blow-up solutions \eqref{gss} can be regarded as \textit{exact} self-similar blow-up solutions to a transformed version of \eqref{eq:nlw} obtained by taking the spatial derivative. More precisely, let $h = u_x$, then 
	\begin{align}\label{eq:h}
		h_{tt} - h_{xx} = 2hh_x.
	\end{align}
	Consider the corresponding self-similar ansatz 
	\begin{align*}
		h(t,x) = \frac{1}{T-t} H\left(s, y\right), \quad s= -\log(T-t) + \log(T), \ y=\frac{x-x_0}{T-t}.
	\end{align*}
	Then, $H$ satisfies
	\begin{equation}\label{ss-eq-h}
		\partial_{ss} H + 3\partial_s H + 2 y\partial_y \partial_s H + 4 y\partial_y H +  (y^2-1) \partial_{yy}H + 2H
		= 2H \partial_y H.
	\end{equation}
	The stationary solutions to \eqref{ss-eq-h} satisfy
	\begin{align}\label{ss-h}
		4 y\partial_y H +  (y^2-1) \partial_{yy}H + 2H
		= 2H \partial_y H. 
	\end{align}
	which is equivalent to 
	\begin{align*}
		\partial_y ((y^2-1)\partial_y H +  2yH )
		= \partial_y H^2. 
	\end{align*}
	Thus, 
	\begin{align*}
		(y^2-1)\partial_y H +  2yH +\alpha
		= H^2. 
	\end{align*}
	This goes back to the equation \eqref{regu-profile}. Using the same method, we can solve \eqref{ss-h} as
	\begin{align}
		H_{\alpha, \beta}(y)= \frac{\alpha}{\sqrt{1+\alpha} + y} \frac{1 - 2 \alpha\sqrt{1+\alpha}\beta \left(\frac{1-y}{1+y}\right)^{\sqrt{1+\alpha}} }{\frac{\sqrt{1+\alpha} -y}{\sqrt{1+\alpha} +y} + 2 \alpha\sqrt{1+\alpha}\beta \left(\frac{1-y}{1+y}\right)^{\sqrt{1+\alpha}} }, \quad y\in [-1,1].
	\end{align}
	Here $H_{\alpha,\beta}(y)$ is always continuous inside the light cone for all $\alpha>0, \beta\ge 0$, but $H_{\alpha,\beta}(y)$ does not belong to $C^k[-1,1]$ for $k>\sqrt{1+\alpha}$, when $\sqrt{1+\alpha}\notin \mathbb{N}$ and $0<\beta<\infty$.
	And for each fixed $y\in [-1,1]$, $H_{\alpha,\beta}(y)$ is monotone decreasing with respect to $\beta$, i.e.
	$$ H_{\alpha,\infty}(y) \le H_{\alpha, \beta}(y) \le H_{\alpha,0}(y).$$
	In particular, when $\beta=0$, 
	$$H_{\alpha,0}(y) = \frac{\alpha}{\sqrt{1+\alpha} - y};$$
	when $\beta=\infty,$
	$$H_{\alpha,\infty}(y) = \frac{-\alpha}{\sqrt{1+\alpha} + y}.$$
	Note that for $0<\beta<\infty$, 
	\begin{align*}
		&\lim_{y\to -1} H_{\alpha,\beta}(y) = -1 - \sqrt{1+\alpha} = H_{\alpha,\infty}(-1),\\
		&\lim_{y\to 1} H_{\alpha,\beta}(y) = 1 + \sqrt{1+\alpha} = H_{\alpha,0}(1).
	\end{align*}
	Thus, $H_{\alpha,\beta}$ connects to $H_{\alpha,0}$ and $H_{\alpha,\infty}$ at $-1$ and $1$ respectively. 
	Moreover, we have the following symmetry
	\begin{align*}
		H_{\alpha,\beta}(-y) = - H_{\alpha, \beta'} (y), \quad \beta' = \frac{1}{4\alpha^2(1+\alpha) \beta}.
	\end{align*}
	In particular, when $\beta_* = \frac{1}{2 \alpha\sqrt{1+\alpha}}$, $H_{\alpha, \beta_*}$ is odd, $H_{\alpha,0} (-y) = -H_{\alpha, \infty}(y)$.

	%We expect that $U_{\alpha,\infty,\kappa}(s,y)$ are only codimensional stable in $H^{k+1}\times H^k$ for $k\le \sqrt{1+\alpha}$.

	\section{Mode stability}\label{sec:3}
	
	In this section, we take the first step in studying the linear stability of the smooth generalized self-similar blow-up solutions $U_{\alpha, \infty,\kappa}$. Specifically, we aim to establish the mode stability of the linearized equation at  $U_{\alpha, \infty,\kappa}$, demonstrating that the linearized operator has no unstable eigenvalues other than $0$ and $1$, which arise from the symmetries of \eqref{eq:nlw}.
	
	%Hereafter we always fix the center of the backward lightcone $x_0=0$ for simplicity, since Eq. \eqref{eq:nlw} is invariant under the space translation, .
	
	Let $U = U_{\alpha, \infty,\kappa} + \eta $, then by \eqref{eq:nlw-ss} the corresponding linearized equation of $\eta$ reads
	\begin{align}\label{linear-ss}
		\partial_{ss} \eta + \partial_s \eta + 2 y\partial_y \partial_s \eta + \left(2y - 2 \partial_y U_{\alpha, \infty,\kappa} \right)\partial_y \eta + ( y^2 -1) \partial_{yy} \eta =0.  
		%\partial_{ss} \eta + 3\partial_s \eta + 2 y\partial_y \partial_s \eta + \left(4y - 2 H_{C, \infty} \right)\partial_y \eta + ( y^2 -1) \partial_{yy} \eta + (2-2\partial_y H_{\alpha,\infty}) \eta =0.
	\end{align}
	Denote $\mathbf{q} = (q_1, q_2)^T = (\eta, \partial_s \eta + y\partial_y \eta)^T$, then we can rewrite \eqref{linear-ss} as a first-order PDE system
	\begin{align}\label{linear-ss-system}
		\partial_s \left(\begin{array}{l}
			q_1  \\
			q_2  
		\end{array}\right) &= \left(\begin{array}{c}
			- y\partial_y q_1 + q_2 \\
			- q_2 - y\partial_y q_2 + \partial_{yy}q_1  + 2 \partial_y U_{\alpha,\infty,\kappa} \partial_y q_1 
		\end{array}\right)  =: \tilde{\mathbf{L}}_{\alpha}  \left(\begin{array}{l}
			q_1  \\
			q_2  
		\end{array}\right),
	\end{align}
	where the potential
	$$ \partial_y U_{\alpha,\infty,\kappa}(s, y) = -\frac{\alpha}{\sqrt{1+\alpha} + y} $$
	is independent of time $s$. Note that the linearized operator $\tilde{\mathbf{L}}_{\alpha}$ above does not depend on the parameters $\kappa$, $T$, and $x_0$, and it is non-self-adjoint, which makes the stability analysis of \eqref{linear-ss-system} difficult.  %To resolve this problem, in a suitable Sobolev space with super critical regularity. 
	\subsection{Unstable eigenvalues induced by symmetries}
	
	We postpone the rigorous functional setup of the linearized operator $\tilde{\mathbf{L}}_{\alpha}$ to the next section, and regard $\tilde{\mathbf{L}}_{\alpha}$ as a formal differential operator. We first study the possible unstable eigenvalues of $\tilde{\mathbf{L}}_{\alpha}$ as a step towards proving its mode stability. 
	\begin{defn}
		$\lambda \in \mathbb{C}$ is called an eigenvalue of $\tilde{\mathbf{L}}_{\alpha}$ if there exists a nontrivial smooth function $\varphi\in C^\infty[-1,1]$ such that 
		\begin{align}\label{eigen-eq}
			(\lambda^2+\lambda )\varphi + \left((2\lambda+2) y+ \frac{2\alpha}{\sqrt{1+\alpha}+y} \right)\partial_y \varphi  +(y^2-1)\partial_{yy} \varphi =0. 
		\end{align}
		We call an eigenvalue $\lambda$ is unstable if its real part is non-negative, i.e. $\operatorname{Re} \lambda \ge 0$. Otherwise, we call $\lambda$ a stable eigenvalue.
	\end{defn}
	Actually, if $\tilde{\mathbf{L}}_{\alpha} \mathbf{q} = \lambda \mathbf{q}$ for some $\mathbf{q} = (q_1,q_2)^T \in (C^\infty[-1,1])^2$, then $q_1$ satisfies \eqref{eigen-eq} and $\lambda $ is an eigenvalue of $\tilde{\mathbf{L}}_{\alpha}$. Conversely, if $\lambda $ is an eigenvalue of $\tilde{\mathbf{L}}_{\alpha}$ with a smooth eigenfunction $\varphi(y)$, then $\mathbf{q} := (\varphi, \lambda\varphi + y\partial_y \varphi)^T$ satisfies $\tilde{\mathbf{L}}_{\alpha} \mathbf{q} = \lambda \mathbf{q}$.
	
	\begin{remark}
		It is equivalent to define eigenvalues of $\tilde{\mathbf{L}}_{\alpha}$ such that the corresponding eigenfunctions are only in the Sobolev space $H^{k+1}(-1,1)$ for $k\ge k_{\alpha}+2$. This is due to the fact that the eigen-equation \eqref{eigen-eq} is a second-order elliptic equation whose solutions are always $C^\infty$ smooth in $(-1,1)$. Moreover, we shall prove that any $H^{k+1}(-1,1)$ solution to \eqref{eigen-eq} is also smooth (and even analytical) at the end points $\pm 1$, see Proposition \ref{prop-1dsol-space} in the Appendix.    
		%The exponent $k_\alpha+2$ of the Sobolev regularity is required to exclude possible unstable modes, since $\tilde{\mathbf{L}}_{\alpha}$ may allow for infinitely many different unstable eigenvalues in low regularity space. We shall discuss this later.
	\end{remark}
	
	\iffalse
	More precisely, for the linearized operator $L$, the corresponding eigenfunctions for eigenvalue $\lambda$ satisfy
	\begin{align}\label{eq:eigen}
		(\lambda^2+\lambda )\eta + \left((2\lambda+2) y- 2\partial_y U_{C,B} \right)\partial_y \eta  +(y^2-1)\partial_{yy} \eta =0,
	\end{align}
	where by \eqref{eq:W}
	$$ \partial_y U_{C,B} = \frac{C-2C^2\sqrt{1+\alpha} B\left(\frac{1-y}{1+y}\right)^{\sqrt{1+\alpha}} }{2 \alpha\sqrt{1+\alpha} B\left(\frac{1-y}{1+y}\right)^{\sqrt{1+\alpha}}(\sqrt{1+\alpha}+y) 
		+(\sqrt{1+\alpha}-y)}.$$
	\fi
	
	First, the symmetries of Eq. \eqref{eq:nlw-ss} give us some explicit unstable eigenvalues. More precisely, if $U(s,y)$ is a solution, then the following transformations of $U(s,y)$ are also solutions to \eqref{eq:nlw-ss}:
	\begin{itemize}
		\item (Space translation in $x$) for any $a \in \mathbb{R}$, the function
		%defined for all $s \in\left[-\log \left(t_0-t_1\right),+\infty\right)$ and $y \in$ $\left(-a e^s-1,-a e^s+1\right)$ by
		\begin{align*}
			U_1(s, y)= U\left(s, y+a e^s\right);
		\end{align*}
		\item (Time translation in $t$) for any $b \in \mathbb{R}$, the function 
		%$w_2(y, s)$ defined for all $s \geqslant-\log \left(t_0-t_1-b\right)$ and $|y|<$ $1+b e^s$ by
		\begin{align*}
			U_2(s, y)= U\left(s-\log \left(1+b e^s\right), \frac{y}{1+b e^s}\right);
		\end{align*} %\left(1+b e^s\right)^{-\frac{2}{p-1}}
		\item(Time translation in $s$)  for any $c \in \mathbb{R}$, the function 
		%$w_3(y, s)$ defined for all $|y|<1$ and $s \in\left[-\log \left(t_0-t_1\right)-c,+\infty\right)$ by
		\begin{align*}
			U_3(s,y) = U(s+c, y);
		\end{align*}
		\item (Translation in $u$)  for any $d \in \mathbb{R}$, the function 
		%$U_4(s,y)$ defined for all $|y|<1$ and $s \in\left[-\log \left(t_0-t_1\right),+\infty\right)$ by
		\begin{align*}
			U_4(s,y)=U(s, y) +d;
		\end{align*}
		\item(Scaling)  for any $\mu \in \mathbb{R}$, the function 
		%$w_5(y, s)$ defined for all $|y|<1$ and $s \in [-\log (t_0-t_1-(\lambda-1)T)-\log(\lambda),+\infty )$ by
		\begin{align*}
			U_5(s,y)= U\left(\frac{y}{1+ (\mu-1)T e^s}, s-\log \left(1+(\mu-1)T e^s\right) + \log(\mu) \right).
		\end{align*}
		This invariance can be viewed as a composition of the time translation in $t$ and $s$, with $b=(\mu-1)T$ and $c= \log(\mu)$.
		%\item[Lorentz transformation]  The transposition in self-similar variables of the Lorentz transform which will be given in this section.
	\end{itemize}
	For generalized self-similar blow-up solutions 
	$$U_{\alpha,\infty,\kappa}(s,y )= \alpha s - \alpha\log(\sqrt{1+\alpha}+y) + \kappa,$$
	the parameter $\alpha$ also gives another family of solutions to  Eq. \eqref{eq:nlw-ss}. Moreover, we observe that  
	\begin{itemize}
		\item $\partial_{a} U_1|_{a=0}$ and $ \partial_{b} U_2|_{b=0}$ generate the same 1d eigenspace of eigenvalue 1;
		\item $\partial_{c} U_3|_{c=0}$ and $ \partial_{d} U_4|_{d=0}$ generate the same 1d eigenspace of eigenvalue 0;
		\item $\partial_{\alpha} U_{\alpha,\infty,\kappa}$ is a generalized eigenfunction of eigenvalue 0, i.e. 
		for $\mathbf{q} = (\partial_{\alpha} U_{\alpha,\infty,\kappa}, y\partial_y\partial_{\alpha} U_{\alpha,\infty,\kappa})^T$, $\tilde{\mathbf{L}}_{\alpha} \mathbf{q} =(1,0)^T \neq 0$ and $\tilde{\mathbf{L}}_{\alpha}^2 \mathbf{q} =0$.
	\end{itemize}
	\begin{remark} Note that when $\beta=\infty$, $\partial_\beta U_{\alpha,\beta,\kappa} \equiv 0$; when $\beta=0$, $\partial_\beta U_{\alpha,\beta,\kappa} $ is singular at $y=-1$. Therefore, the parameter $\beta$ does not generate an unstable eigenvalue of $\tilde{\mathbf{L}}_{\alpha}$.
		%but for $0<\beta<\infty$, $\partial_\beta U_{\alpha,\beta,\kappa} $ is regular for $|y|\le 1$. In the following, we would first consider the stability of $U_{\alpha,\infty,\kappa}$, since the corresponding linearized operator has only two unstable directions.   
	\end{remark}
	In summary, the linearized operator $\tilde{\mathbf{L}}_{\alpha}$ has at least two unstable eigenvalues, $0$ and $1$, each with an explicit eigenfunction. Nevertheless, we introduce the following notion of mode stability, see \cite{donninger2011stable,donninger2024spectral}. 
	\begin{defn}[Mode stability]
		We say that the blow-up solution $U_{\alpha,\infty,\kappa}$ is mode stable if the existence of a nontrivial smooth function $\varphi \in C^{\infty}[-1,1]$ that satisfies Eq. \eqref{eigen-eq} necessarily implies that $\lambda \in \{0,1\}$ or $\operatorname{Re} \lambda<0$.
	\end{defn}
	
	\begin{remark}
		In fact, we shall prove a stronger result than the mode stability above, replacing  $\operatorname{Re} \lambda<0$ by $\operatorname{Re} \lambda\le -1$. This could help us derive a uniform spectral gap, see Proposition \ref{prop-spectrum}. 
	\end{remark}
	
	To study the mode stability of $U_{\alpha,\infty,\kappa}$, we need to solve Eq. \eqref{eigen-eq}, which can be rewritten as a Heun type ODE:
	\begin{align}\label{eigen-eq-y}
		\varphi'' + \left[\frac{\lambda+\sqrt{1+\alpha}}{y-1} + \frac{\lambda - \sqrt{1+\alpha}}{y+1} + \frac{2}{y+ \sqrt{1+\alpha}}\right] \varphi' + \frac{\lambda^2+\lambda}{y^2-1} \varphi = 0.
	\end{align}
	Note that Eq. \eqref{eigen-eq-y} has four regular singular points $\pm 1, -\sqrt{1+\alpha}$ and $\infty$.  The standard form of a Heun type equation reads
	\begin{align*}
		g^{\prime \prime}(z)+\left[\frac{\gamma}{z}+\frac{\delta}{z-1}+\frac{\epsilon}{z-d}\right] g^{\prime}(z)+\frac{ab z-c}{z(z-1)(z-d)} g(z)=0,
	\end{align*}
	where $a, b, c, d, \gamma, \delta, \epsilon \in \mathbb{C}$ and $\epsilon = a+b +1 -\gamma-\delta.$ 
	Let $y=2z-1$, we can rewrite \eqref{eigen-eq-y} as the standard form
	\begin{align}\label{eigen-eq-z}
		\varphi'' + \left[\frac{\lambda - \sqrt{1+\alpha}}{z}+ \frac{\lambda+\sqrt{1+\alpha}}{z-1}  + \frac{2}{z+ \frac{\sqrt{1+\alpha}-1}{2}}\right] \varphi' + \frac{\lambda^2+\lambda}{z(z-1)} \varphi = 0,
	\end{align}
	where we take $\gamma = \lambda -\sqrt{1+\alpha}, \delta = \lambda +\sqrt{1+\alpha}, d = -\frac{\sqrt{1+\alpha}-1}{2}, a=\lambda, b = \lambda+1, c= -\frac12 (\lambda^2+\lambda) (\sqrt{1+\alpha}-1), \epsilon = 2 = \alpha+\beta +1 -\gamma-\delta.$ 
	
	However, the existence of smooth solutions to such type of ODEs remains largely open, which is also known as ``connection problem" \cite{donninger2024spectral}. We bypass this difficulty using the Lorentz transformation in self-similar variables to convert the Heun type ODE \eqref{eigen-eq-z} to a new hypergeometric ODE, see \eqref{eigen-eq-z'}. 
	
	%In the following, we shall first use Frobenius analysis to convert the ODE problem to a connection problem, and then further reduce it to prove the non-triviality of the corresponding Wronskian, which can be verified rigorously by numerical computations. This scheme is motivated by \cite{bizon2005unusual,bizon2015generic} and was first adapted to study the mode stability of self-similar blow-up solutions to wave maps, see \cite{bizon2005unusual,donninger2011stable}. In the next subsection, we will also present a novel rigorous proof of the mode stability, where we bypass the difficulty by using the Lorentz transformation to convert the Heun type ODE \eqref{eigen-eq-z} to a new hypergeometric ODE, see \eqref{eigen-eq-z'}. Nevertheless, the numerical scheme is of interest itself and can be widely applied to other problems where rigorous proofs are absent.

	\subsection{Proof of the mode stability}\label{subsec:mode-stability}
	The idea is motivated by our new observation that the blow-up solution $U_{\alpha,\infty,\kappa}$ can be transformed to the ODE blow-up solution of a new derivative NLW equation using the Lorentz transformation. More precisely, we recall the Lorentz transformation and its inverse
	\begin{align}
		\left\{ \begin{aligned}
			t' = \frac{t-\gamma x}{\sqrt{1-\gamma^2}}, \\
			x' = \frac{x-\gamma t}{\sqrt{1-\gamma^2}},
		\end{aligned}\right. \quad \text{~and~} \left\{ \begin{aligned}
			t = \frac{t'+\gamma x'}{\sqrt{1-\gamma^2}}, \\
			x = \frac{x'+\gamma t'}{\sqrt{1-\gamma^2}},
		\end{aligned}\right.
	\end{align}
	where $\gamma\in(-1,1)$. Let $v(t',x') = u(t,x) = u(t(t',x'),x(t',x'))$, then if $u$ solves \eqref{eq:nlw}, $v(t',x')$ satisfies
	\begin{align}\label{nlw-transform}
		v_{t't'} - v_{x'x'} = \frac{1}{1-\gamma^2}\left(v_{x'}^2 -2\gamma v_{x'} v_{t'} +\gamma^2 v_{t'}^2\right).
	\end{align}
	And under this transformation, the generalized self-similar blow-up solutions
	\begin{align*}
		u_{\alpha,\infty,\kappa,T,x_0}(t,x) &= -\alpha\log\left(1-\frac{t}{T}\right) -\alpha\log\left(\sqrt{1+\alpha}+\frac{x-x_0}{T-t}\right) +\kappa \\
		&= -\alpha \log\left(\sqrt{1+\alpha}(T-t)+x-x_0\right) +\kappa +\log(T)
	\end{align*}
	correspond to 
	\begin{align*}
		v_{\alpha,\infty,\kappa,T,x_0}(t',x') &:= -\alpha \log\left(\sqrt{1+\alpha}(T-t)+x-x_0\right) +\kappa +\log(T)  \\
		& = -\alpha \log\left(\sqrt{1+\alpha}T - \frac{\sqrt{1+\alpha}-\gamma}{\sqrt{1-\gamma^2}}t' + \frac{1 - \sqrt{1+\alpha}\gamma}{\sqrt{1-\gamma^2}} x'- x_0\right) +\kappa +\log(T). 
	\end{align*}
	In particular, if we take $\gamma = \frac{1}{\sqrt{1+\alpha}}$, then $v_{\alpha,\infty,\kappa,T,x_0}$ is independent of the spacial variable $x'$, i.e $v_{\alpha,\infty,\kappa,T,x_0}$ is an ODE blow-up solution to Eq. \eqref{nlw-transform}. Note that even with this transformation it is not easy to directly reduce the stability of $u_{\alpha,\infty,\kappa,T,x_0}$ to the stability of $v_{\alpha,\infty,\kappa,T,x_0}$ due to the following reasons:
	\begin{enumerate}
		\item The initial perturbations to $u_{\alpha,\infty,\kappa,T,x_0}$ and $v_{\alpha,\infty,\kappa,T,x_0}$ are imposed on different regions, i.e. $\{(t,x) \in \mathbb{R}^2 : t=0\}$ and  $\{(t',x')  \in \mathbb{R}^2 : t'=0\} = \{ (t,x) \in \mathbb{R}^2 : t-\gamma x=0\}$ respectively.
		\item The stable evolution in $t'$ direction does not necessarily imply a stable evolution in $t$ direction.
		%\item The stability problem of $v_{\alpha,\infty,\kappa,T,x_0}$ has its own technical difficulties, for example, the essential spectrum of the linearized operator at $v_{\alpha,\infty,\kappa,T,x_0}$ is different with the free wave operator and unstable. We will discuss this in detail in our future work.
	\end{enumerate}
	Nevertheless, the advantage of this transformation is that it allows us to study the spectrum of the linearized operator at $u_{\alpha,\infty,\kappa,T,x_0}$ by analyzing the spectrum of the linearized operator at $v_{\alpha,\infty,\kappa,T,x_0}$ in self-similar variables, which facilitates a rigorous proof of the mode stability of $U_{\alpha,\infty,\kappa}$. To begin with, we consider the Lorentz transformation in self-similar variables which was first introduced by \cite{merle2007existence} for NLW with power nonlinearities. 
	
	For $T'>0, x_0'\in \mathbb{R}$, define the corresponding self-similar variables for $t'$ and $x'$ as 
	$$ s' = -\log(T'-t') + \log(T'), \quad y' =\frac{x'-x_0'}{T'-t'},$$
	and let $V(s',y') = v(t',x')$, if $v(t',x')$ solves Eq. \eqref{nlw-transform}, then $V(s',y')$ satisfies 
	\begin{align}\label{nlw-ss-transform}
		\begin{aligned}
			&\partial_{s's'} V + \partial_{s'} V +2y'\partial_{y'}\partial_{s'} V + 2 y'\partial_{y'} V  + ({y'}^2-1) \partial_{y'y'} V \\
			& = \frac{1}{1-\gamma^2}\left( (\partial_{y'} V)^2 -2\gamma \partial_{y'} V (\partial_{s'} V + y'\partial_{y'} V ) + \gamma^2 (\partial_{s'} V + y'\partial_{y'} V )^2  \right).
		\end{aligned}
	\end{align}
	A direct computation implies the following transformation in self-similar variables.
	\begin{prop}[The Lorentz transform in self-similar variables]\label{prop-lorentz-ss} Consider $U(s,y)$ a solution to \eqref{eq:nlw-ss} defined for all $|y|<1$, and introduce for any $\gamma \in (-1,1)$, the function $V \equiv \mathcal{T}_{\gamma} U$ defined by
		\begin{align*}
			V(s',y') = U(s,y), \text{~where~}  s =s' - \log\frac{1-\gamma y'}{\sqrt{1-\gamma^2}}, \quad y=\frac{y'-\gamma}{1-\gamma y'},
		\end{align*}
		then $V(s',y')$ solves \eqref{nlw-ss-transform} for all $|y'|<1$.
	\end{prop}
	
	Note that for $U_{\alpha,\infty,\kappa} = \alpha s -\alpha\log(\sqrt{1+\alpha}+y) +\kappa$, the corresponding function $V_{\alpha,\infty,\kappa} \equiv \mathcal{T}_{\gamma} U_{\alpha,\infty,\kappa}$ solves Eq. \eqref{nlw-ss-transform} and satisfies
	\begin{align*}
		V_{\alpha,\infty,\kappa}(s',y') &= \alpha\left(s' - \log\frac{1-\gamma y'}{\sqrt{1-\gamma^2}}\right) - \alpha\log\left(\sqrt{1+\alpha}+\frac{y'-\gamma}{1-\gamma y'}\right) +\kappa\\
		& = \alpha s' -\alpha\log\frac{1}{\sqrt{1-\gamma^2}}\left( \sqrt{1+\alpha}-\gamma + (1-\sqrt{1+\alpha}\gamma)y' \right) +\kappa.
	\end{align*}
	Again, when $\gamma = \frac{1}{\sqrt{1+\alpha}}$, $V_{\alpha,\infty,\kappa}$ is independent of the variable $y'$. Hereafter, we always take $\gamma = \frac{1}{\sqrt{1+\alpha}}$. Let $V= V_{\alpha,\infty,\kappa} + \xi$, then the linearized equation of  Eq. \eqref{nlw-ss-transform}  at $V_{\alpha,\infty,\kappa}$ is
	\begin{align}\label{linearized-ss-transform}
		\partial_{s's'} \xi - \partial_{s'} \xi +2y'\partial_{y'}\partial_{s'} \xi + 2\sqrt{1+\alpha} \partial_{y'} \xi  + ({y'}^2-1) \partial_{y'y'} \xi  = 0.
	\end{align}
	Denote $\mathbf{r} = (r_1, r_2) = (\xi, \partial_{s'} \xi + y'\partial_{y'} \xi)$, then we can rewrite \eqref{linearized-ss-transform} as a first-order PDE system
	\begin{align}\label{linear-ss-system-transform}
		\partial_s \left(\begin{array}{l}
			r_1  \\
			r_2  
		\end{array}\right) &= \left(\begin{array}{c}
			- y'\partial_{y'} r_1 + r_2 \\
			r_2 - y'\partial_{y'} r_2 + \partial_{y'y'}r_1  - 2 \sqrt{1+\alpha}\partial_{y'} r_1 
		\end{array}\right)  =: \mathbf{L}^{\prime}_\alpha  \left(\begin{array}{l}
			r_1  \\
			r_2  
		\end{array}\right).
	\end{align}
	\begin{defn}
		$\lambda \in \mathbb{C}$ is called an eigenvalue of $\mathbf{L}^{\prime}_\alpha$ if there exists a nontrivial smooth function $\psi\in C^{\infty}[-1,1]$ such that 
		\begin{align}\label{eigen-eq-transform}
			(\lambda^2-\lambda )\psi + \left(2\lambda y'+ 2 \sqrt{1+\alpha} \right)\partial_{y'} \psi  +(y'^2-1)\partial_{y'y'} \psi =0. 
		\end{align}
	\end{defn}
	%Apparently, if $\tilde{\mathbf{L}}_{\alpha} q = \lambda q$ for some $q = (q_1,q_2) \in (C^\infty([-1,1]))^2$, then $q_1$ satisfies \eqref{eigen-eq} and $\lambda $ is an eigenvalue of $\tilde{\mathbf{L}}_{\alpha}$. Conversely, if $\lambda $ is an eigenvalue of $\tilde{\mathbf{L}}_{\alpha}$, $q = (\varphi, \lambda\varphi + y\partial_y \varphi)$ satisfies $\tilde{\mathbf{L}}_{\alpha} q = \lambda q$.
	
	Same as Proposition \ref{prop-lorentz-ss}, the Lorentz transformation also transforms the linearized equation at $U_{\alpha,\infty,\kappa}$ to the linearized equation at $V_{\alpha,\infty,\kappa}$.
	\begin{prop}[Transformation of the linearized operator at $U_{\alpha,\infty,\kappa}$]\label{prop-transform-linear-op} Consider $U(s,y) = U_{\alpha,\infty,\kappa} + \eta(s,y) $, where $\eta$ satisfies the linearized equation \eqref{linear-ss}, and introduce $\xi \equiv \mathcal{T}_\gamma \eta$ with $\gamma = \frac{1}{\sqrt{1+\alpha}}$, then $\xi$ satisfies \eqref{linearized-ss-transform}.
	\end{prop}
	\begin{proof}
		The proof follows from a direct computation. 
	\end{proof}

	\begin{cor}\label{cor-eigen}
		$\lambda$ is an eigenvalue of $\tilde{\mathbf{L}}_{\alpha}$ if and only if $\lambda$ is an eigenvalue of $\mathbf{L}'_\alpha$.
	\end{cor}
	\begin{proof}
		If $\lambda$ is an eigenvalue of $\tilde{\mathbf{L}}_{\alpha}$, then there exists a smooth function $\varphi\in  C^{\infty}[-1,1]$ such that $e^{\lambda s} \varphi(y)$ solves the linearized equation \eqref{linear-ss}. By the Lorentz transformation and Proposition \ref{prop-transform-linear-op}, we have
		$$\mathcal{T}_{\gamma}(e^{\lambda s} \varphi)(s',y') = e^{\lambda s'} \left( \frac{1-\gamma y'}{\sqrt{1-\gamma^2}}\right)^{-\lambda} \varphi\left(\frac{y'-\gamma}{1-\gamma y'}\right) := e^{\lambda s'} \psi(y') $$
		solves the linearized equation \eqref{linearized-ss-transform}, where $\gamma = \frac{1}{\sqrt{1+\alpha}}$. Thus, $\psi(y')\in  C^{\infty}[-1,1]$ satisfies \eqref{eigen-eq-transform} and $\lambda$ is also an eigenvalue of $\mathbf{L}'_\alpha$. Conversely, by the inverse Lorentz transformation, any eigenvalue of $\mathbf{L}'_\alpha$ is also an eigenvalue of $\tilde{\mathbf{L}}_{\alpha}$. 
	\end{proof}
	
	\begin{remark}
		According to the proof of Corollary \ref{cor-eigen}, the following transformation (induced by the Lorentz transformation $\mathcal{T}_{\gamma}$)
		\begin{align}\label{transform-eigen-eq}
			\psi(y') = \left( \frac{1-\gamma y'}{\sqrt{1-\gamma^2}}\right)^{-\lambda} \varphi\left(\frac{y'-\gamma}{1-\gamma y'}\right), \quad \gamma = \frac{1}{\sqrt{1+\alpha}}
		\end{align}
		actually gives the transformation from the eigen-equation \eqref{eigen-eq} to \eqref{eigen-eq-transform}.
	\end{remark}
	Thus, the mode stability of $U_{\alpha,\infty,\kappa}$ is equivalent to the mode stability of $V_{\alpha,\infty,\kappa}$. Instead of the Heun type ODE \eqref{eigen-eq}, the new eigen-equation \eqref{eigen-eq-transform} has only three singular points $\pm1,\infty$, which is known as the \textit{hypergeometric differential equation}, see \cite{olver2010nist}. Fortunately, the connection problem of the hypergeometric ODE can be solved which leads to a rigorous proof of the mode stability of  $V_{\alpha,\infty,\kappa}$.

	More precisely, recall that the standard form of the hypergeometric differential equation is given by
	$$ z'(1- z')\partial_{z'z'} \psi + \left( c - (a+b+1) z' \right)\partial_{z'} \psi -ab\psi =0,$$
	where $a,b,c$ are constants. Let $y' = 2z'-1$, Eq. \eqref{eigen-eq-transform} can be rewritten as the following standard hypergeometric ODE
	\begin{align}\label{eigen-eq-z'}
		z'(1- z')\partial_{z'z'} \psi + \left( (\lambda - \sqrt{1+\alpha}) - 2\lambda z' \right)\partial_{z'} \psi -\lambda(\lambda-1)\psi =0, 
	\end{align}
	where we take $a=\lambda, b=\lambda-1, c= \lambda - \sqrt{1+\alpha}.$ We shall prove that 
	
	\begin{prop}\label{prop-mode-stability} For all $\alpha>0$, there are no non-trivial smooth solutions $\psi \in C^{\infty}[-1,1]$ to Eq. \eqref{eigen-eq-z'} for $\operatorname{Re} \lambda >-1$, expect $0$ and $1$. As a corollary, $U_{\alpha,\infty,\kappa}$ is mode stable.
		% , i.e. there exist no eigenvalues of $\tilde{\mathbf{L}}_{\alpha}$ such that $\operatorname{Re} \lambda >-1$, expect $0$ and $1$.
	\end{prop}
	
	To prove Proposition \ref{prop-mode-stability}, we need the following standard Frobenius analysis.
	\begin{lem}[Frobenius theory \cite{donninger2024spectral,Teschl2012}]\label{lem:Frobenius}
		Consider
		\begin{align}\label{Fuchsian-eq}
			f^{\prime \prime}(z)+p(z) f^{\prime}(z)+q(z) f(z)=0
		\end{align}
		where $p$ and $q$ are given functions and $f$ is the unknown. Denote $\mathbb{D}_R:=$ $\{z \in \mathbb{C}:|z|<R\}$. Let $R>0$ and let $p, q: \mathbb{D}_R \backslash\{0\} \rightarrow \mathbb{C}$ be holomorphic. Suppose that the limits
		\begin{align*}
			p_0:=\lim _{z \rightarrow 0}[z p(z)], \quad q_0:=\lim _{z \rightarrow 0}\left[z^2 q(z)\right]
		\end{align*}
		exist and let $s_{ \pm} \in \mathbb{C}$ satisfy $P\left(s_{ \pm}\right)=0$, where
		\begin{align*}
			P(s):=s(s-1)+p_0 s+q_0
		\end{align*}
		is the indicial polynomial. Let $\operatorname{Re} s_{+} \geq \operatorname{Re} s_{-}$. Then there exists a holomorphic function $h_{+}: \mathbb{D}_R \rightarrow \mathbb{C}$ with $h_{+}(0)=1$ and such that $f: \mathbb{D}_R \backslash(-\infty, 0] \rightarrow \mathbb{C}$, given by $f(z)=z^{s_{+}} h_{+}(z)$, satisfies Eq. \eqref{Fuchsian-eq}. Furthermore, 
		\begin{itemize}
			\item if $s_{+}-s_{-} \notin \mathbb{N}_0$, there exists a holomorphic function $h_{-}$: $\mathbb{D}_R \rightarrow \mathbb{C}$ with $h_{-}(0)=1$ and such that $f(z)=z^{s_{-}} h_{-}(z)$ is another solution of Eq. \eqref{Fuchsian-eq} on $\mathbb{D}_R \backslash(-\infty, 0]$;
			\item if $s_{+}-s_{-} \in \mathbb{N}_0$, there exist a constant $c \in \mathbb{C}$ and a holomorphic function $h_{-}: \mathbb{D}_R \rightarrow \mathbb{C}$ with $h_{-}(0)=1$ such that 
			\begin{align*}
				f(z)=z^{s_{-}} h_{-}(z)+c z^{s_{+}} h_{+}(z) \log z
			\end{align*}
			is another solution of Eq. \eqref{Fuchsian-eq} on $\mathbb{D}_R \backslash(-\infty, 0]$. The constant $c$ might be zero unless $s_{+}=s_{-}$. 
		\end{itemize}
	\end{lem}

	By Frobenius theory, we can always find an infinite series \textit{local} solution of the form
	\begin{align}\label{series-exp-psi}
		\psi(z') = (z'-z'_0)^s \sum_{k=0}^\infty a_k (z'-z'_0)^k 
	\end{align}
	near each singular point $z'_0\in \{0,1\}$. More precisely, by substituting \eqref{series-exp-psi} into \eqref{eigen-eq-z'} and matching the lowest order term, we have the following indicial polynomials 
	\begin{align*}
		P'_0(s) = s(s-1+c) =0 \text{~for~} z'_0 =0, 
	\end{align*}
	and
	\begin{align*}
		P'_1(s) = s(s+a+b-c) =0 \text{~for~} z'_0 =1. 
	\end{align*}
	
	\paragraph{Frobenius analysis at $z'_0=0$} For $z'_0 =0$,
	$$s_+ = 1+\sqrt{1+\alpha} -\lambda,\quad s_-= 0, \quad \textit{~if~}  -1< \operatorname{Re}\lambda \le  1+\sqrt{1+\alpha} $$
	and 
	$$s_+ =0, \quad s_-= 1+\sqrt{1+\alpha} -\lambda, \quad \textit{~if~} \operatorname{Re} \lambda > 1+\sqrt{1+\alpha}. $$
	Therefore, we have
	\begin{enumerate}
		\item[Case 1] if $-1< \operatorname{Re}\lambda \le 1+\sqrt{1+\alpha}$ and $ s_+-s_-\notin \mathbb{N}_0$, the two independent local solutions are given by 
		\begin{align*}
			\psi_{01} = {z'}^{s_+} h_+(z') \quad \text{~and~} \quad \psi_{02} = h_-(z'),
		\end{align*}
		with
		$$h_+(z')={ }_2 F_1(a+1-c, b+1-c ; 2-c ; z'), \quad h_-(z') = { }_2 F_1(a, b ; c ; z'),$$
		where
		$$ { }_2 F_1(a, b ; c ; z')=\sum_{n=0}^{\infty} \frac{(a)_n(b)_n}{(c)_n} \frac{z'^n}{n!}=1+\frac{a b}{c} \frac{z'}{1!}+\frac{a(a+1) b(b+1)}{c(c+1)} \frac{z'^2}{2!}+\cdots$$
		is the hypergeometric function defined for all $|z'|<1$, see \cite{olver2010nist}.
		Note that $\psi_{01}$ is not smooth at $0$ since $s_+\notin \mathbb{N}_0$. Thus, any local smooth solution around $z_0'=0$ must be a multiple of $\psi_{02}$, thus is also analytic at $0$. 
		
		\item[Case 2] if $-1< \operatorname{Re}\lambda \le  1+\sqrt{1+\alpha}$ and $ s_+-s_-\in \mathbb{N}_0$, the two independent local solutions are given by 
		\begin{align*}
			\psi_{01} = z'^{s_+} h_+(z') \quad \text{~and~} \quad \psi_{02} = h_-(z') + c'z^{s_+} h_+(z') \log(z'),
		\end{align*}
		where 
		$$ h_+(z')={ }_2 F_1(a+1-c, b+1-c ; 2-c ; z'), $$
		and 
		$h_{-}$ is an analytic function around $0$ and $c'\in \mathbb{C}$ is a constant. Note that $z^{s_+} h_+(z') \log(z')$ is not a smooth function, thus a local smooth solution around $z_0'=0$ exists only if it is a multiple of $\psi_{01}$ or $c'=0$. In either case, the local smooth solution must also be analytic at 0.  
		
		%Note that in this case, $\lambda = 1+\sqrt{1+\alpha} - m \ge 0, m\in \mathbb{N}$ has only finite number of possible values. And we find that for such choices of $\lambda$, the constant $c'$ is always nonzero \rmk{(to be checked)}. Thus, there is also only one dimensional space of $H^{k_\alpha+2}$ local solutions around $z_0'=0$, i.e. $\{c_1\psi_{01} \mid c_1\in \mathbb{C}\}$.
		\item[Case 3] if $\operatorname{Re} \lambda > 1+\sqrt{1+\alpha}$,
		the two independent local solutions are given by 
		\begin{align*}
			\psi_{01} = h_+(z') \quad \text{~and~} \quad \psi_{02} =  z'^{s_-}h_-(z') + c' h_+(z') \log(z'),
		\end{align*}
		where
		$$h_+(z')= { }_2 F_1(a, b ; c ; z')$$
		and 
		$h_{-}$ is an analytic function around $0$ and $c'\in \mathbb{C}$ is a constant. Similarly, any local smooth solution around $z_0'=0$ must be also analytic at $0$. 
		%Since $\operatorname{Re} s_-<0$, $h_-$ does not belong to $H^{k_\alpha+2}(0,1)$ regardless the choice of $c'\in \mathbb{C}$. Thus, the one dimensional space of $H^{k_\alpha+2}$ local solutions around $z_0'=0$ is given by  $\{c_3\psi_{01}\mid c_3\in \mathbb{C}\}$
	\end{enumerate}
	
	\paragraph{Frobenius analysis at $z'_0=1$}
	For $z'_0 =1$,
	$$ s_+ =1 -\sqrt{1+\alpha} -\lambda, \quad s_{-} = 0, \quad \text{if~} -1< \operatorname{Re} \lambda \le 1 -\sqrt{1+\alpha}$$
	and
	$$ s_+ = 0,\quad  s_{-} = 1 -\sqrt{1+\alpha} -\lambda, \quad \text{if~} \operatorname{Re} \lambda > 1 -\sqrt{1+\alpha}. $$
	Therefore, we have
	\begin{itemize}
		\item[Case 1']  if  $-1< \operatorname{Re}\lambda \le 1-\sqrt{1+\alpha}$ and $ s_+ - s_{-}\notin \mathbb{N}_0$, the two independent local solutions are given by 
		\begin{align*}
			\psi_{11} = (z'-1)^{s_+} h_+(z'-1) \quad \text{~and~} \quad \psi_{12} = h_-(z'-1),
		\end{align*}
		where  
		$$h_-(z'-1)= { }_2 F_1(a, b ; 1+a+b-c ; 1-z')$$
		and 
		$h_{+}(z'-1)$ is an analytic function around $1$.
		Note that $\psi_{11}$ is not smooth at $1$ since $s_+\notin \mathbb{N}_0$. Thus, any local smooth solution must be a multiple of $\psi_{12}$, thus is also analytic at 1. 
		%And the set of $H^{k_\alpha+2}$ local solutions around $z_0=0$ is one dimensional, i.e. $\{ c_0\varphi_{02} \mid  c_0\in\mathbb{C}\}$.
		\item[Case 2'] if $-1< \operatorname{Re}\lambda \le 1-\sqrt{1+\alpha}$ and $ s_+ - s_{-}\in \mathbb{N}_0$, then $s_{+} = s_{-} =0$, since $s_{+} \in \mathbb{N}_0$ and $\operatorname{Re} s_+ <1$ for all $\operatorname{Re}\lambda >-1$. Thus, the two independent local solutions are given by 
		\begin{align*}
			\psi_{11} = h_+(z'-1) \quad \text{~and~} \quad \psi_{12} = h_-(z'-1) + c' h_+(z'-1) \log(z'-1).
		\end{align*}
		where
		$$h_+(z'-1)= { }_2 F_1(a, b ; 1+a+b-c ; 1-z'),$$
		and 
		$h_{-}(z'-1)$ is an analytic function around $1$ and the constant
		$c'\neq 0$ since $s_{+} = s_{-}$. Note that $ h_+(z'-1) \log(z'-1)$ is not smooth at $1$. Thus, any local smooth solution must be a multiple of $\psi_{11}$, and is also analytic at 1.    
		\item[Case 3'] if $\operatorname{Re} \lambda > 1-\sqrt{1+\alpha}$,
		the two independent local solutions are given by 
		\begin{align*}
			\psi_{11} = h_+(z-1) \quad \text{~and~} \quad \psi_{12} =  (z-1)^{s_-}h_-(z-1) + c' h_+(z-1) \log(z-1),
		\end{align*}
		where 
		$$h_+(z'-1)= { }_2 F_1(a, b ; 1+a+b-c ; 1-z'),$$
		and 
		$h_{-}$ is an analytic function around $0$ and $c'\in \mathbb{C}$ is a (possibly vanishing) constant. Since $\operatorname{Re} s_{-} = 1-\sqrt{1+\alpha} -\operatorname{Re}\lambda<0$, $\psi_{12}$ is not smooth at $1$ regardless the choice of $c'\in \mathbb{C}$. Thus, any local smooth solution around $z'_0=1$ is a multiple of $\psi_{11}$, and is also analytic at 1. 
	\end{itemize}
	%Now, we can reduce the existence of nontrivial smooth solutions $\psi \in C^{\infty}(0,1)\cap H^{k_\alpha+2}(0,1)$ to Eq. \eqref{eigen-eq-z'} with $\operatorname{Re} \lambda \ge -\frac{1}{2}$ to the connection problem of the one dimensional solution spaces at $z'_0=0$ and $z'_0=1$. 
	%More precisely, denote $\psi_0$ and $\psi_1$ the generator of the one dimensional solution spaces at $z'_0=0$ and $z'_0=1$ respectively, then $\operatorname{Re} \lambda \ge 0$ is an eigenvalue of $L'_\alpha$ if and only if the following Wronskian
	%\begin{align}\label{wronskian}
	%W\left[\psi_0, \psi_1\right](\lambda):=\partial_{z'}\psi_0(z') \psi_1(z')-  \psi_0(z') \partial_{z'}\psi_1(z')\end{align}is zero.
	
    Now we are ready to prove the mode stability of $U_{\alpha,\infty,\kappa}$. 

\begin{proof}[Proof of Proposition \ref{prop-mode-stability}]
	%to Proposition \ref{prop-1dsol-space} and the transformation \eqref{transform-eigen-eq}
	By the Frobenius analysis above, for any $\lambda \in \mathbb{C}$ with $\operatorname{Re}\lambda >-1$ and any nontrivial smooth solution $\psi(z')\in C^{\infty}[0,1]$ to \eqref{eigen-eq-z'},  $\psi(z')$ must be also analytic on $[0,1]$. We aim to show that $\psi(z')$ must fail to be analytic at $z'=0$ unless $\lambda \in \{0,1\}$.
	
	The Frobenius analysis at $z_0'=1$ above shows that in all three cases, any local smooth solution around $z_0'=1$ must be given by a multiple of 
	$$ { }_2 F_1(a, b ; 1+a+b-c ; 1-z').$$
	Therefore, without the loss of generality, we assume
	$$\psi(z') = { }_2 F_1(a, b ; 1+a+b-c ; 1-z'). $$
	Let $\tilde{z}=1-z'$, then $\psi(z') = \psi(1-\tilde{z}) \equiv \tilde{\psi}(\tilde{z})$ is an analytic function around $\tilde{z}=0$. Now the finite regular singular points of the hypergeometric ODE \eqref{eigen-eq-z'} are $\{0,1\}$. Thus, $\psi(z') = \tilde{\psi}(\tilde{z})$ fails to be analytic at $z'=0 $ (i.e. $\tilde{z} =1$) precisely when the radius of convergence of the series 
	$$ \tilde{\psi}(\tilde{z}) =  { }_2 F_1(a, b ; 1+a+b-c ; \tilde{z})$$
	is equal to one.
	
	By the definition of the hypergeometric function, the coefficients of the series expansion of ${ }_2 F_1(a, b ; 1+a+b-c ; \tilde{z})$ are
	$$ a_n(\lambda) = \frac{(a)_n(b)_n}{(1+a+b-c)_n n!} = \frac{(\lambda)_n(\lambda-1)_n}{(\lambda +\sqrt{1+\alpha})_n n!}, $$
	where $(a)_n = a(a+1)\cdots (a+n-1).$ Therefore, if $\lambda \notin \{0,1\}$ and $\operatorname{Re}\lambda > -1$, we have $a_n(\lambda) \neq 0$ for all $n\in \mathbb{N}$, and
	\begin{align*}
		r_n(\lambda):=\frac{a_{n+1}(\lambda)}{a_{n}(\lambda)} = \frac{(\lambda + n)(\lambda + n-1)}{(\lambda + \sqrt{1+\alpha} + n)(n+1)} \to 1, \quad \text{as~} n\to \infty.
	\end{align*}
	This implies that the radius of convergence of $\tilde{\psi}(\tilde{z})$ must be 1.
	
\end{proof}

\section{Spectral analysis of $\tilde{\mathbf{L}}_{\alpha}$}\label{sec:4}
In this section, we study the essential spectrum of $\tilde{\mathbf{L}}_{\alpha}$. The key idea is to decompose $\tilde{\mathbf{L}}_{\alpha}$ as a maximally dissipative operator plus a compact perturbation. To begin with, we present the functional setup. 

\subsection{Functional setup}
By \eqref{linear-ss-system}, the linearized equation can be written as
$$ \partial_s \mathbf{q} = \tilde{\mathbf{L}}_{\alpha} \mathbf{q}, \quad \tilde{\mathbf{L}}_{\alpha} \mathbf{q} =: \tilde{\mathbf{L}}_{\alpha} \begin{pmatrix}
	q_1\\ q_2
\end{pmatrix} =  \begin{pmatrix}
	-y\partial_y q_1 + q_2\\
	\partial_{yy}q_1 -\frac{2\alpha}{\sqrt{1+\alpha}+y} \partial_y q_1 -q_2-y\partial_y q_2
\end{pmatrix}.$$
We can further decompose the operator
$$ \tilde{\mathbf{L}}_{\alpha} = \tilde{\mathbf{L}} + \mathbf{L}_{\alpha,1},$$
where 
\begin{align*}
	\tilde{\mathbf{L}} \mathbf{q} := \begin{pmatrix}
		-y\partial_y q_1+ q_2 - q_1(-1)\\
		\partial_{yy} q_1 -q_2-y\partial_y q_2
	\end{pmatrix}
\end{align*}
is a modified \textit{free wave operator}, and 
\begin{align*}
	\mathbf{L}_{\alpha,1} \mathbf{q} = \begin{pmatrix}
		q_1(-1)\\
		-\frac{2\alpha}{\sqrt{1+\alpha}+y} \partial_y q_1
	\end{pmatrix}
\end{align*}
is called  \textit{potential}. We introduce the following sesquilinear form 
$$ \langle \mathbf{q}, \tilde{\mathbf{q}} \rangle_0 := \int_{-1}^1 \partial_y q_1 \overline{\partial_y \tilde{q}_1} dy + \int_{-1}^1  q_2 \overline{ \tilde{q}_2} dy + q_1(-1)\overline{\tilde{q}_1(-1)} $$
on the space $ C^1[-1,1]\times C^0[-1,1]$, and for $k\ge 1$
$$ \langle \mathbf{q}, \tilde{\mathbf{q}} \rangle_k := \int_{-1}^1 \partial^{k+1}_y q_1 \overline{\partial^{k+1}_y \tilde{q}_1} dy + \int_{-1}^1  \partial^k_y q_2 \overline{  \partial^k_y \tilde{q}_2} dy + \langle \mathbf{q}, \tilde{\mathbf{q}} \rangle_0  $$
on the space $ C^{k+1}[-1,1]\times C^k[-1,1]$. Define the corresponding norms
$$ \|\mathbf{q}\|_{k}^2 := \langle \mathbf{q}, \mathbf{q} \rangle_k.$$
We claim that the norm $\|\cdot\|_k$ is equivalent to the Sobolev norm $H^{k+1} \times H^{k}$.
\begin{lem}\label{lem-norm} For $k\ge 0$ and all $\mathbf{q}\in C^{k+1}[-1,1]\times C^k[-1,1]$, it holds
	\begin{align*}
		\| \mathbf{q} \|_{k} \simeq \|q_1\|_{H^{k+1}} + \|q_2\|_{H^k}.
	\end{align*}
	And the completion of $ C^{k+1}[-1,1]\times C^k[-1,1]$ with the norm $\|\cdot\|_k$ is the Hilbert space $H^{k+1}(-1,1) \times H^{k}(-1,1)$. 
\end{lem}
\begin{proof}
	By the definition of $\|\cdot\|_k$, it suffices to show that $$ \|\mathbf{q}\|_{0} \gtrsim \|q_1\|_{L^2} $$
	for all $\mathbf{q}\in C^{1}[-1,1]\times C^0[-1,1]$.
	This follows from 
	\begin{align*}
		|q_1(y)|^2 = \left|q_1(-1) + \int_{-1}^{y} \partial_z q_1 dz\right|^2 \lesssim |q_1(-1)|^2 +  \left| \int_{-1}^{y} \partial_z q_1 dz\right|^2
		\lesssim |q_1(-1)|^2 + \int_{-1}^{1} \left|\partial_y q_1\right|^2 dy
	\end{align*}
	and
	\begin{align*}
		\|q_1\|_{L^2}^2 = \int_{-1}^1 |q_1|^2 dy \lesssim |q_1(-1)|^2 + \int_{-1}^{1} \left|\partial_y q_1\right|^2 dy \lesssim \|\mathbf{q}\|_{0}^2.
	\end{align*}
\end{proof}

\subsection{Generation of the semigroup} For $\alpha>0$, let $\mathcal{D}(\tilde{\mathbf{L}}_{\alpha}) = \mathcal{D}(\tilde{\mathbf{L}}) =  C^{\infty}[-1,1]\times C^{\infty}[-1,1]$. Define $\mathcal{H}^k =H^{k+1}(-1,1) \times H^{k}(-1,1)$, then the free wave operator $\tilde{\mathbf{L}}: \mathcal{D}(\tilde{\mathbf{L}}) \subset \mathcal{H}^k \rightarrow  \mathcal{H}^k$ is a well-defined operator on $\mathcal{H}^k$ for all $k\ge 0$.

\begin{lem}[Dissipativity of the free operator]\label{lem-dissipativity} For $k\ge 0$ and $\mathbf{q}\in  C^{k+1}[-1,1]\times C^k[-1,1]$,
	\begin{align*}
		\operatorname{Re} \langle \tilde{\mathbf{L}} \mathbf{q}, \mathbf{q} \rangle_k \le -\frac{1}{2} \| \mathbf{q} \|_{k}^2.
	\end{align*}
\end{lem}
\begin{proof}
	By definition, for $k=0$, we have
	\begin{align*}
		\operatorname{Re} \left\langle  \tilde{\mathbf{L}} \mathbf{q}, \mathbf{q} \right\rangle_{0}
		& = \operatorname{Re} \int_{-1}^1 \partial_y \left(-y\partial_y q_1 + q_2\right) \overline{\partial_y q_1 } dy + \operatorname{Re}\left( \left(- y\partial_y q_1 + q_2 -q_1\right) \overline{ q_1 }|_{-1} \right)\\
		& \quad + \operatorname{Re} \int_{-1}^1  \left(- q_2 - y\partial_y q_2 +\partial_{yy} q_1\right) \overline{q_2} dy\\
		& = -\frac12 \left(\|\partial_y q_1\|_{L^2}^2 + \| q_2\|_{L^2}^2 \right) - \frac12|y\partial q_1 - q_2 |^2(\pm 1) - |q_1(-1)|^2 \\
		& \quad - \operatorname{Re}\left( \left( y\partial_y q_1 - q_2\right)(-1) \overline{ q_1 }(-1)\right)\\
		& = -\frac12 \left(\|\partial_y q_1\|_{L^2}^2 + \| q_2\|_{L^2}^2 \right) - \frac12|y\partial q_1 - q_2 |^2(1) - \frac{1}{2}|q_1(-1)|^2 \\
		& \quad -\frac12 \left| y\partial_y q_1 - q_2 +q_1\right|^2(-1)\\
		& \le -\frac{1}{2} \| \mathbf{q} \|_{0}^2.
	\end{align*}
	For $k\ge 1$, we have
	\begin{align*}
		\operatorname{Re} \left\langle \tilde{\mathbf{L}} \mathbf{q}, \mathbf{q} \right\rangle_{k}& = \operatorname{Re}\int_{-1}^1 \partial_y^{k+1} \left(- y\partial_y q_1 + q_2\right) \overline{\partial_y^{k+1} q_1 } dy \\
		& \quad + \operatorname{Re} \int_{-1}^1 \partial^{k} \left(- q_2 - y\partial_y q_2 +\partial_{yy} q_1\right) \overline{\partial^{k} q_2} dy + \operatorname{Re} \left\langle  \tilde{\mathbf{L}} \mathbf{q}, \mathbf{q} \right\rangle_{0}\\
		& = -\left(k+\frac12\right) \left(\|\partial^{k+1} q_1\|_{L^2}^2 + \|\partial^{k} q_2\|_{L^2}^2 \right) - \frac12|y\partial^{k+1} q_1 - \partial^{k} q_2 |^2(\pm 1) \\
		&\quad + \operatorname{Re} \left\langle  \tilde{\mathbf{L}} \mathbf{q}, \mathbf{q} \right\rangle_{0}\\
		&\le -\frac12\left(\|\partial^{k+1} q_1\|_{L^2}^2 + \|\partial^{k} q_2\|_{L^2}^2 +  \| \mathbf{q} \|_{0}^2 \right) = -\frac12 \|\mathbf{q}\|_k^2.
	\end{align*}
\end{proof}

Next, we show that the range of $\lambda - \tilde{\mathbf{L}}$ is dense in $\mathcal{H}^k$ for some $\lambda>-1$ (we simply take $\lambda =0$ for convenience) and all $k\ge 0$.
\begin{lem}\label{lem-dense}
	Let $k\ge0$ and $\mathbf{f} \in \mathcal{H}^k$. For any $\epsilon>0$, there exists a function $\mathbf{q}\in (C^{\infty}[-1,1])^2$ such that 
	$$ \|  - \tilde{\mathbf{L}} \mathbf{q} - \mathbf{f} \|_{k} \le \epsilon. $$
\end{lem}
\begin{proof}
	Since $(C^{\infty}[-1,1])^2$ is dense in $\mathcal{H}^k$, it suffices to prove that for all $\mathbf{f} \in (C^{\infty}[-1,1])^2$, there exists a unique smooth function $\mathbf{q}\in (C^{\infty}[-1,1])^2$ such that
	\begin{align}\label{eq:dense-L}
		- \tilde{\mathbf{L}} \mathbf{q} = \mathbf{f},
	\end{align}
	which is equivalent to
	\begin{align*}
		\left\{
		\begin{aligned}
			y\partial_y q_1 - q_2 + q_1(-1) = f_1\\
			q_2 +y\partial_y q_2 -\partial_{yy} q_1 = f_2. 
		\end{aligned} \right.
	\end{align*}
	Using the first equation to solve for $q_2$, we see that solving this system of ODEs reduces to 
	\begin{align}\label{eq:dense}
		2y\partial_y q_1 + (y^2-1)\partial_{yy} q_1 =  f_1 + y\partial_y f_1 + f_2 - q_1(-1).  
	\end{align}
	Denote $F = f_1 + y\partial_y f_1 + f_2 \in C^{\infty}[-1,1],$ then we can rewrite \eqref{eq:dense} as
	\begin{align*}
		\partial_y ((y^2-1)\partial_{y} q_1) = F(y) - q_1(-1).
	\end{align*}
	Thus, we have
	\begin{align}\label{eq:q_1}
		q_1(y) = q_1(-1) + \int_{-1}^y  \frac{\int_{-1}^{w} \left( F(z) -q_1(-1) \right) dz}{w^2-1} dw.
	\end{align}
	To get rid of the potential singularity at $y=1$, we have to choose
	\begin{align}\label{eq:q1-endpointvalue}
		q_1(-1) = \frac{1}{2} \int_{-1}^{1} F(z) dz.
	\end{align}
	Now we show that the function $q_1(y)$ given by \eqref{eq:q_1}-\eqref{eq:q1-endpointvalue} belongs to $C^{\infty}[-1,1]$. First, since $F\in C^{\infty}[-1,1]$ and $\frac{1}{w^2-1} \in C^{\infty}(-1,1)$, we have $q_1(y)\in C^{\infty}(-1,1)$. So it remains to prove $q_1$ is also smooth at endpoints $\pm1$. 
	
	Near $y=-1$, we have 
	\begin{align*}
		q_1(y) = q_1(-1) + \int_{-1}^y \frac{1}{w-1} \frac{\int_{-1}^{w} \left( F(z) -q_1(-1) \right) dz}{w+1} dw.
	\end{align*}
	Note that $\frac{1}{w-1}$ and $\frac{\int_{-1}^{w} \left( F(z) -q_1(-1) \right) dz}{w+1}$ are smooth at $-1$, thus $q_1(y)$ is also smooth at $-1$.
	
	Near $y=1$, by \eqref{eq:q1-endpointvalue}, we can rewrite \eqref{eq:q_1} as
	\begin{align*}
		q_1(y) &= q_1(-1) + \int_{-1}^y \frac{\int_{w}^{1} \left( F(z) -q_1(-1) \right) dz}{w^2-1} dw\\
		& = q_1(-1) + \int_{-1}^y \frac{1}{w+1}\frac{\int_{w}^{1} \left( F(z) -q_1(-1) \right) dz}{w-1} dw.
	\end{align*}
	Note that $\frac{1}{w+1}$ and $\frac{\int_{w}^{1} \left( F(z) -q_1(-1) \right) dz}{w-1}$ are smooth at $1$, $q_1(y)$ is also smooth at $1$, i.e. $q_1\in C^{\infty}[-1,1]$. Since $q_2 = y\partial_y q_1 + q_1(-1) - f_1$, $q_2 \in C^{\infty}[-1,1]$. So, we have $\mathbf{q}\in (C^{\infty}[-1,1])^2$.

\end{proof}

\begin{prop}\label{prop-L} For $k\ge 0$, the operator $\tilde{\mathbf{L}}: \mathcal{D}(\tilde{\mathbf{L}}) \subset \mathcal{H}^k \rightarrow \mathcal{H}^k$ is densely defined, closable, and its closure $\mathbf{L}$ generates a strongly continuous one-parameter semigroup $\mathbf{S}:[0, \infty) \rightarrow$ $\mathcal{L}(\mathcal{H}^k)$ which satisfies
	\begin{align*}
		\|\mathbf{S}(s)\|_{\mathcal{L}(\mathcal{H}^k)} \le C e^{-\frac{1}{2} s}
	\end{align*}
	for all $s \geq 0$, where $C>0$ is a constant.
\end{prop}
\begin{proof}
	Since $ \mathcal{D}(\tilde{\mathbf{L}})= C^{\infty}[-1,1]\times C^{\infty}[-1,1]$ is dense in $\mathcal{H}^k$, $\tilde{\mathbf{L}}$ is densely defined. By Lemma \ref{lem-dissipativity} and Lemma \ref{lem-dense}, we obtain that the operator $\tilde{\mathbf{L}}$  is a dissipative operator and $\operatorname{rg}(- \tilde{\mathbf{L}})$ is also dense in $\mathcal{H}^k$ for all $k\ge 0$. Therefore, the claim follows from the Lumer-Phillips Theorem \cite[p.83, Theorem 3.15]{engel2000one} and Lemma \ref{lem-norm}.
	
\end{proof}
%\subsection{Relatively compactness of the potential}
Next, we deal with the potential operator $\mathbf{L}_{\alpha,1}$. Compare to the power nonlinearities \cite{merle2007existence,donninger2016blowup}, the potential $\mathbf{L}_{\alpha,1}$ is no longer a compact operator. Even worse, $\mathbf{L}_{\alpha,1}$ is not a relatively compact operator with respect to $\mathbf{L}$, either. Nevertheless, we can prove that
\begin{prop}\label{prop-S-alpha}
	For all $\alpha>0$ and $k\ge 0$, the operator $\mathbf{L}_{\alpha,1}$ is bounded on $\mathcal{H}^k$. As a corollary, $\mathbf{L}_{\alpha}:= \mathbf{L} + \mathbf{L}_{\alpha,1}$ generates a strongly continuous one-parameter semigroup $\mathbf{S}_{\alpha}:[0, \infty) \rightarrow \mathcal{L}(\mathcal{H}^k)$ which satisfies
	\begin{align*}
		\|\mathbf{S}_{\alpha}(s)\|_{\mathcal{L}(\mathcal{H}^k)} \leq C e^{\left(-\frac{1}{2} + \|\mathbf{L}_{\alpha,1}\|\right) s}
	\end{align*}
	for all $s \geq 0$, where $C>0$ is a constant independent of $\alpha$.
	%, and relatively compact with respect to $\mathbf{L}$
\end{prop}
\begin{proof}
	By the definition of $\mathbf{L}_{\alpha,1}$, for any $\mathbf{q}\in \mathcal{H}^k$, we have
	\begin{align*}
		\|\mathbf{L}_{\alpha,1} \mathbf{q} \|_{\mathcal{H}^k} = \|q_1(-1)\|_{H^{k+1}} + \left\| -\frac{2\alpha}{\sqrt{1+\alpha}+y} \partial_y q_1\right\|_{H^k} \lesssim_{\alpha} \|q_1\|_{H^{k+1}} \lesssim_{\alpha} \|\mathbf{q}\|_{\mathcal{H}^k},  
	\end{align*}
	which implies that $\mathbf{L}_{\alpha,1}$ is bounded on $\mathcal{H}^k$. Thus, by the bounded perturbation theorem \cite[p. 158]{engel2000one}  $\mathbf{L}_{\alpha}$ also generates a strongly continuous one-parameter semigroup with the above growth bound.   
	%To prove the relatively compactness, it suffices to show that for any uniformly bounded sequence $\{\mathbf{q}_n\}_{n\in \mathbb{N}}\subset (C^\infty[-1,1])^2$ with respect to the graph norm of $\mathbf{L}$, i.e.
	%$$\sup_{n\in \mathbb{N}}\left( \|\mathbf{q}_n\|_{\mathcal{H}^k} +  \|\mathbf{L}\mathbf{q}_n\|_{\mathcal{H}^k}\right) <\infty,$$
	%there exists a convergent subsequence of $\{\mathbf{L}_{\alpha,1} \mathbf{q}_n\}$ in $\mathcal{H}^k$.
\end{proof}

\subsection{A new decomposition with a compact perturbation}\label{subsec:4.3}

Since $\mathbf{L}_{\alpha,1}$ is merely a bounded perturbation without any compactness compared to the free wave operator $\mathbf{L}$, we can not get any information of the essential spectrum of $\mathbf{L}_{\alpha}$, thus no uniform growth bound of the semigroup $\mathbf{S}_{\alpha}$ is available either. To overcome this difficulty, we shall introduce a new decomposition of the linearized operator $\mathbf{L}_{\alpha}$ as a new dissipative operator plus a compact perturbation. More precisely, recall that
\begin{align*}
	\mathbf{L}_{\alpha} := \begin{pmatrix}
		-y\partial_y & 1\\
		\partial_{yy} +2\partial_y U_{\alpha,\infty, \kappa} \partial_y & -1-y\partial_y
	\end{pmatrix},
\end{align*}
where the potential
$$ \partial_y U_{\alpha,\infty,\kappa}(s, y) = -\frac{\alpha}{\sqrt{1+\alpha} + y}.$$
Note that the dissipativity of $\mathbf{L}_{\alpha}$ itself is not guaranteed. Indeed, we have
\begin{align*}
	\operatorname{Re}\left\langle \mathbf{L}_{\alpha} \mathbf{q}, \mathbf{q}\right\rangle_{\dot{H}^1\times L^2} &\le  -\frac{1}{2}\left(\|\partial_y q_1\|_{L^2}^2+\|q_2\|_{L^2}^2\right) + \operatorname{Re} \int 2\partial_y U_{\alpha,\infty,\kappa} \partial_y q_1 \overline{q_2} dy \\
	&\le -\frac{1}{2}\left(\|\partial_y q_1\|_{L^2}^2+\|q_2\|_{L^2}^2\right) +  (\sqrt{1+\alpha}+1)  \left(\|\partial_y q_1\|_{L^2}^2+\|q_2\|_{L^2}^2\right)\\
	&\le (\sqrt{1+\alpha}+\frac{1}{2})  \left(\|\partial_y q_1\|_{L^2}^2+\|q_2\|_{L^2}^2\right),
\end{align*}
where we use
\begin{align*}
	\left\| \partial_y U_{\alpha,\infty,\kappa}\right\|_{L^\infty(-1,1)} \le  \left\| -\frac{\alpha}{\sqrt{1+\alpha}+y}\right\|_{L^\infty(-1,1)} = \frac{\alpha}{\sqrt{1+\alpha}-1} = \sqrt{1+\alpha}+1.
\end{align*}

To gain more dissipation, we need to consider $\mathbf{L}_{\alpha}$ in Sobolev spaces with higher regularity, i.e $\mathcal{H}^k = H^{k+1}(-1,1)\times H^k(-1,1), k\ge k_{\alpha}+2$. 
\begin{lem}[Commuting with derivatives]\label{lem-commu} For $\alpha>0$ and $ k \ge 0$, there holds
	\begin{align*}
		\partial^k \mathbf{L}_{\alpha} =\mathbf{L}_{\alpha,k} \partial^k  + \mathbf{L}_{\alpha,k}' 
	\end{align*}
	where
	\begin{align*}
		\mathbf{L}_{\alpha,k}=\begin{pmatrix}
			-k-y\partial_y & 1\\
			\partial_{yy} + 2 \partial_y U_{\alpha,\infty, \kappa} \partial_y & -1-k-y\partial_y
		\end{pmatrix},
	\end{align*}
	and $\mathbf{L}_{\alpha,k}'$ satisfies the pointwise bound
	\begin{align*}
		\left|\mathbf{L}_{\alpha,k}' \mathbf{q}\right| \lesssim\binom{0}{\sum_{j=0}^{ k-1} |\partial^{k-j+1} U_{\alpha,\infty, \kappa}|\left|\partial^{j+1} q_1\right|} \lesssim_{\alpha} \binom{0}{\sum_{j=0}^{ k-1} \left|\partial^{j+1} q_1\right|}
	\end{align*}
\end{lem}
\begin{proof}
	Direct computation yields the following formulae
	\begin{align*}
		\left[\partial^k, \partial_y U_{\alpha,\infty, \kappa} \partial_y \right]=\sum_{j=0}^{k-1} C_k^j \partial^{k-j+1} U_{\alpha,\infty, \kappa} \partial^{j+1}, \quad\left[\partial^k, y\partial_y\right]= k \partial^k .
	\end{align*}
	The pointwise bound follows from 
	\begin{align*}
		\|\partial^{k-j+1} U_{\alpha,\infty, \kappa}\|_{L^{\infty}(-1,1)} &= \left\|(k-j)! \frac{\alpha}{(\sqrt{1+\alpha} +y)^{k-j+1}}\right\|_{L^{\infty}(-1,1)} \\
		&\lesssim \frac{\alpha}{(\sqrt{1+\alpha}-1)^{k-j+1}}.
	\end{align*}
\end{proof}

For $k\ge 0$, we define a new inner product
\begin{align*}
	\langle\langle \Psi, \tilde{\Psi} \rangle\rangle_k = (\partial^k \Psi, \partial^k \tilde{\Psi})_{L^2}  + (\Psi,\tilde{\Psi})_{L^2} = \int_{-1}^1 \partial^k \Psi \overline{\partial^k \tilde{\Psi}} dy + \int_{-1}^1 \Psi \overline{\tilde{\Psi}} dy.
\end{align*}
By definition, the norm induced by the inner produce $\langle\langle \cdot, \cdot \rangle\rangle_k$ is also equivalent to the standard Sobolev norm $\|\cdot\|_{H^k(-1,1)}$. For tuples of function $\mathbf{q}=(q_1, q_2)$, we introduce the corresponding inner product
\begin{align*}
	\langle\langle \mathbf{q}, \tilde{\mathbf{q}} \rangle\rangle_k = \langle\langle q_1, \tilde{q}_1 \rangle\rangle_{k+1} +  \langle\langle q_2, \tilde{q}_2 \rangle\rangle_{k},
\end{align*}
and its associate norm is equivalent to $\|\cdot\|_{H^{k+1}\times H^k}$. To control $\mathbf{L}_{\alpha,k}'$, we prove 
\begin{lem}[Subcoercivity estimate]\label{lem-subcoer} For $k\ge 0$, there exist $\epsilon_n>0$ with $\lim _{n \rightarrow \infty} \epsilon_n= 0$ and $\{\Pi_i\}_{1\le i\le n} \subset H^{k+1}(-1,1), c_n>0$ such that for all $n \geq 0, \Psi \in H^{k+1}(-1,1)$,
	\begin{align}\label{eq:subcoer}
		\epsilon_n \langle\langle \Psi, \Psi\rangle\rangle_k \ge  \left\| \Psi\right\|^2_{H^{k}} - c_n \sum_{i=1}^n\left(\Psi, \Pi_i\right)_{L^2}^2  
	\end{align}
\end{lem}
\begin{proof}
	Given $T \in L^2(-1,1)$, the antilinear map $h \mapsto(T, h)$ is continuous on $H^k(-1,1)$ since
	\begin{align*}
		(\Psi, \Psi)_{L^2} \leq \langle\langle\Psi, \Psi\rangle\rangle_k.
	\end{align*}
	By Riesz representation theorem, there exists a unique $l(T) \in H^k(-1,1)$ such that
	\begin{align}\label{eq:def-l}
		\forall h \in H^k(-1,1), \quad\langle\langle l(T), h\rangle\rangle_k=(T, h)_{L^2}
	\end{align}
	and $l: L^2(-1,1) \rightarrow H^k(-1,1)$ is a bounded linear map. By Sobolev embedding theorem, the operator $\iota: H^k(-1,1)\to L^2(-1,1), \iota(\Psi) =\Psi$ is compact, thus the map $\iota \circ l: L^2(-1,1) \rightarrow L^2(-1,1)$ is also compact. Moreover, if $\Psi_i=l(T_i), i=1,2$, then
	\begin{align*}
		(l(T_1), T_2)_{L^2}=(\Psi_1, T_2)_{L^2}=\langle\langle \Psi_1, l(T_2)\rangle\rangle_k=\langle\langle\Psi_1, \Psi_2\rangle\rangle_k.
	\end{align*}
	Similarly,
	\begin{align*}
		(T_1,l(T_2))_{L^2}=\langle\langle \Psi_1, \Psi_2\rangle\rangle_k=\left(l\left(T_1\right), T_2\right)_{L^2}
	\end{align*}
	i.e. $\iota \circ l$ is self-adjoint on $L^2(-1,1)$. By \eqref{eq:def-l}, 
	$$(T, l(T))_{L^2} = \langle\langle l(T), l(T)\rangle\rangle_k\ge 0,$$
	$\iota \circ l$ is also positive definite. Thus, by spectral decomposition of positive definite compact operators, there exists an $L^2$-orthonormal eigenbasis $\left(\Pi_{n, i}\right)_{1 \leq i \leq I(n)}$ of $\iota \circ l$ with positive eigenvalues $\lambda_n \rightarrow 0$. The eigenvalue equation implies $\Pi_{n, i} \in H^k(-1,1)$. Let
	\begin{align*}
		\mathcal{A}_n=\left\{\Psi \in H^k(-1,1) \mid(\Psi, \Psi)_{L^2}=1,\left(\Psi, \Pi_{j, i}\right)_{L^2}=0,1 \leq i \leq I(j), 1 \leq j \leq n\right\}
	\end{align*}
	and consider the minimization problem
	\begin{align*}
		I_n=\inf _{\Psi \in \mathcal{A}_n}\langle\langle\Psi, \Psi\rangle\rangle_k,
	\end{align*}
	whose infimum is attained at $\Psi \in \mathcal{A}_n$ since $\iota \circ l$ is compact. Also, by a standard Lagrange multiplier argument,
	\begin{align*}
		\forall h \in H^k(-1,1), \quad\langle\langle\Psi, h\rangle\rangle_k=\sum_{j=1}^n \sum_{i=1}^{I(j)} \beta_{i, j}\left(\Pi_{j, i}, h\right)_{L^2}+\beta(\Psi, h)_{L^2} .
	\end{align*}
	Set $h=\Pi_{j, i}$ and since $\Pi_{j,i}$ is an eigenvector of $\iota \circ l$, we infer $\beta_{i,j}=0$ and in view of \eqref{eq:def-l}, $l(\Psi)=\beta^{-1} \Psi$. Together with the orthogonality conditions, $\beta^{-1} \leq \lambda_{n+1}$. Hence
	\begin{align}\label{I_n}
		I_n=\langle\langle\Psi, \Psi\rangle\rangle_k=\beta(\Psi, \Psi)_{L^2} \geq \frac{1}{\lambda_{n+1}}
	\end{align}
	For all $\epsilon>0, k \geq 0$, from Gagliardo-Nirenberg interpolation inequality together with Young's inequality, we have
	\begin{align*}
		\left\| \Psi\right\|_{H^k}^2 \leq \epsilon \left\|\Psi\right\|^2_{\dot{H}^{k+1}} +c_{\epsilon, k} \|\Psi\|_{L^2}^2.
	\end{align*}
	Together with \eqref{I_n}, we have that for all $\Psi$ satisfying orthogonality condition of $\mathcal{A}_n$,
	\begin{align*}
		\left\| \Psi\right\|_{H^k}^2 \leq (\epsilon+c_{\epsilon, k} \lambda_{n+1})\langle\langle\Psi, \Psi\rangle\rangle_k.
	\end{align*}
	Choosing $\epsilon_n \rightarrow 0$ such that $c_{\epsilon_n, k} \lambda_{n+1} \leq \epsilon_n$ yields \eqref{eq:subcoer}.
\end{proof}

Based on this, we are able to show dissipativity of $\mathbf{L}_{\alpha}$ upon a compact perturbation. 
\begin{prop}[Maximal Dissipativity]\label{prop-hat-L-alpha} For $\alpha >0$, $k \ge k_{\alpha}+2$ and $\epsilon>0$ sufficiently small, there exists $\left(\mathbf{\Pi}_{\alpha,i}\right)_{1 \leq i \leq N} \in \mathcal{H}^{k}$ such that for the finite rank projection operator
	\begin{align}\label{eq:proj}
		\hat{\mathbf{P}}_{\alpha}=\sum_{i=1}^N\left\langle\langle\cdot, \mathbf{\Pi}_{\alpha,i}\right\rangle\rangle_k \mathbf{\Pi}_{\alpha,i},
	\end{align}
	the modified operator
	\begin{align*}
		\hat{\mathbf{L}}_{\alpha}=\mathbf{L}_{\alpha}-\hat{\mathbf{P}}_{\alpha}
	\end{align*}
	is dissipative:
	\begin{align}\label{eq:dissipative}
		\begin{aligned}
			\forall \mathbf{q} \in \mathcal{D}(\mathbf{L}_{\alpha}), \quad \operatorname{Re}\langle\langle-\hat{\mathbf{L}}_{\alpha} \mathbf{q}, \mathbf{q}\rangle\rangle_k &\ge \left(k-\sqrt{1+\alpha}-\frac{1}{2}-\epsilon\right)\langle\langle\mathbf{q}, \mathbf{q}\rangle\rangle_{k}\\
			&\ge \left(\frac{3}{2}-\epsilon \right)\langle\langle\mathbf{q}, \mathbf{q}\rangle\rangle_{k},
		\end{aligned}
	\end{align}
	and is maximal:
	\begin{align*}
		\lambda - \hat{\mathbf{L}}_{\alpha} \text{~is surjective~} \text{ for some } \lambda>0.
	\end{align*}
	
\end{prop}

\begin{proof}
	We first prove \eqref{eq:dissipative} for all $\mathbf{q} \in (C^{\infty}[-1,1])^2$. For $\mathbf{q} \in (C^{\infty}[-1,1])^2$, by Lemma \ref{lem-commu} we have %\mathcal{D}(\mathbf{L}_{\alpha})
	\begin{align*}
		\langle\langle -\mathbf{L}_{\alpha} \mathbf{q}, \mathbf{q}\rangle\rangle_k & = (-\mathbf{L}_{\alpha} \mathbf{q}, \mathbf{q})_{\dot{H}^{k+1}\times \dot{H}^k} + ( -\mathbf{L}_{\alpha} \mathbf{q}, \mathbf{q})_{L^2\times L^2}\\
		& =(-(\partial^{k+1}\mathbf{L}_{\alpha} \mathbf{q})_1, \partial^{k+1}q_1)_{L^2}+  (-(\partial^{k}\mathbf{L}_{\alpha} \mathbf{q})_2, \partial^{k}q_2)_{L^2} + ( -\mathbf{L}_{\alpha} \mathbf{q}, \mathbf{q})_{L^2\times L^2}\\
		& =  (-(\mathbf{L}_{\alpha,k+1}\partial^{k+1} \mathbf{q})_1, \partial^{k+1}q_1)_{L^2} 
		+ (-(\mathbf{L}_{\alpha,k}\partial^{k} \mathbf{q})_2, \partial^{k}q_2)_{L^2}\\
		&\quad + (-(\mathbf{L}_{\alpha,k+1}'\mathbf{q})_1, \partial^{k+1}q_1)_{L^2}  + (-(\mathbf{L}_{\alpha,k}'\mathbf{q})_2, \partial^{k}q_2)_{L^2} + ( -\mathbf{L}_{\alpha} \mathbf{q}, \mathbf{q})_{L^2\times L^2}      
		%\langle\langle -\mathbf{L}_{\alpha} \mathbf{q}, \mathbf{q}\rangle\rangle_k
		%= \sum_{j=0}^k(-\partial^{j}\mathbf{L}_{\alpha} \mathbf{q}, \partial^{j}\mathbf{q})_{H^1\times L^2} = \sum_{j=0}^k(-\mathbf{L}_{\alpha,j}\partial^{j} \mathbf{q} + \mathbf{L}_{\alpha,j}'\mathbf{q}, \partial^{j}\mathbf{q})_{H^1\times L^2}\\
		%& = \sum_{j=0}^k 
	\end{align*} 
	For the first two terms, by integration by parts we derive
	\begin{align*}
		&\operatorname{Re}(-(\mathbf{L}_{\alpha,k+1}\partial^{k+1} \mathbf{q})_1, \partial^{k+1}q_1)_{L^2} 
		+ \operatorname{Re} (-(\mathbf{L}_{\alpha,k}\partial^{k} \mathbf{q})_2, \partial^{k}q_2)_{L^2}\\
		&\ge \left(k -\sqrt{1+\alpha}-\frac{1}{2}\right) \left(\|\partial^{k+1} q_1\|_{L^2}^2+\|\partial^k q_2\|_{L^2}^2 \right)
	\end{align*}
	For the last three terms, by Young's inequality we have
	\begin{align*}
		&\left|(-(\mathbf{L}_{\alpha,k+1}'\mathbf{q})_1, \partial^{k+1}q_1)_{L^2} \right| + \left|(-(\mathbf{L}_{\alpha,k}'\mathbf{q})_2, \partial^{k}q_2)_{L^2} \right|+ \left|( -\mathbf{L}_{\alpha} \mathbf{q}, \mathbf{q})_{L^2\times L^2} \right|\\
		& \le \frac{\epsilon}{2}\|\partial^k q_2\|_{L^2} + c_{\epsilon,k,\alpha} \left(\|q_1\|_{H^k} + \|q_2\|_{H^{k-1}}\right).
	\end{align*}
	Thus, we obtain
	\begin{align*}
		\operatorname{Re }\langle\langle -\mathbf{L}_{\alpha} \mathbf{q}, \mathbf{q}\rangle\rangle_k \ge \left(k -\sqrt{1+\alpha}-\frac{1}{2} -\frac{\epsilon}{2}\right) \left(\|\partial^{k+1} q_1\|_{L^2}^2+\|\partial^k q_2\|_{L^2}^2 \right) -c_{\epsilon,k,\alpha} \left(\|q_1\|_{H^k} + \|q_2\|_{H^{k-1}}\right).
	\end{align*}
	Using Lemma \ref{lem-subcoer} for $q_1$ and $q_2$ respectively, we derive
	\begin{align*}
		\operatorname{Re }\langle\langle -\mathbf{L}_{\alpha} \mathbf{q}, \mathbf{q}\rangle\rangle_k \ge& \left(k -\sqrt{1+\alpha}-\frac{1}{2} -\epsilon\right) \left(\|\partial^{k+1} q_1\|_{L^2}^2+\|\partial^k q_2\|_{L^2}^2 \right) \\
		& -c_{\epsilon,k,\alpha}' \left(\sum_{i=1}^{N_1}(q_1,\Pi^{(1)}_i)_{L^2}^2 + \sum_{i=1}^{N_2}(q_2,\Pi^{(2)}_i)_{L^2}^2\right)
	\end{align*}
	for some $(\Pi^{(1)}_i, \Pi^{(2)}_i) \in \mathcal{H}^k$. Note that linear functionals
	$$ \mathbf{q}\mapsto \sqrt{c_{\epsilon,k,\alpha}'}(q_1,\Pi^{(1)}_i)_{L^2} \text{~and~} \mathbf{q}\mapsto \sqrt{c_{\epsilon,k,\alpha}'} (q_2,\Pi^{(2)}_i)_{L^2}   $$
	are continuous on $\mathcal{H}^k$, thus by Riesz theorem, there exists $\mathbf{\Pi}_{\alpha,i} \in \mathcal{H}^k$ such that
	$$ \langle\langle \mathbf{q}, \mathbf{\Pi}_{\alpha,i}\rangle\rangle_k = \sqrt{c_{\epsilon,k,\alpha}'}(q_1,\Pi^{(1)}_i)_{L^2},\quad  1\le i\le N_1, $$
	$$ \langle\langle \mathbf{q}, \mathbf{\Pi}_{\alpha,i}\rangle\rangle_k = \sqrt{c_{\epsilon,k,\alpha}'} (q_2,\Pi^{(2)}_i)_{L^2}, \quad N_1+1\le i\le N_1+N_2:=N. $$
	Define the finite rank projection $\hat{\mathbf{P}}_{\alpha}$ as \eqref{eq:proj}, we obtain the dissipativity \eqref{eq:dissipative} for all $\mathbf{q} \in (C^{\infty}[-1,1])^2$. 
	
	Now we claim that \eqref{eq:dissipative} also holds for all 
	$\mathbf{q} \in \mathcal{D}(\mathbf{L}_{\alpha})$. Indeed, if $\mathbf{q} \in \mathcal{D}(\mathbf{L}_{\alpha})$, we have $\mathbf{q}\in\mathcal{H}^k$ and $\mathbf{L}_{\alpha}\mathbf{q}\in\mathcal{H}^k$,
	since by Proposition \ref{prop-S-alpha} $\mathbf{L}_{\alpha}$ is the generator of the semigroup $\mathbf{S}_{\alpha}$ on $\mathcal{H}^k$. Note that 
	$\mathbf{L}_{\alpha} = \mathbf{L} + \mathbf{L}_{\alpha,1}$ and $\mathbf{L}_{\alpha,1}$ is a bounded operator in $\mathcal{H}^k$, we obtain $\mathbf{L}\mathbf{q}\in\mathcal{H}^k$ as well. Since $(C^{\infty}[-1,1])^2$ is dense in $\mathcal{H}^k$, there exists a sequence $\{\mathbf{f}_n\}\subset (C^{\infty}[-1,1])^2$ such that 
	$$ \mathbf{f}_n \to  - \mathbf{L} \mathbf{q}.$$
	By the proof of Lemma \ref{lem-dense}, there exists a unique $\mathbf{q}_n \in (C^{\infty}[-1,1])^2 $ such that 
	$$ -\mathbf{L} \mathbf{q}_n = \mathbf{f}_n.$$
	We show that $\mathbf{q}_n \to \mathbf{q}$. To see that, by 
	Lemma \ref{lem-dissipativity} and $\mathbf{q}_n \in (C^{\infty}[-1,1])^2 $, we have
	\begin{align*}
		\| \mathbf{q}_n - \mathbf{q}_m \|_{k} \lesssim \|\mathbf{L}\mathbf{q}_n -\mathbf{L}\mathbf{q}_m \|_k = \|\mathbf{f}_n -\mathbf{f}_m \|_k \to 0, 
	\end{align*}
	i.e. $\{\mathbf{q}_n\}$ is a Cauchy sequence in $\mathcal{H}^k$, thus there exists a $\mathbf{q}' \in \mathcal{H}^k$ such that $\mathbf{q}_n \to \mathbf{q}'$. By the closedness of $\mathbf{L}$ we have $-\mathbf{L}\mathbf{q}' = -\mathbf{L} \mathbf{q}$. By Proposition \ref{prop-L}, $0\in \rho(\mathbf{L})$, thus $-\mathbf{L}$ is invertible. This implies that $\mathbf{q}' = \mathbf{q}$, so $\mathbf{q}_n \to \mathbf{q}$. Note that $\hat{\mathbf{L}}_{\alpha} = \mathbf{L}_{\alpha} - \hat{\mathbf{P}}_{\alpha} = \mathbf{L} + \mathbf{L}_{\alpha,1} - \hat{\mathbf{P}}_{\alpha}$ and $\mathbf{L}_{\alpha,1}, \hat{\mathbf{P}}_{\alpha}$ are bounded, we have $\hat{\mathbf{L}}_{\alpha}\mathbf{q}_n \to \hat{\mathbf{L}}_{\alpha}\mathbf{q}$. Since $\mathbf{q}_n$ satisfies \eqref{eq:dissipative}, let $n\to \infty$,  \eqref{eq:dissipative} also holds for  $\mathbf{q}$.
	
	It remains to prove that $\hat{\mathbf{L}}_{\alpha}$ is maximal. Note that by Proposition \ref{prop-S-alpha}, $\mathbf{L}_{\alpha}$ is the generator of a strongly continuous semigroup with a finite growth bound $w:= -\frac{1}{2} + \|\mathbf{L}_{\alpha,1}\|$, thus for $\lambda >0$ sufficiently large, $\lambda - \mathbf{L}_{\alpha}$ is invertible and
	$$ \|(\lambda -\mathbf{L}_{\alpha})^{-1}\|_{\mathcal{L}(\mathcal{H}^k)}\lesssim \frac{1}{\lambda -w} \lesssim \frac{2}{\lambda}.$$
	Then, since $\hat{\mathbf{P}}_{\alpha}$ is bounded,
	$$ \lambda - \hat{\mathbf{L}}_{\alpha} = \lambda -\mathbf{L}_{\alpha} + \hat{\mathbf{P}}_{\alpha} =(\lambda -\mathbf{L}_{\alpha})(\mathbf{I} +(\lambda -\mathbf{L}_{\alpha})^{-1} \hat{\mathbf{P}}_{\alpha} ) $$
	is also invertible, thus $\lambda -\hat{\mathbf{L}}_{\alpha}$ is surjective.
	
\end{proof}

\subsection{Spectral analysis of the generator}
Now we can give a sufficiently detailed description of the spectrum of $\mathbf{L}_\alpha$.
\begin{prop}\label{prop-spectrum} Let $\alpha>0$ and $k\ge k_{\alpha}+2$, we have
	\begin{align}\label{spectrum}
		\sigma\left(\mathbf{L}_\alpha\right) \subset\left\{z \in \mathbb{C}: \operatorname{Re} z \leq -1 \right\} \cup\{0,1\}
	\end{align}
	and $\{0,1\} \subset \sigma_p\left(\mathbf{L}_{\alpha}\right)$. Furthermore, the geometric eigenspaces of eigenvalues 1 and 0 are spanned by the functions $\mathbf{f}_{1,\alpha}$ and $\mathbf{f}_{0,\alpha}$, respectively, where
	\begin{align*}
		\mathbf{f}_{1,\alpha}(y)  =\begin{pmatrix}
			\frac{ \alpha\sqrt{1+\alpha}}{\sqrt{1+\alpha}+y} \\
			\frac{ \alpha(1+\alpha)}{(\sqrt{1+\alpha}+y)^2}
		\end{pmatrix} \text{~and~} \quad \mathbf{f}_{0,\alpha}(y)  =\binom{1}{0}
	\end{align*} 
	Moreover, there is a generalized eigenfunction $\mathbf{g}_{0,\alpha}$ of the eigenvalue 0
	\begin{align*}
		\mathbf{g}_{0,\alpha} = \begin{pmatrix}
			-\log(\sqrt{1+\alpha}+y) -\frac{\alpha}{2\sqrt{1+\alpha}}\frac{1}{\sqrt{1+\alpha}+y}\\
			1 -\frac{y}{\sqrt{1+\alpha}+y} +\frac{\alpha}{2\sqrt{1+\alpha}}\frac{y}{(\sqrt{1+\alpha}+y)^2}
		\end{pmatrix},
	\end{align*}
	satisfying 
	$$ \mathbf{L}_\alpha \mathbf{g}_{0,\alpha} = \mathbf{f}_{0,\alpha}, \quad  \mathbf{L}_\alpha^2 \mathbf{g}_{0,\alpha}=0.  $$
\end{prop}
\begin{proof}
	By Proposition \ref{prop-hat-L-alpha}, $\hat{\mathbf{L}}_{\alpha}$ is maximally dissipative, thus $\sigma(\hat{\mathbf{L}}_{\alpha}) \subset \{z\in\mathbb{C}: \operatorname{Re}z \le -1\}$ for $k\ge k_{\alpha}+2$. Since $\mathbf{L}_{\alpha} = \hat{\mathbf{L}}_{\alpha} + \hat{\mathbf{P}}_{\alpha}$ and $\hat{\mathbf{P}}_{\alpha}$ is compact, we have
	$\sigma_{ess}(\mathbf{L}_{\alpha}) = \sigma_{ess}(\hat{\mathbf{L}}_{\alpha}) \subset \sigma(\hat{\mathbf{L}}_{\alpha}).$ To prove \eqref{spectrum}, it remains to show that all eigenvalues of $\mathbf{L}_{\alpha}$ in  $\{z\in\mathbb{C}: \operatorname{Re}z > -1\}$ must belong to $\{0,1\}$, which has already been verified by the mode stability of $\tilde{\mathbf{L}}_{\alpha}$, see Proposition \ref{prop-mode-stability} and Proposition \ref{prop-1dsol-space} in the Appendix.
	
	Moreover, a direct computation shows that $\mathbf{f}_{0,\alpha},\mathbf{f}_{1,\alpha}$ are eigenfunctions and $\mathbf{g}_{0,\alpha}$ is a generalized eigenfunction of 0. By the Frobenius analysis in Proposition \ref{prop-1dsol-space}, the local smooth solutions around $1$ is always one-dimensional, thus the geometric eigenspaces of eigenvalues 0 and 1 are spanned by $\mathbf{f}_{0,\alpha},\mathbf{f}_{1,\alpha}$ respectively.
\end{proof}
Now we define the usual Riesz projections associated to the eigenvalues 0 and 1 of $\mathbf{L}_\alpha$. We set
\begin{align*}
	\begin{aligned}
		\mathbf{P}_{0,\alpha} & :=\frac{1}{2 \pi \mathrm{i}} \int_{\gamma_0} \mathbf{R}_{\mathbf{L}_\alpha}(z) d z \\
		\mathbf{P}_{1,\alpha} & :=\frac{1}{2 \pi \mathrm{i}} \int_{\gamma_1} \mathbf{R}_{\mathbf{L}_\alpha}(z) d z
	\end{aligned}
\end{align*}
with the curves $\gamma_j:[0,1] \rightarrow \mathbb{C}, j \in\{0,1\}$, given by
\begin{align*}
	\gamma_0(s)=\frac{1}{4} e^{2 \pi \mathrm{i} s}, \quad \gamma_1(s)=1+\frac{1}{2} e^{2 \pi \mathrm{i} s}
\end{align*}

The following important result shows that the algebraic multiplicity of the eigenvalues 1 equals its geometric multiplicity, while the algebraic multiplicity of the eigenvalues 0 is strictly larger than the geometric multiplicity. This is different from the power nonlinear case \cite{merle2007existence,ostermann2024stable}, and the generalized 0-mode would induce a polynomial growth in time.
\begin{lem}\label{lem:rank}
	The projections $\mathbf{P}_{1,\alpha}$ and $\mathbf{P}_{0,\alpha}$ have rank 1 and 2, respectively.
\end{lem}

\begin{proof}
	First of all we note that $\mathbf{P}_{1,\alpha}$ and $\mathbf{P}_{0,\alpha}$ have finite rank. This is because the eigenvalues 0 and 1 are generated by a compact perturbation and an eigenvalue of infinite algebraic multiplicity is invariant under such a perturbation (see \cite[p. 239, Theorem 5.28 and p. 244, Theorem 5.35]{kato1995}). By definition, $\mathbf{P}_{1,\alpha}$ and $\mathbf{P}_{0,\alpha}$ depends continuously on the parameter $\alpha$. By \cite[p. 34, Lemma 4.10]{kato1995}, the ranks of $\mathbf{P}_{1,\alpha}$ and   $\mathbf{P}_{0,\alpha}$ are constants for all $\alpha>0$. Thus, it suffices to prove the conclusion for some $\alpha>0$. For convenience, we take $\alpha=3$.  
	
	Next, observe that $\operatorname{rg} \mathbf{P}_{0,\alpha} \subset \mathcal{D}(\mathbf{L}_{\alpha})$. To see this, let $\mathbf{v} \in \operatorname{rg} \mathbf{P}_{0,\alpha}$. By the density of $\mathcal{D}(\mathbf{L}_{\alpha})$ in $\mathcal{H}^k$ we find $\left(\mathbf{u}_n\right) \subset \mathcal{D}(\mathbf{L}_{\alpha})$ with $\mathbf{u}_n \to \mathbf{v}$. Since $\mathbf{P}_{0,\alpha} \mathcal{D}(\mathbf{L}_{\alpha}) \subset \mathcal{D}(\mathbf{L}_{\alpha})$ (\cite[p. 178, Theorem 6.17]{kato1995}) we see that $\left(\mathbf{v}_n\right):=\left(\mathbf{P}_{0,\alpha} \mathbf{u}_n\right) \subset \mathcal{D}(\mathbf{L}_{\alpha}) \cap \operatorname{rg} \mathbf{P}_{0,\alpha}$ and $\mathbf{v}_n=\mathbf{P}_{0,\alpha} \mathbf{u}_n \to \mathbf{P}_{0,\alpha} \mathbf{v}=\mathbf{v}$ by the boundedness of $\mathbf{P}_{0,\alpha}$. The operator $\mathbf{L}_{\alpha} \mid_{ \mathcal{D}(\mathbf{L}_{\alpha}) \cap \mathrm{rg} \mathbf{P}_{0,\alpha}}$ is bounded and this implies $\mathbf{L}_{\alpha} \mathbf{v}_n \to \mathbf{f}$ for some $\mathbf{f} \in \operatorname{rg} \mathbf{P}_{0,\alpha}$. Thus, the closedness of $\mathbf{L}_{\alpha}$ implies $\mathbf{v} \in \mathcal{D}(\mathbf{L}_{\alpha})$ and we obtain $\operatorname{rg} \mathbf{P}_{0,\alpha} \subset \mathcal{D}(\mathbf{L}_{\alpha})$ as claimed. Let
	$\mathbf{A}:=\mathbf{L}_\alpha |_{\operatorname{rg} \mathbf{P}_{0,\alpha}}$, then $\mathbf{A}$ is a bounded operator on the finite-dimensional Hilbert space $\operatorname{rg} \mathbf{P}_{0,\alpha}$ with spectrum equal to $\{0\}$ (\cite[p. 178, Theorem 6.17]{kato1995}). This implies that $\mathbf{A}$ is nilpotent, i.e., there exists a (minimal) $n \in \mathbb{N}$ such that $\mathbf{A}^n=\mathbf{0}$. We claim that $n=2$ for $\alpha=3$.

	%By Proposition \ref{prop-spectrum}, $\mathbf{L}_{\alpha}$ has one eigenfunction $\mathbf{f}_{0,\alpha}$ and one 

	If $n=1$ we have $\operatorname{rg} \mathbf{P}_{0,3}=\operatorname{ker} \mathbf{A}$ and any element of $\operatorname{rg} \mathbf{P}_{0,3}$ is an eigenvector of $\mathbf{A}$ (hence $\mathbf{L}_3$) with eigenvalue 0. From Proposition \ref{prop-spectrum} we infer $\operatorname{rg} \mathbf{P}_{0,3}=\left\langle \mathbf{f}_{0,3}\right\rangle$. Since $\mathbf{P}_{0,3}\mathbf{L}_{3} \subset \mathbf{L}_{3}\mathbf{P}_{0,3}$ (\cite[p. 178, Theorem 6.17]{kato1995}), we have $\mathbf{L}_{3}\mathbf{P}_{0,3} \mathbf{g}_{0,3} =\mathbf{P}_{0,3}\mathbf{L}_{3} \mathbf{g}_{0,3} =\mathbf{P}_{0,3}\mathbf{f}_{0,3} =\mathbf{f}_{0,3} \neq 0$, this contradicts $\operatorname{rg} \mathbf{P}_{0,3}=\left\langle \mathbf{f}_{0,3}\right\rangle$, thus $n\ge 2$. Now suppose $n > 2$, then there exists a $\mathbf{u} \in \operatorname{rg} \mathbf{P}_{0,3} \subset \mathcal{D}(\mathbf{L}_{3})$ such that $\mathbf{A}^2 \mathbf{u}$ is a nontrivial element of $\operatorname{ker} \mathbf{A} \subset \operatorname{ker} \mathbf{L}_{3} = \langle\mathbf{f}_{0,3}\rangle$, i.e. $\mathbf{L}_{3}^2 \mathbf{u}=c_1\mathbf{f}_{0,3}$ for some constant $c_1 \neq 0$. Without loss of generality, we take $c_1=1$ for simplicity. Since $\mathbf{L}_{3} \mathbf{g}_{0,3} = \mathbf{f}_{0,3}$, we have $\mathbf{L}_{3}(\mathbf{L}_{3} \mathbf{u} - \mathbf{g}_{0,3}) =0$, i.e $\mathbf{L}_{3} \mathbf{u} - \mathbf{g}_{0,3} = c_2 \mathbf{f}_{0,3}$ for some constant $c_2\in \mathbb{R}$, thus
	$\mathbf{L}_{3}(\mathbf{u} - c_2\mathbf{g}_{0,3}) =  \mathbf{g}_{0,3}$. However, there are no solutions $\mathbf{v}\in \mathcal{D}(\mathbf{L}_{3})$ to the equation $\mathbf{L}_{3} \mathbf{v} =  \mathbf{g}_{0,3},$ 
	see Lemma \ref{lem:generalized-eigenfunction} in the Appendix. A similar argument shows that $\mathbf{P}_{1,3}$ has rank 1.
\end{proof}
%\rmk{to be chekced}

\subsection{The resolvent estimate}\label{subsec:resolvent}
For $0<w_0<1$, we define
\begin{align*}
	\Omega_{m, n}:=\left\{z \in \mathbb{C}: \operatorname{Re} z \in\left[-w_0, m\right], \operatorname{Im} z \in\left[-n, n\right]\right\}
\end{align*}
and
\begin{align*}
	\Omega_{m,n}^{\prime}:=\left\{z \in \mathbb{C}: \operatorname{Re} z \geq-w_0\right\} \backslash \Omega_{m,n}.
\end{align*}
The next estimate confines possible unstable eigenvalues of $\mathbf{L}_\alpha$ to a compact domain $\Omega_{m, n}$.
\begin{prop}\label{prop:resolvent-esti}
	Let $\alpha_0>0$, $k\ge k_{\alpha_0}+3$, $0<w_0<1$, and $0<\epsilon < \min\{1,\frac{\alpha_0}{2}\}$,
	there exist $m, n>0$ such that $\Omega_{m, n}^{\prime} \subset \rho\left(\mathbf{L}_\alpha\right)$ and
	\begin{align}\label{eq:resolvent-esti}
		\sup_{\lambda \in \Omega_{m, n}^{\prime}}\left\|\mathbf{R}_{\mathbf{L}_\alpha}(\lambda)\right\|_{\mathcal{L}(\mathcal{H}^k)} \leq C
	\end{align}
	for all $\alpha\in \overline{B_{\epsilon}}(\alpha_0)$, where $C>0$ is a constant independent of $\alpha$.
\end{prop}

\begin{proof}
	For $0<\epsilon < \min\{1,\frac{\alpha_0}{2}\}$ and $\alpha\in \overline{B_{\epsilon}}(\alpha_0)$, we have $\alpha\ge \frac{\alpha_0}{2}>0$ and $\alpha<\alpha_0+1$, which implies that  $k_{\alpha} \le k_{\alpha_0}+1$. Thus, $k\ge k_{\alpha_0}+3 \ge k_{\alpha}+2$ holds for all $\alpha\in \overline{B_{\epsilon}}(\alpha_0)$. For any $m, n>0$, by Proposition \ref{prop-spectrum} the condition $\lambda \in \Omega_{m, n}^{\prime}$ implies $\lambda \in \rho(\mathbf{L}_{\alpha})$. To prove the resolvent estimate \eqref{eq:resolvent-esti}, it suffices to show
	\begin{align}\label{coersive}
		\left|\left( (\lambda - \mathbf{L}_\alpha) \mathbf{q}, \mathbf{q} \right)_{\dot{H}^{k+1}\times \dot{H}^{k}}\right| + |\langle (\lambda - \mathbf{L}_\alpha) \mathbf{q}, \mathbf{q} \rangle_{0}| \gtrsim \|\mathbf{q}\|_{\mathcal{H}^k}^2 
	\end{align} 
	for all $\lambda \in \Omega_{m, n}^{\prime}$, $\alpha\in \overline{B_{\epsilon}}(\alpha_0)$ and $\mathbf{q} \in \mathcal{D}(\mathbf{L}_\alpha)$. Indeed, if \eqref{coersive} holds, then by Cauchy's inequality we have
	$$ \|\mathbf{q}\|_{\mathcal{H}^k} \lesssim \|(\lambda - \mathbf{L}_\alpha) \mathbf{q}\|_{\mathcal{H}^k}, $$
	which implies the resolvent estimate. Now we prove \eqref{coersive}. As in the proof of Proposition \ref{prop-hat-L-alpha}, it suffices to prove \eqref{coersive} for all $\mathbf{q} \in (C^{\infty}[-1,1])^2$. First, we divide the domain $\Omega_{m, n}^{\prime} = S_1\cup S_2$, where
	$$S_1 = \left\{z \in \mathbb{C}: \operatorname{Re} z \geq-w_0, |\operatorname{Im} z|\ge n\right\} $$
	and 
	$$ S_2 = \left\{z \in \mathbb{C}: \operatorname{Re} z \ge m, \operatorname{Im} z\in [-n, n] \right\}.$$
	Note that 
	\begin{align*}
		&\left( (\lambda - \mathbf{L}_\alpha) \mathbf{q}, \mathbf{q} \right)_{\dot{H}^{k+1}\times \dot{H}^{k}}\\
		& = \int_{-1}^1 \partial^{k+1} \left(\lambda q_1 + y\partial_y q_1 - q_2\right) \overline{\partial^{k+1} q_1 } dy \\
		& \quad + \int_{-1}^1 \partial^{k} \left((\lambda+1) q_2 + y\partial_y q_2 + \frac{2\alpha}{\sqrt{1+\alpha}+y} \partial_y q_1 -\partial_{yy} q_1\right) \overline{\partial^{k} q_2} dy\\
		& = \left(\frac12+k+\operatorname{Re}\lambda\right) \left(\|\partial^{k+1} q_1\|_{L^2} + \|\partial^{k} q_2\|_{L^2} \right) + \frac12|y\partial^{k+1} q_1 - \partial^{k} q_2 |^2(\pm 1) \\
		&\quad + \int_{-1}^1 \frac{2\alpha}{\sqrt{1+\alpha}+y} \partial^{k+1} q_1 \overline{\partial^{k} q_2} dy + \sum_{j=1}^{k}C_{k}^j \partial^{j}\left(\frac{2\alpha}{\sqrt{1+\alpha}+y}  \right)\partial^{k+1-j} q_1 \overline{\partial^{k} q_2} dy\\
		&\quad + i\operatorname{Im}\lambda \left(\|\partial^{k+1} q_1\|_{L^2} + \|\partial^{k} q_2\|_{L^2} \right)\\
		& \quad + \frac12 \int_{-1}^1 \left( y\partial^{2+k} q_1 \overline{\partial^{k}q_1} - y\overline{\partial^{2+k} q_1} \partial^{k}q_1 \right) dy +  \frac12 \int_{-1}^1 \left( y\partial^{k+1} q_2 \overline{\partial^{k}q_2} - y\overline{\partial^{k+1} q_2} \partial^{k}q_2 \right) dy\\
		& \quad +\int_{-1}^1 \left( \partial^{k+1} q_1 \overline{\partial^{k+1} q_2} -  \overline{\partial^{k+1} q_1}\partial^{k+1} q_2 \right)dy  + \frac12 \left( \partial^{k+1} q_1 \overline{\partial^{k} q_2}- \overline{\partial^{k+1} q_1}\partial^{k} q_2 \right)\left|_{-1}^{1}\right..
	\end{align*}
	In the above equation, the first two terms on the right hand side are real, while the terms in the last three lines are all pure imaginary. For the two terms in the second line, we have 
	\begin{align*}
		\left|\int_{-1}^1 \frac{2\alpha}{\sqrt{1+\alpha}+y} \partial^{k+1} q_1 \overline{\partial^{k} q_2} dy \right| &\le \frac{\alpha}{\sqrt{1+\alpha}-1}\left(\|\partial^{k+1} q_1\|_{L^2} + \|\partial^{k} q_2\|_{L^2} \right)\\
		&= \left(\sqrt{1+\alpha}+1\right)\left(\|\partial^{k+1} q_1\|_{L^2} + \|\partial^{k} q_2\|_{L^2} \right),
	\end{align*}
	and 
	\begin{align*}
		\left|\sum_{j=1}^{k}C_{k}^j \partial^{j}\left(\frac{2\alpha}{\sqrt{1+\alpha}+y}  \right)\partial^{k+1-j} q_1 \overline{\partial^{k} q_2} dy \right| \le M\|\partial q_1\|_{L^2} + \epsilon' \left(  \|\partial^{k+1} q_1\|_{L^2}+ \|\partial^{k} q_2\|_{L^2}\right)
	\end{align*}
	for all $\alpha\in \overline{B_{\epsilon}}(\alpha_0)$,
	where the constant $M=M(\alpha_0,\epsilon')>0$ and $\epsilon'>0$ is sufficiently small.  The constant $M$ is independent of $\alpha$ follows from the following estimates
	$$ \left|\partial^{j}\left(\frac{2\alpha}{\sqrt{1+\alpha}+y}\right)\right| \lesssim \frac{2\alpha}{(\sqrt{1+\alpha} -1)^{j+1}} = \frac{2(\sqrt{1+\alpha} +1)^{j+1}}{\alpha^{j}} \lesssim \frac{2^{j+1}(\sqrt{1+\alpha_0} +2)^{j+1}}{\alpha_0^{j}}.$$
	Since  $ k\ge k_{\alpha} +2$, we have
	\begin{align*}
		\left|\left( (\lambda - \mathbf{L}_\alpha) \mathbf{q}, \mathbf{q} \right)_{\dot{H}^{k+1}\times \dot{H}^{k}}\right| + M\|\partial q_1\|_{L^2} \ge \left(\frac32+\operatorname{Re}\lambda -\epsilon' \right) \left(\|\partial^{k+1} q_1\|_{L^2} + \|\partial^{k} q_2\|_{L^2} \right).
	\end{align*}
	Similarly, for the second term in \eqref{coersive}, we have 
	\begin{align*}
		&\left\langle (\lambda - \mathbf{L}_\alpha) \mathbf{q}, \mathbf{q} \right\rangle_{0}\\
		& = \int_{-1}^1 \partial \left(\lambda q_1 + y\partial_y q_1 - q_2\right) \overline{\partial q_1 } dy + \left(\lambda q_1 + y\partial_y q_1 - q_2\right) \overline{ q_1 }|_{-1} \\
		& \quad + \int_{-1}^1  \left((\lambda+1) q_2 + y\partial_y q_2 + \frac{2\alpha}{\sqrt{1+\alpha}+y} \partial_y q_1 -\partial_{yy} q_1\right) \overline{q_2} dy\\
		& = \left(\frac12+\operatorname{Re}\lambda\right) \left(\|\partial q_1\|_{L^2}^2 + \| q_2\|_{L^2}^2 \right) + \frac12|y\partial q_1 - q_2 |^2(\pm 1) + \operatorname{Re}\lambda |q_1(-1)|^2 \\
		&\quad + \int_{-1}^1 \frac{2\alpha}{\sqrt{1+\alpha}+y} \partial q_1 \overline{ q_2} dy + \left( y\partial_y q_1 - q_2\right) \overline{ q_1 }|_{-1} \\
		&\quad + i\operatorname{Im}\lambda \left(\|\partial q_1\|_{L^2}^2 + \| q_2\|_{L^2}^2 + |q_1(-1)|^2\right) \\
		& \quad + \frac12 \int_{-1}^1 \left( y\partial^{2} q_1 \overline{q_1} - y\overline{\partial^{2} q_1} q_1 \right) dy +  \frac12 \int_{-1}^1 \left( y\partial q_2 \overline{q_2} - y\overline{\partial q_2} q_2 \right) dy\\
		& \quad +\int_{-1}^1 \left( \partial q_1 \overline{\partial q_2} -  \overline{\partial q_1}\partial q_2 \right)dy  + \frac12 \left( \partial q_1 \overline{ q_2}- \overline{\partial q_1} q_2 \right)\left|_{-1}^{1}\right..
	\end{align*}
	In the above equation, the terms in the first line are all real, and the terms in the last three lines are all pure imaginary. For the terms in the second line, we have
	\begin{align*}
		\left|\int_{-1}^1 \frac{2\alpha}{\sqrt{1+\alpha}+y} \partial q_1 \overline{ q_2} dy \right| \le \left(\sqrt{1+\alpha} +1\right)\left(\|\partial q_1\|_{L^2}^2 + \| q_2\|_{L^2}^2 \right)
	\end{align*}
	and 
	\begin{align*}
		\left|\left( y\partial_y q_1 - q_2\right) \overline{ q_1 }|_{-1}\right|&\le (\|\partial q_1\|_{L^\infty} + \|q_2\|_{L^\infty}) \|q_1\|_{L^\infty}\lesssim (\|\partial q_1\|_{H^1} + \|q_2\|_{H^1}) \|q_1\|_{H^1}\\
		& \le \epsilon' \left(  \|\partial^{k+1} q_1\|_{L^2}+ \|\partial^{k} q_2\|_{L^2}\right) + M_0 \left(\|\partial q_1\|_{L^2}^2 + \| q_2\|_{L^2}^2 \right).
	\end{align*}
	For the terms in the last two lines, we have
	\begin{align*}
		&\left|  \frac12 \int_{-1}^1 \left( y\partial^{2} q_1 \overline{q_1} - y\overline{\partial^{2} q_1} q_1 \right) dy \right| \le \|\partial^2 q_1\|_{L^2}\|q_1\|_{L^2} \le \epsilon' \|\partial^{k+1} q_1\|_{L^2}^2 + M_1 \left(\|\partial q_1\|_{L^2}^2 + |q_1(-1)|^2\right),\\
		&\left| \frac12 \int_{-1}^1 \left( y\partial q_2 \overline{q_2} - y\overline{\partial q_2} q_2 \right) dy \right| \le \|\partial q_2\|_{L^2}\|q_2\|_{L^2} \le \epsilon' \|\partial^{k} q_2\|_{L^2}^2 + M_2 \|q_2\|_{L^2}^2,\\
		& \left|\int_{-1}^1 \left( \partial q_1 \overline{\partial q_2} -  \overline{\partial q_1}\partial q_2 \right)dy\right|\le 2\|\partial q_1\|_{L^2}\|\partial q_2\|_{L^2} \le \epsilon' \|\partial^{k} q_2\|_{L^2}^2 + M_3 \|\partial q_1\|_{L^2}^2,\\
		& \begin{aligned}
			\left|\frac12 \left( \partial q_1 \overline{ q_2}- \overline{\partial q_1} q_2 \right)\left|_{-1}^{1}\right.\right|&\le 2|\partial q_1|_{L^\infty}|q_2|_{L^\infty}\lesssim  |\partial q_1|_{H^1}|q_2|_{H^1} \\
			&\le\epsilon' \left(  \|\partial^{k+1} q_1\|_{L^2}+ \|\partial^{k} q_2\|_{L^2}\right) + M_4 \left(\|\partial q_1\|_{L^2}^2 + \| q_2\|_{L^2}^2 \right),
		\end{aligned}
	\end{align*}
	where constants $M_i>0, i=0,\dots,4$ are independent of $\alpha$. Thus, let $M' = \sum_{i=0}^4 M_i$, we obtain
	\begin{align*}
		\left| \left\langle (\lambda - \mathbf{L}_\alpha) \mathbf{q}, \mathbf{q} \right\rangle_{0} \right| +4\epsilon'\left(  \|\partial^{k+1} q_1\|_{L^2}+ \|\partial^{k} q_2\|_{L^2}\right)  \ge \left( \operatorname{Im}\lambda -\sqrt{1+\alpha}-1-M'
		\right)\left\langle \mathbf{q}, \mathbf{q} \right\rangle_{0}
	\end{align*}
	Combine the two estimates, we have 
	\begin{align*}
		&\left|\left( (\lambda - \mathbf{L}_\alpha) \mathbf{q}, \mathbf{q} \right)_{\dot{H}^{k+1}\times \dot{H}^{k}}\right| + |\langle (\lambda - \mathbf{L}_\alpha) \mathbf{q}, \mathbf{q} \rangle_{0}|\\
		&\ge \left( \operatorname{Im}\lambda -\sqrt{1+\alpha}-1-M'-M
		\right)\left\langle \mathbf{q}, \mathbf{q} \right\rangle_{0} + \left(\frac32+\operatorname{Re}\lambda -5\epsilon' \right) \left(\|\partial^{k+1} q_1\|_{L^2} + \|\partial^{k} q_2\|_{L^2} \right)\\
		& \ge \frac12 \|\mathbf{q}\|^2
	\end{align*}
	when $\epsilon'$ is sufficiently small, $\operatorname{Im}\lambda$ is sufficiently large and $\operatorname{Re}\lambda \ge -w_0$, i.e \eqref{coersive} holds for $\lambda \in S_1$ if we take $n$ large enough.
	
	It remains to prove \eqref{coersive} for $\lambda\in S_2$. In this case, by Proposition \ref{prop-S-alpha} and Hille-Yosida theorem for the strongly continuous semigroup $\mathbf{S}_{\alpha}(s)$ and its generator $\mathbf{L}_{\alpha}$, we have
	$$ \left\|\mathbf{R}_{\mathbf{L}_\alpha}(\lambda)\right\|_{\mathcal{L}(\mathcal{H}^k)} \le  \frac{C}{\operatorname{Re}\lambda-w} \le \frac{C}{m-w}, $$
	where $ w:= -\frac{1}{2} + \|\mathbf{L}_{\alpha,1}\|_{\mathcal{L}(\mathcal{H}^k)}$ controls the growth bound of $\mathbf{S}_{\alpha}(s)$, $C>0$ is a constant independent of $\alpha$. Note that by definition
	$$\|\mathbf{L}_{\alpha,1}\|_{\mathcal{L}(\mathcal{H}^k)} \lesssim \left|\partial^{k}\left(\frac{2\alpha}{\sqrt{1+\alpha}+y}\right)\right| \lesssim \frac{2^{k+1}(\sqrt{1+\alpha_0} +2)^{k+1}}{\alpha_0^{k}} $$
	for all  $\alpha\in \overline{B_{\epsilon}}(\alpha_0)$.
	Take $m$ large enough, we finish the proof of \eqref{coersive}.
	
	%Thus, we may use the identity $\lambda-\mathbf{L}_\alpha=\left[1-\mathbf{L}_{C,1} \mathbf{R}_{\mathbf{L}}(\lambda)\right](\lambda-$ $\mathbf{L})$ to relate $\mathbf{R}_{\mathbf{L}_\alpha}(\lambda)$ to the free resolvent. Note that $\lambda \in \Omega_{m, n}^{\prime}$ implies $|\lambda|$ large if $\delta_0$ and $\omega_0$ are large. Consequently, in view of the Neumann series it suffices to prove that
	%\begin{align*}
	%\left\|\mathbf{L}_{C,1}\mathbf{R}_{\mathbf{L}}(\lambda)\right\| \lesssim|\lambda|^{-1}
	%\end{align*}
	%for all $\lambda \in \Omega_{10,10}^{\prime}$. This can be proved by using the boundedness of $\mathbf{L}_{C,1}$ and $\left\|\mathbf{R}_{\mathbf{L}}(\lambda)\right\| \leq \frac{1}{\operatorname{Re} \lambda+\frac{1}{2}}$ by Proposition xx. \rmk{This is not enough! here $\operatorname{Re} \lambda +1/2 \ll |\lambda|$}
\end{proof}

\section{Growth bound of the linearized flow}\label{sec:5}
% along the stable and unstable subspace
The above spectral analysis leads to a sufficiently complete description of the linearized evolution.
\begin{prop}\label{prop-growth-bound} Let $\alpha_0>0$, $k\ge k_{\alpha_0}+3$, $0<w_0<1$, $0<\epsilon < \min\{1,\frac{\alpha_0}{2}\}$, and $\alpha\in \overline{B_{\epsilon}}(\alpha_0)$. The following properties hold:
	\begin{align*}
		\left[\mathbf{S}_\alpha(s), \mathbf{P}_{1,\alpha}\right]=\left[\mathbf{S}_\alpha(s), \mathbf{P}_{0,\alpha}\right]=\mathbf{0}
	\end{align*}
	for all $s \geq 0$ such that
	\begin{align*}
		\begin{aligned}
			\mathbf{S}_\alpha(s) \mathbf{P}_{1,\alpha} & =e^s \mathbf{P}_{1,\alpha}, \\
			\mathbf{S}_\alpha(s) \mathbf{P}_{0,\alpha} & =\mathbf{P}_{0,\alpha} + s\mathbf{L}_{\alpha}\mathbf{P}_{0,\alpha},\\
			\left\|\mathbf{S}_\alpha(s) \tilde{\mathbf{P}}_\alpha \mathbf{q}\right\|_{\mathcal{H}^k(B)} & \le M e^{-w_0 s}\left\|\tilde{\mathbf{P}}_\alpha \mathbf{q}\right\|_{\mathcal{H}^k(B)}
		\end{aligned}
	\end{align*}
	for all $s \geq 0, \mathbf{q} \in \mathcal{H}^k$, and $\alpha\in \overline{B_{\epsilon}}(\alpha_0)$, where $\tilde{\mathbf{P}}_\alpha:=\mathbf{I}-\mathbf{P}_{1,\alpha}- \mathbf{P}_{0,\alpha}$, $M = M(\alpha_0, w_0)$ is a  constant depending only on $\alpha_0,w_0$. Furthermore, we have
	\begin{align*}
		\begin{aligned}
			\operatorname{rg} \mathbf{P}_{1,\alpha} & =\left\langle\mathbf{f}_{1,\alpha}\right\rangle \\
			\operatorname{rg} \mathbf{P}_{0,\alpha} & =\left\langle\mathbf{f}_{0, \alpha}, \mathbf{g}_{0, \alpha}\right\rangle,
		\end{aligned}
	\end{align*}
	where $\mathbf{f}_{1,\alpha} \in \mathcal{H}^k$ is an eigenfunction of $\mathbf{L}_\alpha$ with eigenvalue 1,  $\mathbf{f}_{0,\alpha} \in \mathcal{H}^k$ is an eigenfunction of $\mathbf{L}_\alpha$ with eigenvalues 0, and $\mathbf{g}_{0,\alpha}\in \mathcal{H}^k$ is a generalized eigenfunction of $\mathbf{L}_\alpha$ with eigenvalues 0 satisfying $\mathbf{L}_\alpha \mathbf{g}_{0,\alpha} = \mathbf{f}_{0,\alpha}$. Finally, $\mathbf{P}_{0,\alpha} \mathbf{P}_{1,\alpha}=$ $\mathbf{P}_{1,\alpha} \mathbf{P}_{0,\alpha}=\mathbf{0}$.
\end{prop}
\begin{proof}
	Any strongly continuous semigroup commutes with its generator and hence also with the resolvent of its generator. This implies that $\mathbf{S}_{\alpha}$ commutes with the spectral projections $\mathbf{P}_{\lambda,\alpha}, \lambda \in \{0,1\}$. By Proposition \ref{prop-spectrum} and Lemma \ref{lem:rank}, we have 
	$$\operatorname{rg}\mathbf{P}_{1,\alpha} = \langle \mathbf{f}_{1,\alpha}\rangle, \quad \operatorname{rg} \mathbf{P}_{0,\alpha}  =\left\langle\mathbf{f}_{0, \alpha}, \mathbf{g}_{0, \alpha}\right\rangle.$$ 
	Thus, we derive
	\begin{align*}
		\partial_s \mathbf{S}_\alpha(s) \mathbf{P}_{1,\alpha} \mathbf{q}= \mathbf{L}_{\alpha}\mathbf{S}_\alpha(s) \mathbf{P}_{1,\alpha} \mathbf{q} =\mathbf{S}_\alpha(s) \mathbf{L}_{\alpha} \mathbf{P}_{1,\alpha} \mathbf{q} = \mathbf{S}_\alpha(s) \mathbf{P}_{1,\alpha} \mathbf{q},
	\end{align*}
	and
	\begin{align*}
		&\partial_s \mathbf{S}_\alpha(s) \mathbf{P}_{0,\alpha} \mathbf{q}= \mathbf{L}_{\alpha}\mathbf{S}_\alpha(s) \mathbf{P}_{0,\alpha} \mathbf{q} =\mathbf{S}_\alpha(s) \mathbf{L}_{\alpha} \mathbf{P}_{0,\alpha} \mathbf{q},\\
		&\partial_s \mathbf{S}_\alpha(s) \mathbf{L}_{\alpha} \mathbf{P}_{0,\alpha} \mathbf{q} =\mathbf{S}_\alpha(s) \mathbf{L}^2_{\alpha} \mathbf{P}_{0,\alpha} \mathbf{q}=0,
	\end{align*}
	which yields 
	\begin{align*}
		&\mathbf{S}_\alpha(s) \mathbf{P}_{1,\alpha} \mathbf{q} = e^{s} \mathbf{P}_{1,\alpha} \mathbf{q},\\
		&\mathbf{S}_\alpha(s) \mathbf{P}_{0,\alpha} \mathbf{q} =  s\mathbf{L}_{\alpha}\mathbf{P}_{0,\alpha} \mathbf{q} +  \mathbf{P}_{0,\alpha} \mathbf{q}.
	\end{align*}
	
	It remains to prove the uniform growth bound of $\mathbf{S}_\alpha(s)|_{\operatorname{rg} (\tilde{\mathbf{P}}_\alpha)}$.   By \cite[p. 178, Theorem 6.17]{kato1995}, $\sigma(\mathbf{L}_{\alpha}|_{\operatorname{rg}(\tilde{\mathbf{P}}_\alpha)}) \subset \left\{z \in \mathbb{C}: \operatorname{Re} z \leq -1\right\}$ and the resolvent $\mathbf{R}_{\mathbf{L}_{\alpha}|_{\operatorname{rg}(\tilde{\mathbf{P}}_\alpha)}} (\lambda)$ is analytic in $\Omega_{m, n}$ for all $m, n>0$. By Proposition \ref{prop:resolvent-esti}, we have
	\begin{align*}    	\left\|\mathbf{R}_{\mathbf{L}_{\alpha}|_{\operatorname{rg}(\tilde{\mathbf{P}}_\alpha)}}(\lambda)\right\| \leq C
	\end{align*}
	for all $\lambda \in \left\{z \in \mathbb{C}: \operatorname{Re} z \ge -w_0\right\}$ and $\alpha\in \overline{B_{\epsilon}}(\alpha_0)$, where the constant $C>0$ is independent of $\alpha$. Then, the uniform growth bound follows from an adaption of the Gearhart-Prüss-Greiner Theorem (\cite[p. 302, Theorem 1.11]{engel2000one}), see also \cite[Appendix A]{ostermann2024stable}.
	%Since $\mathbf{L}_{\alpha}$ and $\mathbf{S}_{\alpha}$ commute with $ \tilde{\mathbf{P}}_\alpha$, the generator of $\mathbf{S}_\alpha(s) \tilde{\mathbf{P}}_\alpha$ is given by  $\mathbf{L}_\alpha \tilde{\mathbf{P}}_\alpha$. 
	
	%Note that by definition  $\tilde{P}_{\alpha}$ gives a decomposition of 
\end{proof}

\section{Nonlinear stability}\label{sec:6}
Now we derive the nonlinear stability. The full nonlinear equation can be written as 
\begin{align}\label{eq:q-nonlinear}
	\left\{\begin{aligned}
		&\partial_s \mathbf{q} = \mathbf{L}_\alpha \mathbf{q} + \mathbf{N}(\mathbf{q}),\\
		&\mathbf{q}(0,\cdot)  = \mathbf{q}_0 
	\end{aligned}    \right.
\end{align}
where 
$$  \mathbf{N}(\mathbf{q}) := \begin{pmatrix}
	0\\
	(\partial_y q_1)^2
\end{pmatrix}. $$
By Duhamel's principle, we have 
\begin{align*}
	\mathbf{q}(s,\cdot) = \mathbf{S}_\alpha(s)\mathbf{q}(0,\cdot) + \int_0^{s} \mathbf{S}_\alpha(s-\tau) N(\mathbf{q}(\tau)) d\tau.
\end{align*}

%\subsection{Properties of the nonlinearity}

\subsection{Stabilized nonlinear evolution}
\begin{defn} Let $k\ge k_{\alpha}+2$, $0< w_0<1$ be arbitrary but fixed. We define a Banach space $(\mathcal{X}^k(B), \|\cdot\|_{\mathcal{X}^k(B)})$ by
	\begin{align*}
		\mathcal{X}^k(B) &:=\left\{ \mathbf{q} \in C([0,\infty), \mathcal{H}^k(B)) \mid \|\mathbf{q}(s)\|_{\mathcal{H}^k(B)} \lesssim e^{-w_0s} \text{~for all~} s\ge 0 \right\}\\
		\|\mathbf{q}\|_{\mathcal{X}^k(B)} &:= \sup_{s\ge 0} \left( e^{w_0s}\|\mathbf{q}\|_{\mathcal{H}^k(B)} \right)
	\end{align*}
\end{defn}
We also define for $\epsilon>0$ closed balls
$$ \mathcal{H}_{\epsilon}^k(B) = \{ \mathbf{f}\in  \mathcal{H}^k(B) \mid   \|\mathbf{f}\|_{\mathcal{H}^k(B)} \le \epsilon \}$$
and 
$$ \mathcal{X}_{\epsilon}^k(B) = \{ \mathbf{q}\in  \mathcal{X}^k(B) \mid   \|\mathbf{q}\|_{\mathcal{X}^k(B)} \le \epsilon \}. $$

To avoid the unstable modes, we first consider a modified problem following Lyapunov-Perron method from the dynamical systems theory. Set $\mathbf{P}_{\alpha}:= \mathbf{P}_{0,\alpha}+ \mathbf{P}_{1,\alpha} = \mathbf{I}- \tilde{\mathbf{P}}_{\alpha}$, we introduce the following correction term 
\begin{align*}
	\mathbf{C}_{\alpha}(\mathbf{f},\mathbf{q}) =& \mathbf{P}_{\alpha}\mathbf{f} +\mathbf{P}_{0,\alpha}\int_0^{\infty} \mathbf{N}(\mathbf{q}(\tau)) d\tau\\
	&+\mathbf{L}_{\alpha}\mathbf{P}_{0,\alpha}\int_0^{\infty} (-\tau)\mathbf{N}(\mathbf{q}(\tau)) d\tau + \mathbf{P}_{1,\alpha} \int_0^{\infty} e^{-\tau} \mathbf{N}(\mathbf{q}(\tau)) d\tau
\end{align*}

\begin{lem}\label{lem-nonlinearity} For all $k\ge 2$, 
	\begin{align*}
		\| \mathbf{N}(\mathbf{q})\|_{\mathcal{H}^k}& \lesssim \|\mathbf{q}\|_{\mathcal{H}^k}^2,\\
		\|\mathbf{N}(\mathbf{q}) - \mathbf{N}(\mathbf{q}')\|_{\mathcal{H}^k}& \lesssim (\|\mathbf{q}\|_{\mathcal{H}^k} + \|\mathbf{q}'\|_{\mathcal{H}^k})\|\mathbf{q} - \mathbf{q}'\|_{\mathcal{H}^k}.
	\end{align*}
\end{lem}
\begin{proof}
	These estimates follow from the Cauchy inequality and Gagliardo-Nirenberg inequality.
\end{proof}

\begin{prop}\label{prop-modified-existence}
	Let $\alpha_0>0, k\ge k_{\alpha_0}+3$, $0< w_0<1$. There are constants $\epsilon_0>0$ and $C_0>1$ (depending on $\alpha_0, k$ and $w_0$),
	such that for all $\alpha \in \overline{B_{\epsilon_0}}(\alpha_0)$ and $\mathbf{f}\in \mathcal{H}^k(B)$ with $\|\mathbf{f}\|_{\mathcal{H}^k(B)}\le \frac{\epsilon_0}{C_0}$, there is a unique solution $\mathbf{q}_{\alpha}\in \mathcal{X}^k(B)$ satisfying $\|\mathbf{q}_{\alpha}\|_{\mathcal{X}^k(B)}\le \epsilon_0$ and 
	\begin{align}\label{eq:modified-system}
		\mathbf{q}_{\alpha}(s) = \mathbf{S}_{\alpha}(s)(\mathbf{f} -\mathbf{C}_{\alpha}(\mathbf{f},\mathbf{q}_{\alpha}) ) + \int_0^{s} \mathbf{S}_{\alpha}(s-\tau) \mathbf{N}(\mathbf{q}_{\alpha}(\tau)) d\tau
	\end{align}
	for all $s\ge 0$. Moreover, the data-to-solution map 
	$$ \mathcal{H}^k_{\frac{\epsilon_0}{C_0}}(B) \to \mathcal{X}^k(B), \quad \mathbf{f}\mapsto \mathbf{q}_{\alpha}, $$
	is Lipschitz continuous.
\end{prop}
\begin{proof}
	By the definition of $k_{\alpha}$, $k_{\alpha}\le k_{\alpha_0}+1$ holds for all $\alpha \in \overline{B_{\epsilon}}(\alpha_0)$ with $0<\epsilon< \min\{1,\frac{\alpha_0}{2}\}$. Then, if $k\ge k_{\alpha_0}+3$, we have $k\ge k_{\alpha}+2$, thus the semigroup $S_{\alpha}(s)$ is well-defined and satisfies Proposition \ref{prop-growth-bound}. Denote the operator 
	$$ \mathbf{K}_{\alpha}(\mathbf{f}, \mathbf{q})(s) :=  \mathbf{S}_{\alpha}(s)(\mathbf{f} -\mathbf{C}_{\alpha}(\mathbf{f},\mathbf{q}) ) + \int_0^{s} \mathbf{S}_{\alpha}(s-\tau) \mathbf{N}(\mathbf{q}(\tau)) d\tau,$$
	we claim that there exist $0<\epsilon_0< \min\{1,\frac{\alpha_0}{2}\}$  sufficiently small and $C_0>1$ such that for  all $\alpha \in \overline{B_{\epsilon_0}}(\alpha_0)$,   $\mathbf{f}\in \mathcal{H}^k(B)$ with $\|\mathbf{f}\|_{\mathcal{H}^k(B)}\le \frac{\epsilon_0}{C_0}$, $\mathbf{K}_{\alpha}(\mathbf{f}, \cdot) : \mathcal{X}_{\epsilon_0}^k(B)\to \mathcal{X}_{\epsilon_0}^k(B) $ is a contraction operator, which implies the existence and uniqueness of a solution $\mathbf{q}_{\alpha}$ to \eqref{eq:modified-system}  by the Banach fixed-point theorem.
	
	First, for any $\mathbf{f}\in \mathcal{H}^k(B)$ and $\mathbf{q}\in \mathcal{X}^k(B)$, by the definition of $ \mathbf{C}_{\alpha}(\mathbf{f},\mathbf{q})$ we derive
	\begin{align}\label{eq:K}
		\begin{aligned}
			\mathbf{K}_{\alpha}(\mathbf{f}, \mathbf{q})(s) &= \mathbf{S}_{\alpha}(s)\tilde{\mathbf{P}}_{\alpha}\mathbf{f} + \int_0^{s} \mathbf{S}_{\alpha}(s-\tau) \tilde{\mathbf{P}}_{\alpha}\mathbf{N}(\mathbf{q}(\tau)) d\tau -\mathbf{P}_{0,\alpha}\int_s^{\infty} \mathbf{N}(\mathbf{q}(\tau)) d\tau \\
			&\quad  -  \mathbf{L}_{\alpha}\mathbf{P}_{0,\alpha}\int_s^{\infty} (s-\tau)\mathbf{N}(\mathbf{q}(\tau)) d\tau - \mathbf{P}_{1,\alpha} \int_s^{\infty} e^{s-\tau} \mathbf{N}(\mathbf{q}(\tau)) d\tau.
		\end{aligned}
	\end{align}
	Then, by Proposition \ref{prop-growth-bound} and Lemma \ref{lem-nonlinearity}, we have
	\begin{align*}
		\|	\mathbf{K}_{\alpha}(\mathbf{f}, \mathbf{q}) \|_{ \mathcal{H}^k(B)}& \lesssim e^{-w_0s} \|\mathbf{f} \|_{ \mathcal{H}^k(B)} + \int_0^s e^{-w_0(s-\tau)} \|\mathbf{q}(\tau)\|^2_{ \mathcal{H}^k(B)}  d\tau + \int_{s}^{\infty}  \|\mathbf{q}(\tau)\|^2_{ \mathcal{H}^k(B)}  d\tau \\
		&\quad + \int_{s}^{\infty} (\tau-s)  \|\mathbf{q}(\tau)\|^2_{ \mathcal{H}^k(B)}  d\tau + \int_{s}^{\infty} e^{s-\tau}  \|\mathbf{q}(\tau)\|^2_{ \mathcal{H}^k(B)}  d\tau \\
		&\lesssim e^{-w_0s}\frac{\epsilon_0}{C_0} + {\epsilon_0}^2\int_0^s e^{-w_0(s-\tau)}  e^{-2w_0\tau}  d\tau + {\epsilon_0}^2 \int_{s}^{\infty}  e^{-2w_0\tau}  d\tau \\
		&\quad + {\epsilon_0}^2\int_{s}^{\infty} (\tau-s) e^{-2w_0\tau}  d\tau + {\epsilon_0}^2\int_{s}^{\infty} e^{s-\tau}  e^{-2w_0\tau}  d\tau \\
		& \lesssim e^{-w_0s} \left(\frac{\epsilon_0}{C_0} +\epsilon_0^2\right)\\
		& \le e^{-w_0s} \epsilon_0 
	\end{align*}
	holds for  all $s\ge 0$, $\alpha \in \overline{B_{\epsilon_0}}(\alpha_0)$, $\mathbf{f}\in \mathcal{H}^k(B)$ with $\|\mathbf{f}\|_{\mathcal{H}^k(B)}\le \frac{\epsilon_0}{C_0}$, and $\mathbf{q}\in \mathcal{X}_{\epsilon_0}^k(B)$ with $0<\epsilon_0< \min\{1,\frac{\alpha_0}{2}\}$ sufficiently small and some constant $C_0>1$.  Thus, $\mathbf{K}_{\alpha}(\mathbf{f}, \cdot)$ maps $ \mathcal{X}_{\epsilon_0}^k(B)$ to $ \mathcal{X}_{\epsilon_0}^k(B)$. 
	
	Next, we show that $\mathbf{K}_{\alpha}(\mathbf{f}, \cdot)$ is a contraction. Indeed, by \eqref{eq:K} and Lemma \ref{lem-nonlinearity} we have
	\begin{align*}
		&\| \mathbf{K}_{\alpha}(\mathbf{f}, \mathbf{q})- \mathbf{K}_{\alpha}(\mathbf{f}, \mathbf{q}') \|_{ \mathcal{H}^k(B)}\\
		& \lesssim \int_0^s e^{-w_0(s-\tau)} \| \mathbf{N}(\mathbf{q})(\tau)- \mathbf{N}(\mathbf{q}')(\tau)\|_{ \mathcal{H}^k(B)}  d\tau + \int_{s}^{\infty}   \| \mathbf{N}(\mathbf{q})(\tau)- \mathbf{N}(\mathbf{q}')(\tau)\|_{ \mathcal{H}^k(B)}  d\tau \\
		&\quad + \int_{s}^{\infty} (\tau-s) \| \mathbf{N}(\mathbf{q})(\tau)- \mathbf{N}(\mathbf{q}')(\tau)\|_{ \mathcal{H}^k(B)} d\tau + \int_{s}^{\infty} e^{s-\tau}  \| \mathbf{N}(\mathbf{q})(\tau)- \mathbf{N}(\mathbf{q}')(\tau)\|_{ \mathcal{H}^k(B)}  d\tau \\
		&\lesssim  {\epsilon_0}\int_0^s e^{-w_0(s-\tau)}  e^{-2w_0\tau}  d\tau  \| \mathbf{q}- \mathbf{q}'\|_{ \mathcal{X}^k(B)} + {\epsilon_0} \int_{s}^{\infty}  e^{-2w_0\tau}  d\tau  \| \mathbf{q}- \mathbf{q}'\|_{ \mathcal{X}^k(B)} \\
		&\quad + {\epsilon_0}\int_{s}^{\infty} (\tau-s) e^{-2w_0\tau}  d\tau  \| \mathbf{q}- \mathbf{q}'\|_{ \mathcal{X}^k(B)}  + {\epsilon_0}\int_{s}^{\infty} e^{s-\tau}  e^{-2w_0\tau}  d\tau  \| \mathbf{q}- \mathbf{q}'\|_{ \mathcal{X}^k(B)} \\
		& \lesssim \epsilon_0 e^{-w_0s}  \| \mathbf{q}- \mathbf{q}'\|_{ \mathcal{X}^k(B)}.
	\end{align*}
	for all $\mathbf{q}, \mathbf{q}' \in \mathcal{X}_{\epsilon_0}^k(B)$, $s \geq 0$, and $\alpha \in \overline{B_{\epsilon_0}}(\alpha_0)$. Upon possibly choosing $\epsilon_0>0$ smaller, we infer
	\begin{align*}
		\left\|\mathbf{K}_\alpha(\mathbf{f}, \mathbf{q})-\mathbf{K}_\alpha(\mathbf{f}, \mathbf{q}')\right\|_{\mathcal{X}^k\left(B\right)} \leq \frac{1}{2}\|\mathbf{q}-\mathbf{q}'\|_{\mathcal{X}^k\left(B\right)}
	\end{align*}
	for all $\mathbf{q}, \mathbf{q}' \in \mathcal{X}_{\epsilon_0}^k\left(B\right)$, $\mathbf{f}\in \mathcal{H}^k(B)$ with $\|\mathbf{f}\|_{\mathcal{H}^k(B)}\le \epsilon_0$, and $\alpha \in \overline{B_{\epsilon_0}}(\alpha_0)$. 
	
	Lastly, let $\mathbf{f}, \mathbf{f}' \in \mathcal{H}_{\frac{\epsilon_0}{C_0}}^k\left(B\right)$ and let $\mathbf{q}_\alpha, \mathbf{q}'_\alpha \in \mathcal{X}_{\epsilon_0}^k\left(B\right)$ be the unique fixed point of $\mathbf{K}_\alpha(\mathbf{f},\cdot) ,$ $\mathbf{K}_\alpha(\mathbf{f}',\cdot)$ respectively. Then
	\begin{align*}
		\begin{aligned}
			\mathbf{q}_\alpha(s)-\mathbf{q}'_\alpha(s) & =\mathbf{K}_\alpha\left(\mathbf{f}, \mathbf{q}_\alpha\right)(s)-\mathbf{K}_\alpha\left(\mathbf{f}', \mathbf{q}'_\alpha\right)(s) \\
			& =\mathbf{K}_\alpha\left(\mathbf{f}, \mathbf{q}_\alpha\right)(s) - \mathbf{K}_\alpha\left(\mathbf{f}, \mathbf{q}'_\alpha\right)(s) +\mathbf{S}_\alpha(s)\tilde{\mathbf{P}}_\alpha (\mathbf{f}-\mathbf{f}')
		\end{aligned}
	\end{align*}
	Applying previous estimates yields
	\begin{align*}
		\begin{aligned}
			\left\|\mathbf{K}_\alpha\left(\mathbf{f}, \mathbf{q}_\alpha\right)-\mathbf{K}_\alpha\left(\mathbf{f}, \mathbf{q}'_\alpha\right)\right\|_{\mathcal{X}^k\left(B\right)} & \leq \frac{1}{2}\left\|\mathbf{q}_\alpha-\mathbf{q}'_\alpha\right\|_{\mathcal{X}^k\left(B\right)}, \\
			\left\| \mathbf{S}_\alpha(s)\tilde{\mathbf{P}}_\alpha (\mathbf{f}-\mathbf{f}') \right\|_{\mathcal{H}^k\left(B\right)} & \lesssim \mathrm{e}^{-w_0 s}\|\mathbf{f}-\mathbf{f}'\|_{\mathcal{H}^k\left(B\right)},
		\end{aligned}
	\end{align*}
	for all $s \geq 0$ and $\alpha \in \overline{B_{\epsilon_0}}(\alpha_0)$.  This shows the Lipschitz continuous dependence on the initial data.
\end{proof}

\subsection{Stable flow near the blow-up solution}
The initial data for our abstract Cauchy problem are introduced as follows.

\begin{defn}
	Let $\alpha_0, \alpha>0$, $T_0, T>0, \kappa_0,\kappa\in \mathbb{R}$, and $k\ge k_{\alpha}+2$ satisfying $T<T_0\sqrt{1+\alpha_0}$. We define the operator
	$$ \mathbf{U}_{\alpha, \kappa, T}: \mathcal{H}^k(\mathbb{R}) \to \mathcal{H}^k(B), \quad \mathbf{f}\mapsto \mathbf{f}^{T} + \mathbf{f}^{T}_0 -\mathbf{f}_{\alpha,\kappa}, $$
	where
	$$\mathbf{f}(y) = \begin{pmatrix}
		f_1(y)\\ f_2(y)
	\end{pmatrix}$$
	and 
	$$\mathbf{f}^{T}(y) = \begin{pmatrix} f_1(Ty)\\ Tf_2(Ty) \end{pmatrix}, \quad \mathbf{f}^{T}_0(y) = \begin{pmatrix}
		\tilde{U}_{\alpha_0,\infty,\kappa_0}(\frac{T}{T_0}y)\\
		\frac{T}{T_0}\alpha_0 + \left(\frac{T}{T_0}\right)^2 y\partial\tilde{U}_{\alpha_0,\infty,\kappa_0}(\frac{T}{T_0}y)
	\end{pmatrix},  \quad \mathbf{f}_{\alpha,\kappa}(y) = \begin{pmatrix}
		\tilde{U}_{\alpha,\infty,\kappa}(y)\\
		\alpha +  y\partial_y\tilde{U}_{\alpha,\infty,\kappa}(y)
	\end{pmatrix} $$
\end{defn}
%\rmk{$T<T_0\sqrt{1+\alpha_0}$ should hold to make $\mathbf{f}^{T}_0(y)$ well-defined.}

\begin{lem}\label{lem-initial} Let $\alpha_0, \alpha>0$, $T_0, T>0, \kappa_0,\kappa\in \mathbb{R}$, and $k\ge k_{\alpha}+2$ satisfying $T<T_0\sqrt{1+\alpha_0}$. Then,
	\begin{align*}
		\mathbf{U}_{\alpha, \kappa, T}(\mathbf{f}) = \mathbf{f}^T + \left[(\kappa_0-\kappa) -\alpha\left(\frac{T}{T_0}-1\right) \right]\mathbf{f}_{0,\alpha} + \left(\frac{T}{T_0}-1\right) \mathbf{f}_{1,\alpha} + (\alpha_0 -\alpha)\mathbf{g}_{0,\alpha} + \mathbf{r}\left(\alpha,\frac{T}{T_0}\right),
	\end{align*}
	for any $\mathbf{f}\in \mathcal{H}^k(\mathbb{R})$, where $\mathbf{f}_{0,\alpha}, \mathbf{f}_{1,\alpha}, \mathbf{g}_{0,\alpha} $ are introduced in Proposition \ref{prop-spectrum}, and 
	\begin{align*}
		\left\|\mathbf{r}\left(\alpha,\frac{T}{T_0}\right)\right\|_{\mathcal{H}^k(B)} \lesssim |\alpha -\alpha_0|^2 + \left|\frac{T}{T_0}-1\right|^{2} 
	\end{align*} 
	for all $\alpha \in	\overline{B_{\epsilon_1}}(\alpha_0)$ and $T\in \overline{B_{T_0 \epsilon_1}}(T_0)$, where $\epsilon_1 :=\min\{1,\frac{\alpha_0}{2}\}$.
\end{lem}
\begin{proof}
	For fixed $y \in [-1,1]$ and  $\epsilon_1>0$ defined above, Taylor's theorem applied to the components of the map	    
	\begin{align*}
		\overline{B_{\epsilon_1}}(\alpha_0) \times \mathbb{R}\times \overline{B_{T_0 \epsilon_1}}(T_0)  \rightarrow \mathbb{R}^2, \quad(\alpha, \kappa, T) \mapsto \mathbf{f}_0^T(y)-\mathbf{f}_{\alpha,\kappa}(y),
	\end{align*}	    
	yields   
	\begin{align*}
		\mathbf{f}_0^T(y)-\mathbf{f}_{\alpha,\kappa}(y) =\left[(\kappa_0-\kappa) -\alpha\left(\frac{T}{T_0}-1\right) \right]\mathbf{f}_{0,\alpha} + \left(\frac{T}{T_0}-1\right) \mathbf{f}_{1,\alpha} + (\alpha_0 -\alpha)\mathbf{g}_{0,\alpha} + {\mathbf{r}}\left(\alpha,\frac{T}{T_0}\right),
	\end{align*}	    
	with remainder	    
	\begin{align*}
		\begin{aligned}
			\left[{\mathbf{r}}\left(\alpha, \frac{T}{T_0}\right)(y)\right]_{\ell}= & \left(\frac{T}{T_0}-1\right)^2 \int_0^1 T_0^2 \left. \partial_{T^{\prime}}^2\left[\mathbf{f}_0^{T^{\prime}}(y)\right]_{\ell}\right|_{T^{\prime}=T_0+z(T-T_0)}(1-z) \mathrm{d} z \\
			& -\left.  (\alpha_0 -\alpha)^2\int_0^1 \partial_{\alpha'}^2 \left[\mathbf{f}_{\alpha^{\prime},\kappa}(y)\right]_{\ell}\right|_{\alpha^{\prime}=\alpha+z(\alpha_0-\alpha) }(1-z) \mathrm{d} z, \quad \ell=1,2 .
		\end{aligned}
	\end{align*}
	Note that $	\partial^2 \mathbf{f}_0^T(y)$ and $\partial^2 \mathbf{f}_{\alpha,\kappa}(y)$	are smooth and uniformly bounded in $[-1,1]$ for all $\alpha \in	\overline{B_{\epsilon_1}}(\alpha)$, $\kappa \in \mathbb{R}$, and $T\in \overline{B_{T_0 \epsilon_1}}(T_0)$. this concludes the proof.   
	
\end{proof}

\begin{prop}\label{prop-main}
	Let $\alpha_0,T_0>0,$  $\kappa_0\in \mathbb{R}$,  $k\ge k_{\alpha_0}+3, 0< w_0<1$. There are constants $\epsilon_2>0$ sufficiently small and $C_2>0$ such that for all real-valued $\mathbf{f}\in \mathcal{H}^k$ with $\|\mathbf{f}\|_{\mathcal{H}^k}\le \frac{\epsilon_2}{C_2^2}$, there exist parameters $\alpha^*, T^*>0, \kappa^*\in\mathbb{R},$ and a unique $\mathbf{q}_{\alpha^*,\kappa^*,T^*}\in C([0,\infty), \mathcal{H}^k(B))$ such that $\|\mathbf{q}_{\alpha^*,\kappa^*,T^*}(s) \|_{\mathcal{H}^k(B)}\le \epsilon_2 e^{-w_0 s} $ and 
	\begin{align}\label{eq:nlw-mild}
		\mathbf{q}_{\alpha^*,\kappa^*,T^*}(s) = \mathbf{S}_{\alpha^*}(s)\mathbf{U}_{\alpha^*,\kappa^*,T^*}(\mathbf{f}) + \int_0^{s} \mathbf{S}_{\alpha^*}(s-\tau) \mathbf{N}(\mathbf{q}_{\alpha^*,\kappa^*,T^*}(\tau)) d\tau
	\end{align}
	with
	\begin{align*}
		|\alpha^*-\alpha_0| + |\kappa^* -\kappa_0| + \left|\frac{T^*}{T_0}-1\right| \le \frac{\epsilon_2}{C_2}.
	\end{align*}
	%such that the correction $\mathbf{C}_{\alpha}(\mathbf{U}(f),\mathbf{q}_\alpha)\equiv 0$, where $U(f)$ is the initial data. 
\end{prop}

\begin{proof}
	Fix $\alpha_0>0, k\ge k_{\alpha_0}+3, 0< w_0<1$,  we take $\epsilon_0>0, C_0>1$ from Proposition \ref{prop-modified-existence} and let $0<\epsilon^{\prime} \leq \min\{\epsilon_0,\epsilon_1\}$ and $C^{\prime} \geq C_0$. Let  
	\begin{align*}
		\epsilon_2 = \frac{\epsilon^{\prime}}{M} \quad \text { and } \quad C_2 = M C^{\prime}
	\end{align*}    
	for $M \geq 1$. Let $\mathbf{f} \in \mathcal{H}^k(\mathbb{R})$ with $\|\mathbf{f}\|_{\mathcal{H}^k(\mathbb{R})} \leq \frac{\epsilon_2}{C_2^2}$. If $M \geq 1$ is large enough, we get from Lemma \ref{lem-initial}
	\begin{align*}
		\|	\mathbf{U}_{\alpha, \kappa, T}(\mathbf{f}) \|_{\mathcal{ H }^k(B)} =& \|\mathbf{f}^T\|_{\mathcal{ H }^k(B)} + \left(|\kappa_0-\kappa| + \alpha\left|\frac{T}{T_0}-1\right| \right) \|\mathbf{f}_{0,\alpha}\|_{\mathcal{ H }^k(B)} \\
		&  + \left|\frac{T}{T_0}-1\right| \|\mathbf{f}_{1,\alpha}\|_{\mathcal{ H }^k(B)} + |\alpha_0 -\alpha| \|\mathbf{g}_{0,\alpha}\|_{\mathcal{ H }^k(B)} +\left\|\mathbf{r}\left(\alpha,\frac{T}{T_0}\right)\right\|_{\mathcal{ H }^k(B)}\\
		\le & \frac{\epsilon_2}{C'} 
	\end{align*}
	for all $\alpha \in \overline{B_{\frac{\epsilon_2}{C_2}}}(\alpha_0)$, $\frac{T}{T_0} \in \overline{\mathbb{B}_{\frac{\epsilon_2}{C_2}}^1(1)}$, and $\kappa \in \overline{B_{\frac{\epsilon_2}{C_2}}}(\kappa_0)$. Hence, we can fix $\epsilon^{\prime}>0$ and $C^{\prime}>1$ for now so that $\mathbf{U}_{\alpha, \kappa, T}(\mathbf{f}) \in \mathcal{H}^k\left(B\right)$ satisfies the assumptions for the initial data in Proposition \ref{prop-modified-existence} and we conclude the existence of a unique solution $\mathbf{q}_{\alpha, \kappa, T} \in \mathcal{X}^k\left(B\right)$ with $\left\|\mathbf{q}_{\alpha, \kappa, T} \right\|_{\mathcal{X}^k\left(B\right)} \leq \epsilon_2$ to
	\begin{align*}
		\mathbf{q}_{\alpha, \kappa, T}(s)=\mathbf{S}_\alpha(s)\left(\mathbf{U}_{\alpha, \kappa, T}(\mathbf{f})-\mathbf{C}_\alpha\left(\mathbf{U}_{\alpha, \kappa, T}(\mathbf{f}), \mathbf{q}_{\alpha, \kappa, T}\right)\right)+\int_0^s \mathbf{S}_\alpha (s-\tau) \mathbf{N}_\alpha \left(\mathbf{q}_{\alpha, \kappa, T}(\tau)\right) \mathrm{d} \tau,
	\end{align*}    
	for each$\alpha \in \overline{B_{\frac{\epsilon_2}{C_2}}}(\alpha_0)$, $\frac{T}{T_0} \in \overline{\mathbb{B}_{\frac{\epsilon_2}{C_2}}^1(1)}$, and $\kappa \in \overline{B_{\frac{\epsilon_2}{C_2}}}(\kappa_0)$. Now, the task is to determine parameters such that the correction term $\mathbf{C}_\alpha\left(\mathbf{U}_{\alpha, \kappa, T}(\mathbf{f}), \mathbf{q}_{\alpha, \kappa, T}\right) \in \operatorname{rg}\left(\mathbf{P}_\alpha\right)$ is equal to $\mathbf{0}$. 
	
	By Proposition \ref{prop-spectrum} we have that $\operatorname{rg}\left(\mathbf{P}_\alpha\right) \subset \mathcal{H}^k(B)$ is a finite-dimensional sub-Hilbert space spanned by the linearly independent set $\left\{\mathbf{f}_{0, \alpha}, \mathbf{f}_{1, \alpha}, \mathbf{g}_{0, \alpha}\right\}$ of symmetry modes. Therefore, it is sufficient to prove that there are parameters for which the linear functional 
	\begin{align*}
		\ell_{\alpha,\kappa, T}: \operatorname{rg}\left(\mathbf{P}_\alpha\right) \rightarrow \mathbb{R}, \quad \mathbf{g} \mapsto\left( \mathbf{C}_\alpha\left(\mathbf{U}_{\alpha, \kappa, T}(\mathbf{f}), \mathbf{q}_{\alpha, \kappa, T}\right)  \mid \mathbf{g} \right)_{\mathcal{\mathcal { H }}^k(B)}
	\end{align*}    
	is identically zero on a basis of $\operatorname{rg}\left(\mathbf{P}_\alpha\right)$. Note with Lemma \ref{lem-initial} that
	\begin{align*}
		\begin{aligned}
			\ell_{\alpha,\kappa, T}(\mathbf{g})= & \left( \mathbf{P}_{\alpha} \mathbf{f}^T \mid \mathbf{g}\right)_{\mathcal{H}^k(B)} +   
			\left[(\kappa_0-\kappa) -\alpha\left(\frac{T}{T_0}-1\right) \right]  \left(\mathbf{f}_{0,\alpha}  \mid \mathbf{g}\right)_{\mathcal{H}^k(B)} +  \left(\frac{T}{T_0}-1\right) \left(\mathbf{f}_{1,\alpha}  \mid \mathbf{g}\right)_{\mathcal{H}^k(B)}\\
			& + (\alpha_0 -\alpha)\left(\mathbf{g}_{0,\alpha}  \mid \mathbf{g}\right)_{\mathcal{H}^k(B)} +   \left(\mathbf{P}_{\alpha}\mathbf{r}\left(\alpha,\frac{T}{T_0}\right) \mid \mathbf{g}\right)_{\mathcal{H}^k(B)}\\
			& + \left(  \mathbf{P}_{0,\alpha}\int_0^{\infty} \mathbf{N}(\mathbf{q}_{\alpha, \kappa, T}(\tau)) d\tau \mid \mathbf{g}\right)_{\mathcal{H}^k(B)} + \left(\mathbf{L}_{\alpha}\mathbf{P}_{0,\alpha}\int_0^{\infty} (-\tau)\mathbf{N}(\mathbf{q}_{\alpha, \kappa, T}(\tau)) d\tau \mid \mathbf{g}\right)_{\mathcal{H}^k(B)}\\
			& +  \left(  \mathbf{P}_{1,\alpha} \int_0^{\infty} e^{-\tau} \mathbf{N}(\mathbf{q}_{\alpha, \kappa, T}(\tau)) d\tau \mid \mathbf{g}\right)_{\mathcal{H}^k(B)}  		
		\end{aligned}
	\end{align*}
	To achieve vanishing of $\ell_{\alpha,\kappa, T}$, choose the dual basis $\left\{\mathbf{g}_\alpha^1, \mathbf{g}_\alpha^2, \mathbf{g}_\alpha^{3}\right\}$ for $\left\{ \mathbf{g}_{0, \alpha},\mathbf{f}_{0, \alpha}, \mathbf{f}_{1, \alpha}\right\}$, which is obtained by letting $\Gamma(\alpha)^{m n}$ for $m, n=1, 2,3$ be the components of the inverse matrix of the real-valued Gram matrix 
	$$\Gamma(\alpha) =  \begin{bmatrix}
		\left(\mathbf{g}_{0, \alpha} \mid\mathbf{g}_{0, \alpha} \right)_{\mathcal{H}^k(B)} & 	\left(\mathbf{g}_{0, \alpha} \mid\mathbf{f}_{0, \alpha} \right)_{\mathcal{H}^k(B)} & 
		\left(\mathbf{g}_{0, \alpha} \mid\mathbf{f}_{1, \alpha} \right)_{\mathcal{H}^k(B)}  \\
		\left(\mathbf{f}_{0, \alpha} \mid\mathbf{g}_{0, \alpha} \right)_{\mathcal{H}^k(B)} & 	\left(\mathbf{f}_{0, \alpha} \mid\mathbf{f}_{0, \alpha} \right)_{\mathcal{H}^k(B)} &
		\left(\mathbf{f}_{0, \alpha} \mid\mathbf{f}_{1, \alpha} \right)_{\mathcal{H}^k(B)} \\
		\left(\mathbf{f}_{1, \alpha} \mid\mathbf{g}_{0, \alpha} \right)_{\mathcal{H}^k(B)}&
		\left(\mathbf{f}_{1, \alpha} \mid\mathbf{f}_{0, \alpha} \right)_{\mathcal{H}^k(B)} &
		\left(\mathbf{f}_{1, \alpha} \mid\mathbf{f}_{1, \alpha} \right)_{\mathcal{H}^k(B)}    	 	
	\end{bmatrix},$$
	and putting 
	\begin{align*}
		\mathbf{g}_\alpha^n:=\Gamma(\alpha)^{n1} \mathbf{g}_{0, \alpha}+\Gamma(\alpha)^{n2} \mathbf{f}_{0, \beta}+\Gamma(\alpha)^{n3} \mathbf{f}_{1, \alpha}, \quad n=1, 2,3.
	\end{align*}
	Then,   	
	\begin{align*}
		\left(\mathbf{g}_{0, \alpha} \mid \mathbf{g}_\alpha^j\right)_{\mathcal{H}^k(B)} =\delta_1^j, \quad	\left(\mathbf{f}_{0, \alpha} \mid \mathbf{g}_\alpha^j\right)_{\mathcal{H}^k(B)} =\delta_2^j, \quad 
		\left(\mathbf{f}_{1, \alpha} \mid \mathbf{g}_\alpha^j\right)_{\mathcal{H}^k(B)} =\delta_3^j
	\end{align*}    	
	for $ j=1, 2, 3$ and the components of each element in $\left\{\mathbf{g}_\alpha^1, \mathbf{g}_\alpha^2, \mathbf{g}_\alpha^{3}\right\} \subset \operatorname{rg}\left(\mathbf{P}_\alpha\right)$ are smooth functions which depend smoothly on the parameter $\alpha$ by Cramer's rule. Next, define the continuous map $F=\left(\alpha_0+F_1, \kappa_0+F_2 + \alpha F_3, T_0-T_0F_{3}\right): \overline{B_{\frac{\epsilon_2}{C_2}}}(\alpha_0) \times \overline{B_{\frac{\epsilon_2}{C_2}}}(\kappa_0) \times \overline{B_{\frac{T_0 \epsilon_2}{C_2}}}(T_0) \to \mathbb{R}^3$ by
	\begin{align*}
		\begin{aligned}
			F_n(\alpha,\kappa, T)= & \left( \mathbf{P}_{\alpha} \mathbf{f}^T \mid 	\mathbf{g}_\alpha^n\right)_{\mathcal{H}^k(B)}  + \left(\mathbf{P}_{\alpha}\mathbf{r}\left(\alpha,\frac{T}{T_0}\right) \mid 	\mathbf{g}_\alpha^n\right)_{\mathcal{H}^k(B)}\\
			& + \left(  \mathbf{P}_{0,\alpha}\int_0^{\infty} \mathbf{N}(\mathbf{q}_{\alpha, \kappa, T}(\tau)) d\tau \mid \mathbf{g}_\alpha^n\right)_{\mathcal{H}^k(B)} + \left(\mathbf{L}_{\alpha}\mathbf{P}_{0,\alpha}\int_0^{\infty} (-\tau)\mathbf{N}(\mathbf{q}_{\alpha, \kappa, T}(\tau)) d\tau \mid \mathbf{g}_\alpha^n\right)_{\mathcal{H}^k(B)}\\
			& +  \left(  \mathbf{P}_{1,\alpha} \int_0^{\infty} e^{-\tau} \mathbf{N}(\mathbf{q}_{\alpha, \kappa, T}(\tau)) d\tau \mid \mathbf{g}_\alpha^n\right)_{\mathcal{H}^k(B)},  		 \quad n=1, 2,3.
		\end{aligned}
	\end{align*}
	We use Cauchy-Schwarz and Lemmas \ref{lem-nonlinearity} and \ref{lem-initial} to get the estimate
	\begin{align*}
		\begin{aligned}
			\left|F_n(\alpha, \kappa, T)\right| & \lesssim\left\|\mathbf{f}^T\right\|_{\mathcal{H}^k(B)}+\left\|\mathbf{r}\left(\alpha,\frac{T}{T_0}\right) \right\|_{\mathcal{H}^k(B)} + \left\|\mathbf{q}_{\alpha, \kappa, T}\right\|_{\mathcal{X}^k(B)}^2 \\
			& \lesssim \frac{\epsilon_2}{C_2^2}+\frac{\epsilon_2^2}{C_2^2}+\epsilon_2^2 \\
			& \lesssim \frac{\epsilon_2}{C_2} \frac{1}{M C^{\prime}}+\frac{\epsilon_2}{C_2} \frac{\epsilon'}{M^2 C^{\prime}}+\frac{\epsilon_2}{C_2} C^{\prime} \epsilon^{\prime}
		\end{aligned}
	\end{align*}
	for all $\alpha \in \overline{B_{\frac{\epsilon_2}{C_2}}}(\alpha_0)$, $T \in \overline{\mathbb{B}_{\frac{T_0\epsilon_2}{C_2}}(T_0)}$, $\kappa \in \overline{B_{\frac{\epsilon_2}{C_2}}}(\kappa_0)$, and $n=1, 2, 3$. Now, upon choosing $\epsilon^{\prime}>0$ smaller and $C'>1, M\ge 1$ larger, we can fix the above values of $\epsilon_2, C_2$ so that $F$ becomes a continuous self-map on $\overline{B_{\frac{\epsilon_2}{C_2}}}(\alpha_0) \times \overline{B_{\frac{\epsilon_2}{C_2}}}(\kappa_0) \times \overline{B_{\frac{T_0 \epsilon_2}{C_2}}}(T_0) $. 
	%for any $0<\epsilon \leq \frac{\epsilon}{M}$ and $C \geq MC'$. 
	According to Brouwer's fixed-point theorem, $F$ has a fixed point $\left(\alpha^*, \kappa^*, T^*\right) \in \overline{B_{\frac{\epsilon_2}{C_2}}}(\alpha_0) \times \overline{B_{\frac{\epsilon_2}{C_2}}}(\kappa_0) \times \overline{B_{\frac{T_0 \epsilon_2}{C_2}}}(T_0)$ which, by construction, satisfies   	
	\begin{align*}
		\alpha^{*}&=\alpha_0 + F_1 (\alpha^*,\kappa^*, T^*)=\alpha^{*} + \ell_{\alpha^*, \kappa^* T^*}\left(\mathbf{g}_{\alpha^*}^1\right),\\
		\kappa^*&=\kappa_0 +F_{2} (\alpha^*,\kappa^*, T^*) + \alpha^* F_{3} (\alpha^*,\kappa^*, T^*)=\kappa^*+ \ell_{\alpha^*, \kappa^* T^*}\left(\mathbf{g}_{\alpha^*}^2\right) + \alpha^*\ell_{\alpha^*, \kappa^* T^*}\left(\mathbf{g}_{\alpha^*}^3\right),\\
		\frac{T^*}{T_0}&=1-F_{3} (\alpha^*,\kappa^*, T^*)=\frac{T^*}{T_0}- \ell_{\alpha^*, \kappa^* T^*}\left(\mathbf{g}_{\alpha^*}^3\right),
	\end{align*}
	Thus $\ell_{\alpha^*,\kappa^*, T^*} \equiv 0$, as desired.
\end{proof}

\subsection{Proof of stability of the generalized self-similar blow-up}
\begin{proof}[Proof of Theorem \ref{thm:main}]
	Fix $0<\delta<1$ and let $w_0 = 1-\delta$. From Proposition \ref{prop-main}, we pick $0<\epsilon_2<1$ and $C_2 \geq 1$. Then, for any real-valued $(f, g) \in H^{k+1}(\mathbb{R})\times H^{k}(\mathbb{R}) $ with $\|(f, g)\|_{H^{k+1}(\mathbb{R})\times H^{k}(\mathbb{R})} \leq \frac{\epsilon_2}{C_2^2}$ there exists a unique mild solution $\mathbf{q}_{\alpha^*,\kappa^*,T^*} = \left(q_{\alpha^*,\kappa^*,T^*, 1}, q_{\alpha^*,\kappa^*,T^*, 2}\right) \in C([0,\infty), \mathcal{H}^k(B))$ to \eqref{eq:nlw-mild} that is also a  classical solution to the abstract Cauchy problem \eqref{eq:q-nonlinear}. Let
	\begin{align*}
		u(t, x):= U_{\alpha^*,\kappa^*,T^*}\left(-\log \frac{T^*-t}{T^*}, \frac{x-x_0}{T^*-t}\right) + q_{\alpha^*,\kappa^*,T^*, 1}\left(-\log \frac{T^*-t}{T^*}, \frac{x-x_0}{T^*-t}\right)
	\end{align*}
	then $u(t,x)$ gives the unique solution to the Cauchy problem \eqref{eq:nlw} in the backward lightcone $\Gamma(T^*, x_0)$ with initial data
	$$ u(0,x) = u_{\alpha_0,\infty,\kappa_0,T_0,x_0}(0,x) + f(x), \quad \partial_t u(0,x) = \partial_t u_{\alpha_0,\infty,\kappa_0,T_0,x_0}(0,x) + g(x), \quad x\in B_{T^*}(x_0).$$
	%is related to $u_{\beta^*, T^*, 2}$ via
	%\begin{align*}
	%\partial_t u(t, x)=\left(T^*-t\right)^{-s_p-1} u_{\beta^*, T^*, 2}\left(\log \frac{T^*}{T^*-t}, \frac{x}{T^*-t}\right)
	%\end{align*}
	Moreover, by scaling of the homogeneous seminorms, we infer from Proposition \ref{prop-main} the bounds
	\begin{align*}
		(T^*-t)^{-\frac{1}{2} + s} \| u(t,\cdot) - u_{\alpha^*, \infty, \kappa^*, T^*, x_0}(t,\cdot)\|_{\dot{H}^{s}(B_{T^*-t}(x_0))} &= \left\|q_1\left(-\log \frac{T^*-t}{T^*}, \cdot\right)\right\|_{\dot{H}^s(B)} \\	
		&\le \epsilon_2 (T^*-t)^{1-\delta} ,
	\end{align*}
	for $s=0,1,\ldots,k+1$, and 
	\begin{align*}    
		(T^*-t)^{-\frac{1}{2} + s} \| \partial_t u(t,\cdot) - \partial_t u_{\alpha^*, \infty, \kappa^*, T^*, x_0}\|_{ \dot{H}^{s-1}(B_{T^*-t}(x_0))}  &= \left\|q_2\left(-\log \frac{T^*-t}{T^*}, \cdot\right)\right\|_{\dot{H}^s(B)} \\	
		&\le \epsilon_2 (T^*-t)^{1-\delta} ,
	\end{align*}
	for $s=1,\ldots,k+1$.
	
\end{proof}

\section*{Acknowledgment} The authors would like to thank Professor Jean-Pierre Eckmann and Professor Hatem Zaag for introducing this problem and for their invaluable discussions and insightful comments.

\appendix

\section{Frobenius analysis}
Recall that the eigen-equation \eqref{eigen-eq} can be rewritten as the standard form
\begin{align}\label{eigen-eq-z-appendix}
	\varphi'' + \left[\frac{\gamma}{z}+ \frac{\delta}{z-1}  + \frac{\epsilon}{z-d}\right] \varphi' + \frac{abz-c}{z(z-1)} \varphi = 0,
\end{align}
where we take $\gamma = \lambda -\sqrt{1+\alpha}, \delta = \lambda +\sqrt{1+\alpha}, d = -\frac{\sqrt{1+\alpha}-1}{2}, a=\lambda, b = \lambda+1, c= -\frac12 (\lambda^2+\lambda) (\sqrt{1+\alpha}-1), \epsilon = 2 = \alpha+\beta +1 -\gamma-\delta.$ 

By Frobenius theory, we can always find an infinite series \textit{local} solution of the form
\begin{align}\label{series-exp}
	\varphi(z) = (z-z_0)^s \sum_{k=0}^\infty a_k (z-z_0)^k 
\end{align}
near each singular point $z_0\in \{0,1\}$. More precisely, by substituting \eqref{series-exp} into \eqref{eigen-eq-z-appendix} and matching the lowest order term, we have the following indicial polynomials 
\begin{align*}
	P_0(s) = s(s-1+\gamma) =0 \text{~for~} z_0 =0, 
\end{align*}
and
\begin{align*}
	P_1(s) = s(s-1+\delta) =0 \text{~for~} z_0 =1. 
\end{align*}

\begin{prop}\label{prop-1dsol-space}
	Let $\lambda \in \mathbb{C}$ with $\operatorname{Re}\lambda > -1,$ and $k\ge k_\alpha+2$, then any $H^{k+1}(0,1)$ solution to \eqref{eigen-eq-z-appendix} must belong to $C^{\infty}[0,1]$. Moreover, the set of local smooth solutions to \eqref{eigen-eq-z-appendix} around $1$ is one-dimensional.
\end{prop}
%For $\lambda \in \mathbb{C}$ with $\operatorname{Re}\lambda > -1,$ and $k\ge k_\alpha+1$, the set of local smooth solutions to \eqref{eigen-eq-z} around $1$ is one-dimensional, and the set of local smooth solutions to \eqref{eigen-eq-z} around $0$ is also one-dimensional except for $\lambda = 1+\sqrt{1+\alpha} - m, m\in \mathbb{N}_0$. Moreover, any $H^{k+1}(0,1)$ solution to \eqref{eigen-eq-z-appendix} must belong to $C^{\infty}[0,1]$ and be analytic at $0$ and $1$.

\begin{proof} We prove by performing Frobenius analysis at $z_0=0$ and $z_0=1$ respectively. For $z_0 =0$, the two roots of the corresponding indicial polynomial $P_0(s)$ are given by
	$$s_+ = 1+\sqrt{1+\alpha} -\lambda,\quad s_-= 0, \quad \textit{~if~}  -1< \operatorname{Re}\lambda \le 1+\sqrt{1+\alpha} $$
	and 
	$$s_+ =0, \quad s_-= 1+\sqrt{1+\alpha} -\lambda, \quad \textit{~if~} \operatorname{Re} \lambda > 1+\sqrt{1+\alpha}. $$
	Therefore, we have
	\begin{itemize}
		\item[Case 1]  if $-1< \operatorname{Re}\lambda \le 1+\sqrt{1+\alpha}$ and $ s_+ - s_{-}\notin \mathbb{N}_0$, by Lemma \ref{lem:Frobenius} the two independent local solutions are given by 
		\begin{align*}
			\varphi_{01} = z^{s_+} h_+(z) \quad \text{~and~} \quad \varphi_{02} = h_-(z),
		\end{align*}
		where $h_{\pm}$ are analytic functions around $0$. Note that $\varphi_{01}$ does not belong to $H^{k_\alpha+3}(0,1)$ since $s_+\notin \mathbb{N}_0$ and $\operatorname{Re} s_+ -k_\alpha-3 <-1$. Thus, any $H^{k+1}$ local solutions must be a multiple of $\varphi_{02}$, thus is smooth and analytic at $0$. 
		%And the set of $H^{k_\alpha+2}$ local solutions around $z_0=0$ is one dimensional, i.e. $\{ c_0\varphi_{02} \mid  c_0\in\mathbb{C}\}$.
		\item[Case 2] if $-1< \operatorname{Re}\lambda \le 1+\sqrt{1+\alpha}$ and $ s_+ - s_{-}\in \mathbb{N}_0$, the two independent local solutions are given by 
		\begin{align*}
			\varphi_{01} = z^{s_+} h_+(z) \quad \text{~and~} \quad \varphi_{02} = h_-(z) + c'z^{s_+} h_+(z) \log(z).
		\end{align*}
		where $h_{\pm}$ are analytic functions around $0$ and $c'\in \mathbb{C}$ is a (possibly vanishing) constant. Note that when $c'\neq 0$, $c'z^{s_+} h_+(z) \log(z)$ does not belong to $H^{k_\alpha+3}(0,1)$, so $\varphi_{02}\notin H^{k+1}(0,1)$. Thus, any $H^{k+1}$ local solution is a multiple of $\varphi_{01}$ unless $c'=0$. In either case, the $H^{k+1}$ local solution must be smooth and analytic at $0$.    
		%Note that in this case, $\lambda = 1+\sqrt{1+\alpha} - m \ge -\frac{1}{2}, m\in \mathbb{N}$ has only finite number of possible values. And by comparing higher order terms of the series expansion, one can check that for such choices of $\lambda$, the constant $c$ in $\varphi_{02}$ is always nonzero \rmk{(Is there any easier and more direct way to check this? prove this in the next subsection)}. Thus, there is also only one dimensional space of $H^{k_\alpha+2}(0,1)$ local smooth solutions around $z_0=0$, i.e. $\{c_1\varphi_{01}\mid c_1\in \mathbb{C}\}$.
		\item[Case 3] if $\operatorname{Re} \lambda > 1+\sqrt{1+\alpha}$,
		the two independent local solutions are given by 
		\begin{align*}
			\varphi_{01} = h_+(z) \quad \text{~and~} \quad \varphi_{02} =  z^{s_-}h_-(z) + c' h_+(z) \log(z),
		\end{align*}
		where $h_{\pm}$ are analytic functions around $0$ and $c'\in \mathbb{C}$ is a (possibly vanishing) constant. Since $\operatorname{Re} s_{-} = 1+\sqrt{1+\alpha} -\operatorname{Re}\lambda<0$, $\varphi_{02}$ does not belong to $H^{k+1}(0,1)$ regardless the choice of $c'\in \mathbb{C}$. Thus, any $H^{k+1}(0,1)$ local smooth solution around $z_0=0$ is a multiple of $\varphi_{01}$, thus is smooth and analytic at $0$. 
	\end{itemize}
	
	Now, we consider $z_0 =1$, then the two roots of the corresponding indicial polynomial $P_1(s)$ are given by
	$$ s_+ =1 -\sqrt{1+\alpha} -\lambda, \quad s_{-} = 0, \quad \text{if~} -1< \operatorname{Re} \lambda \le 1 -\sqrt{1+\alpha}$$
	and
	$$ s_+ = 0,\quad  s_{-} = 1 -\sqrt{1+\alpha} -\lambda, \quad \text{if~} \operatorname{Re} \lambda > 1 -\sqrt{1+\alpha}. $$
	Therefore, we have
	\begin{itemize}
		\item[Case 1']  if  $-1< \operatorname{Re}\lambda \le 1-\sqrt{1+\alpha}$ and $ s_+ - s_{-}\notin \mathbb{N}_0$, by Lemma \ref{lem:Frobenius} the two independent local solutions are given by 
		\begin{align*}
			\varphi_{01} = (z-1)^{s_+} h_+(z-1) \quad \text{~and~} \quad \varphi_{02} = h_-(z-1),
		\end{align*}
		where $h_{\pm}$ are analytic functions around $0$. Note that $\varphi_{01}$ does not belong to $H^{k+1}(0,1)$ since $s_+\notin \mathbb{N}_0$ and $\operatorname{Re} s_+ - k-1 \le -1$ for all $\operatorname{Re}\lambda \ge -1$. Thus, any $H^{k+1}$ local solutions must be a multiple of $\varphi_{02}$, thus is smooth and analytic at $1$. 
		%And the set of $H^{k_\alpha+2}$ local solutions around $z_0=0$ is one dimensional, i.e. $\{ c_0\varphi_{02} \mid  c_0\in\mathbb{C}\}$.
		\item[Case 2'] if $ -1<\operatorname{Re}\lambda \le 1-\sqrt{1+\alpha}$ and $ s_+ - s_{-}\in \mathbb{N}_0$, then $s_{+} = s_{-} =0$, since $s_{+} \in \mathbb{N}_0$ and $\operatorname{Re} s_+ <1$ for all $\operatorname{Re}\lambda > -1$. Thus, the two independent local solutions are given by 
		\begin{align*}
			\varphi_{01} = h_+(z-1) \quad \text{~and~} \quad \varphi_{02} = h_-(z-1) + c' h_+(z-1) \log(z-1).
		\end{align*}
		where $h_{\pm}$ are analytic functions around $0$ and $c'\neq 0$. Note that $ h_+(z-1) \log(z-1)$ does not belong to $H^{k+1}(0,1)$, so $\varphi_{02}\notin H^{k+1}(0,1)$. Thus, any $H^{k+1}$ local solution must be a multiple of $\varphi_{01}$, thus is smooth and analytic at $1$.    
		\item[Case 3'] if $\operatorname{Re} \lambda > 1-\sqrt{1+\alpha}$,
		the two independent local solutions are given by 
		\begin{align*}
			\varphi_{01} = h_+(z-1) \quad \text{~and~} \quad \varphi_{02} =  (z-1)^{s_-}h_-(z-1) + c' h_+(z-1) \log(z-1),
		\end{align*}
		where $h_{\pm}$ are analytic functions around $0$ and $c'\in \mathbb{C}$ is a (possibly vanishing) constant. Since $\operatorname{Re} s_{-} = 1-\sqrt{1+\alpha} -\operatorname{Re}\lambda<0$, $\varphi_{02}$ does not belong to $H^{k+1}(0,1)$ regardless the choice of $c'\in \mathbb{C}$. Thus, any $H^{k+1}(0,1)$ local smooth solution around $z_0=0$ is a multiple of $\varphi_{01}$, thus is smooth and analytic at $1$. 
	\end{itemize}
	
	In summary, any $H^{k+1}(0,1)$ solution to \eqref{eigen-eq-z-appendix} must be smooth and analytic around $0$ and $1$. Since the eigen-equation \eqref{eigen-eq-z-appendix} is a second-order elliptic equation in $(0,1)$, any $H^{k+1}(0,1)$ solution must belong to $C^\infty(0,1)$, thus also  belongs to $C^\infty[0,1]$. Moreover, the set of local smooth solutions to \eqref{eigen-eq-z-appendix} around $1$ is one-dimensional.
	%the sets of local smooth solutions to \eqref{eigen-eq-z} around $0$ and $1$ are one-dimensional except for the case $-\frac{1}{2} \le \operatorname{Re}\lambda \le 1+\sqrt{1+\alpha}$ and $ s_+ - s_{-}\in \mathbb{N}_0$ which is equivalent to $\lambda = 1+\sqrt{1+\alpha} - m, m\in \mathbb{N}_0$ and $m\le \frac{3}{2}+\sqrt{1+\alpha}$.  Moreover,
\end{proof}

\section{Non-existence of other generalized eigenfunctions}

\begin{lem}\label{lem:generalized-eigenfunction} There does not exist functions $\mathbf{v}\in \mathcal{H}^{4}$ such that $\mathbf{L}_{3} \mathbf{v} = \mathbf{g}_{0,3}.$
\end{lem}

\begin{proof}
	Assume that $\mathbf{v}=(v_1, v_2)   \in \mathcal{H}^{4}$ satisfies $\mathbf{L}_{3} \mathbf{v} = \mathbf{g}_{0,3},$ then we have
	\begin{align*}
		\left\{
		\begin{aligned}
			&v_2 -y\partial_y v_1 = g_{0,3,1}\\
			&\partial_{yy} v_1 -\frac{6}{2+y}\partial_y v_1 -v_2 -y\partial_y v_2 = g_{0,3,2}.
		\end{aligned}\right.
	\end{align*}
	Using the first equation to solve for $v_2$, we see that solving this system of ODEs reduces to
	\begin{align*}
		\left(2y+ \frac{6}{2+y} \right)\partial_y v_1  +(y^2-1)\partial_{yy} v_1 = g, 
	\end{align*} 
	where
	\begin{align*}
		g &= -g_{0,3,1} -y\partial_y g_{0,3,1} - g_{0,3,2}\\
		&= \log(2+y) + \frac{1}{(2+y)^2} \left(y^2 - \frac{3}{4} y - \frac{5}{2}\right).
	\end{align*}
	Thus, we have
	\begin{align*}
		\partial_y v_1 =c p(y)^{-1} +  p(y)^{-1} \int_y^1 \frac{p(z)g(z)}{1 -z^2} dz, \quad p(y) = \left(\frac{1-y}{1+y}\right)^{2}(2+y)^2,
	\end{align*}
	where $c\in\mathbb{R}$ is a constant. Since $v_1\in L^2(\frac{1}{2},1)$, we have $c=0$. 
	It suffices to show that $\partial_y v_1 \notin H^3(-1,-\frac{1}{2})$. Indeed, we have
	\begin{align*}
		(1-y)^2(2+y)^2\partial_y v_1 = (1+y)^2 \int_y^1 \frac{G(z)}{(1+z)^3} dz,
	\end{align*}
	where $G(z) = (1-z)\left( (2+z)^2 \log(2+z) + z^2 -\frac{3}{4}z - \frac52 \right)$. By Taylor expansion, we can write 
	$$ G(z) = -\frac32 - \frac{11}{4}(1+z) + \frac{27}{4}(1+z)^2 + R,  $$
	with $R = O((1+z)^3)$ for $z\to -1$. Then,
	\begin{align*}
		(1+y)^2 \int_y^1 \frac{G(z)}{(1+z)^3} dz& = \left(\frac{25}{16} + \frac{27}{4}\log 2\right)(1+y)^2-\frac{11}{4}(1+y) -\frac{3}{4} \\
		&- \frac{27}{4}(1+y)^2\log(1+y) + (1+y)^2  \int_y^1 \frac{R(z)}{(1+z)^3} dz.
	\end{align*} 
	Note that $(1+y)^2  \int_y^1 \frac{R(z)}{(1+z)^3} dz$ and $\frac{1}{(1-y)^2(2+y)^2} $ are smooth near $-1$, thus $\partial_y v_1 \notin H^3(-1,-\frac{1}{2})$ due to the term $- \frac{27}{4}(1+y)^2\log(1+y)$ above.

\end{proof}

%%%%%%%%%%%%%%%%%%%
\bibliographystyle{abbrv}
\bibliography{ref-nlw-p1}

\end{document}